\documentclass[onefignum,onetabnum]{siamart171218}

\usepackage{lipsum}
\usepackage{amsfonts}
\usepackage{graphicx}
\usepackage{epstopdf}
\usepackage{algorithmic}
\usepackage{dsfont}
\usepackage{mathtools}
\usepackage{amssymb}
\usepackage{array}
\usepackage[caption=false]{subfig}
\usepackage{color}

\ifpdf
\DeclareGraphicsExtensions{.eps,.pdf,.png,.jpg}
\else
\DeclareGraphicsExtensions{.eps}
\fi

\newsiamremark{remark}{Remark}
\newsiamremark{hypothesis}{Hypothesis}
\crefname{hypothesis}{Hypothesis}{Hypotheses}
\newsiamthm{claim}{Claim}

\newcommand{\lr}[2]{\langle #1, #2 \rangle}
\definecolor{mypink1}{rgb}{0.858, 0.188, 0.478}
\newcommand{\todo}[1]{#1}
\def\Re{\mathbb{R}}
\def\clZ{{\cal Z}}
\def\E{{\sf E}}
\def\Expect{{\sf E}}
\def\P{\text{Prob}}
\newcommand{\trace}{\text{Tr}}
\newcommand{\ud}{\,\mathrm{d}}

\newcommand{\phiepsN}{{\phi}^{(N)}_\epsilon}
\newcommand{\phieps}{{\phi}_{\epsilon}}

\newcommand{\keps}{{k}_{\epsilon}}
\newcommand{\kepsN}{{k}_{\epsilon}^{(N)}}

\newcommand{\K}{{\sf K}}
\newcommand{\Keps}{{\sf K}_\epsilon}
\newcommand{\KepsN}{{\sf K}_\epsilon^{(N)}}

\newcommand{\Teps}{{T}_{\epsilon}}
\newcommand{\TepsN}{{{T}^{(N)}_\epsilon}}

\newcommand{\geps}{{g}_{\epsilon}}
\newcommand{\pr}{\rho}
\newcommand{\neps}{{n}_{\epsilon}}

\newcommand{\Geps}{{G}_\epsilon}

\newcommand{\Ten}{{\sf T}}
\newcommand{\preps}{\pr_\epsilon}

\newcommand{\nepsN}{{n}_\epsilon^{(N)}}

\newcommand{\hvec}{{\sf h}}
\newcommand{\phivec}{{\sf \Phi}}

\newcommand{\qeps}{q_\epsilon}

\newcommand{\CA}{C_1}
\newcommand{\CB}{C_2}
\newcommand{\Ueps}{U_\epsilon}
\newcommand{\nTeps}{T_\epsilon}

\newcommand{\newP}[1]{\medskip  \noindent{\bf #1}}

\newcommand{\hah}{\hat{h}}
\newcommand{\Gdelta}{G^{(\delta)}}

\headers{Gain function approximation in the FPF}{A.Taghvaei, P. G. Mehta, and S. P. Meyn}
\title{ \bf
	Diffusion map-based algorithm for Gain function approximation in the Feedback Particle Filter\thanks{{Financial support from the NSF CMMI grants 1334987 and 1462773 is gratefully acknowledged.}}}

\author{Amirhossein Taghvaei\thanks{Department of Mechanical Science and Engineering, University of Illinois at Urbana-Champaign, Urbana, IL 
		(\email{taghvae2@illinois.edu}, \email{mehtapg@illinois.edu} %\url{http://www.imag.com/\string~ddoe/}
		).}
	\and Prashant G. Mehta\footnotemark[2]
	\and  Sean P. Meyn\thanks{Department of Electrical and Computer Engineering, University of Florida, Gainesville, FL (\email{spmeyn@gmail.com}) }}

\usepackage{amsopn}

\ifpdf
\hypersetup{
	pdftitle={Gain function approximation in the FPF},
	pdfauthor={A.Taghvaei, P. G. Mehta, and S. P. Meyn}
}
\fi

\begin{document}
	\maketitle
\begin{abstract}
\todo{
Feedback particle filter (FPF) is a numerical algorithm to approximate
the solution of the nonlinear filtering problem in continuous-time
settings.  In any numerical implementation of the FPF algorithm, the
main challenge is to numerically approximate the so-called gain
function. A numerical algorithm for gain function approximation is the
subject of this paper.  The exact gain function is the solution of a
Poisson equation involving a probability-weighted Laplacian
$\Delta_\rho$. The numerical problem is to approximate this solution
using {\em only} finitely many particles sampled from the probability
distribution $\rho$.  A diffusion map-based algorithm was proposed by
the authors in a prior work~\cite{Amir_CDC2016,Amir_ACC17} to solve
this problem.  The algorithm is named as such because it involves, as
an intermediate step, a diffusion map approximation of the exact
semigroup $e^{\Delta_\rho}$. The original contribution of this paper
is to carry out a rigorous error analysis of the diffusion map-based
algorithm.  The error is shown to include two components: bias and
variance.  The bias results from the diffusion map approximation of
the exact semigroup.  The variance arises because of finite sample
size.  Scalings and upper bounds are derived for bias and variance. 
These bounds are then illustrated with numerical experiments that serve to
emphasize the effects of problem dimension and sample size.  The
proposed algorithm is applied to two filtering examples and
comparisons provided with the sequential importance resampling (SIR)
particle filter.    
}
\end{abstract}


% REQUIRED

\section{Introduction}
\label{sec:intro}

This paper is concerned with a numerical solution of a
certain linear partial differential equation (PDE) that arises in nonlinear
filtering problem in continuous-time settings.  

\newP{Nonlinear filtering problem:}
The standard model of the nonlinear filtering problem is given by the following stochastic differential equations (SDE)~\cite{xiong2008}:
\begin{subequations}
	\begin{align}
	\text{State process:}\quad \ud X_t &= a(X_t)\ud t + \ud B_t,\quad X_0 \sim p_0
	\label{eq:Signal_Process}
	\\
	\text{Observation process:}\quad\ud Z_t &= h(X_t)\ud t + \ud W_t,
	\label{eq:Obs_Process}
	\end{align}
\end{subequations}
where $X_t\in\Re^d$ is the (hidden) state at time $t$, $Z_t \in\Re$ is the
observation, and $B_t$, $W_t$ are two mutually independent
standard Wiener processes taking values in $\Re^d$ and $\Re$,
respectively. The mappings $a(\cdot): \Re^d \rightarrow \Re^d$ and
$h(\cdot): \Re^d \rightarrow \Re$ are known $C^1$ functions, and $p_0$
is the density of the prior probability distribution. 

The objective of the filtering problem is to compute the  posterior
distribution of the state $X_t$ given the time history of observations
(filtration) $\clZ_t := \sigma(Z_s:  0\le s \le t)$. 

The problem is {\em linear Gaussian} if $a(\cdot)$, and $h(\cdot)$ are
linear functions and $p_0$ is a Gaussian density. We use $A$ and $H$
to denote the matrices that define these linear functions, i.e,
$a(x)=Ax$ and $h(x)=Hx$. The background on the linear Gaussian
problem, along with its solution given by the Kalman-Bucy
filter~\cite{kalman-bucy}, appears in~\cite{kwakernaak1972linear}.

\newP{Feedback particle filter (FPF)}~is a numerical algorithm to
approximate the posterior distribution in nonlinear non-Gaussian settings~\cite{taoyang_TAC12,yang2016}. \todo{
	The FPF algorithm is an alternative to the sequential
        importance resampling (SIR) particle
	filters~\cite{gordon93,doucet09,bain2009,delmoralbook}.  The
        distinguishing feature of the FPF is that the 
	importance sampling step is replaced with feedback control.  Steps
	such as resampling, reproduction, death or birth of particles are
	altogether avoided.  The particles in FPF have uniform importance
	weights by construction.  Therefore, the FPF does not suffer
        from the particle
	degeneracy issue that is commonly observed in implementations
        of the SIR 
	particle filters~\cite{doucet09}.  
	In independent numerical evaluations and comparisons, it has been
	observed that FPF exhibits smaller simulation variance and better
	scaling properties with the problem dimension~\cite{berntorp2015,stano2014,surace_SIAM_Review}.}

The construction of FPF is based on the following two steps:
\begin{romannum}
	\item[Step 1:] Construct a stochastic process, denoted by $\bar{X}_t \in \Re^d$, whose conditional distribution (given $\clZ_t$) is equal to the conditional distribution of $X_t$; 
	\item[Step 2:] Simulate $N$ stochastic processes, denoted by $\{X^i_t\}_{i=1}^N$, to empirically approximate the distribution of $\bar{X}_t$.
\end{romannum}
\begin{equation*}
\underbrace{\Expect[f(X_t)|\clZ_t]\overset{\text{Step
			1}}{=}\Expect[f(\bar{X}_t)|\clZ_t]}_{\text{exactness condition}}
\overset{\text{Step 2}}{\approx} \frac{1}{N}\sum_{i=1}^N f(X^i_t)\label{eq:exactness}.
\end{equation*}

The process $\bar{X}_t$ is referred to as mean-field process and the
$N$ processes $\{X^i_t\}_{i=1}^N$ are referred to as particles.  The
construction ensures that the filter is {\em exact} in the mean-field
($N=\infty$) limit.    

% The condition (i) is is referred to as  exactness, and (ii) is the empirical approximation.
%1The approximation (ii) captures the accuracy of the filter.

The details of the two steps are as follows:

\newP{Mean-field process:} In the FPF, the mean-field process $\bar{X}_t$ evolves according to the SDE given by
\begin{equation}
\ud \bar{X}_t = \underbrace{a(\bar{X}_t) \ud t + \ud \bar{B}_t}_{\text{propagation}} + \underbrace{\K_t(\bar{X}_t) \circ (\ud Z_t -
	\frac{h(\bar{X}_t) + \hat{h}_t}{2}\ud t)}_{\text{feedback control law}},
\quad \bar{X}_0\sim p_0,
%\quad\text{for}\quad i=1,\ldots N
\label{eq:FPF-mean-field}
\end{equation}  
where $\bar{B}_t$ is a standard Wiener processes independent of $\bar{X}_0$ and
$\hat{h}_t := \E[h(\bar{X}_t)|\mathcal{Z}_t]$. 
The $\circ$ indicates that the sde is expressed in its Stratonovich form. 
The gain function is $\K_t(x):= \nabla \phi_t(x)$ where $\phi_t$ is the solution of the Poisson
equation:
\begin{equation}
\label{eq:Poisson-intro}
\text{Poisson equation:}\quad \frac{1}{p_t(x)}\nabla \cdot (p_t(x)
\nabla \phi_t(x) ) = -(h(x)-\hat{h}_t),\quad\forall \;x\in\Re^d,
\end{equation}
where $\nabla$ and $\nabla \cdot $ denote the
gradient and the divergence operators, respectively, and
$p_t$ denotes the conditional density of $\bar{X}_t$ given
$\mathcal{Z}_t$.  The operator on the left-hand side of the Poisson
equation~\eqref{eq:Poisson-intro} is referred to as the {\em
  probability-weighted Laplacian}.  It is denoted as $\Delta_{\rho}$ where
the probability density $\rho$ is the conditional density $p_t$.  

\newP{Particles:} The particles $\{X^i_t\}_{i=1}^N$ evolve according to:
\begin{equation}
\ud X^i_t = a(X^i_t) \ud t + \ud B^i_t +{\K^{(N)}_t(X^i_t) \circ (\ud Z_t -
	\frac{h(X^i_t) + \hat{h}^{(N)}_t}{2}\ud t)},\quad X^i_0\overset{\text{i.i.d}}{\sim} p_0,
\label{eq:FPF-finite-N}
\end{equation}  
for $i=1,\ldots N$, where $\{B^i_t\}_{i=1}^N$ are mutually independent Wiener processes, 
$\hat{h}^{(N)}_t:=\frac{1}{N}\sum_{i=1}^N h(X^i_t)$, and $\K^{(N)}_t$
is the output of an algorithm that approximates the solution to the Poisson equation~\cref{eq:Poisson-intro} 
\begin{equation}
\text{Gain function approximation:}\quad\K^{(N)}_t := \text{Algorithm}(\{X^i_t\}_{i=1}^N;h).
\label{eq:gain-func-approx} 
\end{equation}
The notation is suggestive of the fact that algorithm is adapted
to the ensemble $\{X^i_t\}_{i=1}^N$ and the function $h$; the density
$p_t(x)$ is not known in an explicit manner.  

Development and error analysis of one such gain function approximation
algorithm is the subject of the present paper.  Before describing the general case, it is
useful to review the filter for the linear Gaussian case where the
solution of the Poisson equation is explicitly known.  

\newP{FPF for Linear Gaussian setting:} Suppose $h(x)=H x$ and $p_t$
is a Gaussian density with mean $\bar{m}_t$ and variance $\bar{\Sigma}_t$.  Then the solution of
the Poisson equation is known in an explicit form~\cite[Sec. D]{yang2016}.  The resulting gain
function is constant and equal to the Kalman gain:
\begin{equation}\label{eq:Kalman-gain}
\K_t(x) \equiv \bar{\Sigma}_tH^\top ,\quad \forall \; x\in \Re^d.
\end{equation}
% where $\bar{\Sigma}_t$ is the conditional covariance matrix of
% $\bar{X}_t$.
Therefore, the mean-field process~\cref{eq:FPF-mean-field} for the
linear Gaussian problem is given by:
\begin{equation*}\label{eq:FPF-mean-field-linear}
\ud \bar{X}_t = A \bar{X}_t \ud t + \ud \bar{B}_t + \bar{\Sigma}_tH^\top  (\ud Z_t - \frac{H\bar{X}_t+H\bar{m}_t}{2}\ud t),\quad \bar{X}_0 \sim p_0.
\end{equation*}  
%where $\bar{m}_t:=\Expect[\bar{X}_t|\clZ_t]$.

Given the explicit form of the gain function~\cref{eq:Kalman-gain}, the empirical approximation of the gain is simply $\K_t^{(N)}=\Sigma_t^{(N)}H^\top$ where $\Sigma_t^{(N)}$ is the empirical covariance of the particles. Therefore, the evolution of the particles is:
\begin{equation}
\ud X^i_t = AX^i_t \ud t + \ud B^i_t +{\K^{(N)}_t  (\ud Z_t -
	\frac{HX^i_t + H m^{(N)}_t}{2}\ud t)},\quad X^i_0\overset{\text{i.i.d}}{\sim} p_0,\quad\label{eq:FPF-linear}
\end{equation}
for $i=1.\ldots,N$, where $m_t^{(N)}$ is the empirical mean of the particles. The empirical quantities are computed as:
\begin{align*}
m^{(N)}_t&:=\frac{1}{N}\sum_{j=1}^N X^i_t,\quad\Sigma^{(N)}_t
:=\frac{1}{N-1}\sum_{j=1}^N (X^i_t-m^{(N)}_t)(X^i_t-m^{(N)}_t)^\top.
%\label{eq:empr_app_mean_var}
\end{align*}
The linear Gaussian FPF~\cref{eq:FPF-linear} is identical to the
square-root form of the ensemble
Kalman filter (EnKF)~\cite[Eq. 3.3]{Reich-ensemble}.  

\medskip

One extension of the Kalman gain is the so called {\em constant gain
	approximation} formula whereby the gain $\K_t$ is approximated by
its expected value (which represents the best 
least-squared approximation of the gain by a constant). Remarkably, the
expected value admits a closed-form expression 
which is then readily approximated empirically using the particles (see~\cref{rem:constant-gain} for derivation):
\begin{equation}
\begin{aligned}
\text{Const. gain approx:}\quad 
\Expect [\K_t (X_t)|\clZ_t] &= \int_{\Re^d} (h(x)-\hat{h}_t)\;
x \; p_t(x) \ud x \\&\approx \frac{1}{N}\sum_{i=1}^N\; (h(X^i_t)-\hat{h}^{(N)}_t) \; X^i_t.
\end{aligned}\label{eq:const-gain-approx}
\end{equation}
The constant gain approximation formula has been used in nonlinear
extensions of the EnKF algorithm~\cite{jana2016stability}.  The connection to the
Poisson equation provides a justification for this formula.  The
formula is attractive because it provides a consistent (as the number of particles
$N\rightarrow\infty$) approximation
of the Kalman gain in the linear Gaussian setting.

\medskip

Design and analysis of the gain function approximation
algorithm~\eqref{eq:gain-func-approx} in the general case is a
challenging problem because of two reasons:  (i) Apart from the
Gaussian case, there are no known closed-form solutions
of~\cref{eq:Poisson}; (ii) The density $p_t(x)$ is not explicitly
known.  At each time-step, one only has samples $\{X^i_t\}_{i=1}^N$. For the
purpose of this paper, these samples are assumed to be i.i.d drawn
from $p_t$.  The assumption is justified because in the limit of large
$N$, the particles are approximately i.i.d (by the propagation of chaos);
cf.,~\cite{sznitman1991}.

\subsection{Contributions of this paper}
\todo{
	The paper presents a diffusion map-based algorithm for the gain
	function approximation problem.  The algorithm is named as such because it
                involves, as an intermediate step, a diffusion map
                approximation of the exact semigroup $e^{\Delta_\rho}$.  
	The following is a summary of specific original contributions made in
	this paper:
	\begin{enumerate}
		\item[(i)]  Error estimates that relate the exact
                  semigroup to its diffusion map approximation.  The
                  error estimates are derived by employing a Feynman-Kac
                  representation of the semigroup (\cref{prop:Tepsn-convergence}); 
		\item[(ii)] A uniform spectral gap for the diffusion
                  map based on the use of the 
		Foster-Lyapunov function method from the theory of
                stochastic stability of
		Markov processes (\cref{prop:DV3}); and 
		\item[(iii)] Error estimates for the empirical approximation of the diffusion
		map (\cref{prop:TepsN-convergence}). 
	\end{enumerate}
	The results from (i) and (ii) are used to derive estimates for
        the bias and to show that the bias converges to zero in a
        certain limit (\cref{thm:bias}).  Results from (iii) are
	used to prove the convergence of the variance error term to
	zero in the infinite-$N$ limit (\cref{thm:variance}).  
The paper contains numerical experiments
		that serve to illustrate the effects of problem dimension and sample
		size.  The algorithm is applied to two filtering examples and
		comparisons provided with the sequential importance
                resampling (SIR) particle filter.

	% The objective of this paper is to analyze the the diffusion map-based algorithm for the gain function approximation problem.  The algorithm was first proposed in our earlier paper~\cite{Amir_CDC2016}. The algorithm is based on a reformulation of the Poisson equation as a fixed-point equation that involves the diffusion semigroup corresponding to the weighted Laplacian. The fixed-point problem is approximated with a finite-dimensional problem in two steps: In the first step, the semigroup is approximated with a Markov operator referred to as diffusion map. In the second step, the diffusion map is approximated empirically, using particles, as a Markov matrix. 

	% The preliminary error analysis of the diffusion map-based algorithm appeared in our earlier paper~\cite{Amir_ACC17}. However, the proofs were incomplete, based on formal arguments. This paper contains rigorous error analysis of the diffusion map-based algorithm. In particular, the following new results are shown:

	\subsection{Relationship to prior work}  The gain function algorithm
	first appeared in the conference version of this
	paper~\cite{Amir_CDC2016}.  Its preliminary error analysis was
	reported in the conference paper~\cite{Amir_ACC17}.  The important
	distinction is that the results in these conference papers
        were preliminary in nature.  The proofs were either altogether
        omitted or based on formal arguments.  The main techniques employed in this
	paper, namely, (i) the use of Feyman-Kac representation to quantify the
	error due to the diffusion map approximation of the exact semigroup,
	and (ii) the use of stochastic stability theory to derive
        uniform spectral gap for the diffusion map, are
	original and do not appear in the conference papers.
	These techniques are important to be able to obtain precise estimates
	as enumerated above in the list of contributions.  Since the main
	technical tools are new, {\em all} the proofs, based on these techniques,
	are new and original contributions of this paper.  The diffusion map
	was introduced in~\cite{coifman}, in the context of spectral
	clustering~\cite{belkin,luxburg2007}.  Results on its convergence analysis
	appears
	in~\cite{hein-consistency-2005,singer2006graph,coifman,gine2006empirical,hein2006,von2008consistency,belkin2007convergence}.
	The use of diffusion map approximations for
	filtering problems is originally due to the authors.

}
%
%We present a
%new basis-free  diffusion-map based algorithm for approximating the solution
%of the gain function.  The key step is to construct a Markov
%matrix on a graph defined on the space of particles $\{X^i_t\}_{i=1}^N$.
%The value of the function $\phi$ for the particles, $\phi(X^i_t)$, is then
%approximated by solving a fixed-point problem involving
%the Markov matrix.  The
%fixed-point problem is shown to be a contraction and the method of
%successive approximation applies to numerically obtain the solution.  
%A procedure
%for carrying out error analysis of the approximation is introduced in
%this paper.
%Certain asymptotic estimates for bias and variance are derived.  Comparison with the
%constant gain approximation formula are provided.  These results are
%illustrated with the aid of some numerical experiments.

\subsection{Literature survey}

Apart from its direct relevance to numerical approximation of the FPF,
there are three topics of current research interest that are relevant
to the subject of this paper: (i) ensemble Kalman filter; (ii) particle flow algorithms for nonlinear
filtering; and (iii) optimal
transport.  Specifically, the algorithms for gain function
approximation described
in this paper are also directly applicable to these other topics.  These relationships are briefly discussed next:

\newP{Ensemble Kalman filter:} The EnKF algorithm was first developed
in the discrete-time setting~\cite{evensen1994sequential}.  In
the continuous-time setting, two formulations of the EnKF have been
developed: stochastic EnKF, and the more recent deterministic
EnKF~\cite{Reich-ensemble,jdw:ReichCotter2015}. As has already been
noted, the
deterministic EnKF is in fact identical to the FPF algorithm~\cref{eq:FPF-linear}
in the linear Gaussian setting~\cite{Reich-ensemble,TaghvaeiASME2017}.

The EnKF algorithm provides a consistent approximation in the linear
Gaussian setting. Compared to the Kalman filter, the main utility of
EnKF is that it does not require propagation of the covariance matrix. This
reduces the computational complexity from $O(d^2)$ for the Kalman filter
to $O(Nd)$. This is clearly advantageous in high dimensional problems
when $N<<d$.  This
property has made EnKF popular in applications such as weather
prediction in high dimensional
settings~\cite{sr:kalnay,sr:Oliver2008}. The disadvantage of the EnKF
algorithm, of course, is that it does not provide a consistent approximation for nonlinear
problems. 

FPF represents a the generalization of the EnKF to the nonlinear
non-Gaussian setting~\cite{TaghvaeiASME2017}: With the constant gain approximation, the
algorithms are identical.  Given this parallel, { the problem of
	improving the EnKF algorithm in more general nonlinear non-Gaussian
	settings is directly related to the problem of better approximating
	the gain function in the FPF}.  In an application software based on
EnKF, it is a relatively simple matter to replace the constant gain
formula for the gain by more sophisticated approximations described in
this paper.  Certain empirical evaluations on the performance of FPF
in high-dimensional settings are reported
in~\cite{surace_SIAM_Review,stano2014,stano2013nonlinear,berntorp2015}.

% Based on the widespread use of EnKF for high-dimensional problems and
% certain empirical observations on the performance of FPF in high-dimensional
% settings~\cite{surace_SIAM_Review,stano2014,stano2013nonlinear,berntorp2015}, 
% better gain function approximation may be useful for particle
% filtering in high dimensions.   

Error analysis and stability of EnKF is an active
area of research;
see~\cite{gland2009,mandel2015,delmoral2016stability} for linear
models
and~\cite{jana2016stability,delmoral2017stability,stuart2014stability}
for nonlinear models.  The error analysis for the gain function
approximation reported in this paper is a step towards error analysis
of the FPF along these lines.     

\newP{Particle flow algorithms:} The following first-order (and hence
an under determined) form of the
Poisson equation appears in most types of particle flow algorithms:
\[
\nabla \cdot (p_t(x) \K(x)) = \text{(rhs)},
\]
where the righthand-side (rhs) is given and $\K(x)$ defines a vector
field that must be obtained to implement the particle flow.  The PDE
appears in the first interacting particle representation of the
continuous-time filtering in~\cite{crisan2007,crisan10} and the
discrete-time filtering in~\cite{daum10}.  % In these formulations, the
% gain $\K$ is obtained directly by solving a first-order (and hence
% an under-determined) pde.  
Stochastic extensions of these have also
recently appeared in~\cite{daum2017generalized} where approximate solutions
are also described based on Gaussian assumption on the density.  
The algorithm described here represent an approximation of a
particular gradient form solution of the first-order PDE.

\newP{Optimal transport:} The mean-field SDE~\cref{eq:FPF-mean-field}
represents a transport that maps the prior distribution at time
$0$ to the posterior distribution at an (arbitrary) future time $t>0$.
Synthesis of optimal transport maps for implementing the Bayes formula
appears
in~\cite{reich11,reich13,MarzoukBayesian,AmirACC2016,heng2015gibbs,chen2016Linear}.  
The relationship with the Poisson equation is through the ensemble
transform filter which relies on a linear programming construction to
approximate the optimal transport map~\cite{reich13}.  As discussed
in~\cite[Sec.~5.5]{TaghvaeiASME2017}, the solution of the Poisson equation
yields an infinitesimal optimal transport map from the ``prior''
$p_t(x)$ to ``posterior'' $\frac{1}{\gamma}p_t(x)e^{-th(x)}$.  
%The resulting class of algorithms are described in~\cite{}.   
Another closely related approach 
is transportation through Gibbs flow~\cite{heng2015gibbs}.

% \newP{Stein variational gradient descent:} The Stein variational gradient descent~\cite{liu2016stein} is an algorithm that transports a set of particles to a desired target probability distribution through a sequence of incremental transport maps. The transport maps are designed to be the gradient descent with respect to the KL-divergence. In particular let $q_t$ to be the current density of particles and $p$ to be the target density. Then the particles evolve according to
% \begin{equation*}
% \dot{X}^i_t = u_t (X^i_t)
% \end{equation*}  
% It is shown in~\cite{liu2016stein} that 
% \begin{equation*}
% \frac{\ud}{\ud t} KL(q_t|p)= \Expect_{X\sim q}[\tr(\mathcal{A}_pu(X)) ]
% \end{equation*}
% where $A_p u(x):= \nabla p(x) u(x)^\top + \nabla u(x)$ is the Stein operator. Then $u_t$ is chosen to be 
% \begin{equation*}
% u_t = \underset{u,\|u\|\leq S(q,p)}{\text{argmax}} \Expect_{X\sim q}[\tr(\mathcal{A}_pu(X)) ]
% \end{equation*}
% where $S(q,p):=\max_{\|u\|\leq 1} |\Expect_{X\sim q}[\tr(\mathcal{A}_p u(X))] |^2$. Then the paper~\cite{liu2016stein} proposes to search for the optimal $u$ over the Reproducible Kernel Hilbert space (RKHS).  The optimal $u$ is 
% \begin{equation*}
% u(x) = \frac{1}{N}\sum_{i=1}^N k(x,X^i)\nabla \log p(X^i) + \nabla_y k(x,X^i)
% \end{equation*}
% Note that for $u=\nabla \phi$ we have $\tr(\mathcal{A}_p u(x))= \Delta_p \phi(x)$. {\color{red} could not make connection} 

\medskip

Directly related to the FPF, the Galerkin method for the numerical
solution of the Poisson equation appeared in original
papers~\cite{yang2016,taoyang_TAC12}.  The Galerkin algorithm
represents the `direct'' PDE approach to construct a
numerical approximation.  The constant gain approximation is a
particular example of a Galerkin solution.  In general, the main problem with the
Galerkin approximation is that it requires a selection of basis
functions.  This becomes intractable in high dimensions.  To mitigate
this issue, a proper
orthogonal decomposition (POD)-based procedure to select basis
functions is introduced in~\cite{berntorp2016}. Other existing approaches  are a
continuation scheme for approximation~\cite{matsuura2016suboptimal}, a probabilistic approach
based on dynamic programming~\cite{Sean_CDC2016}, and a procedure based on expressing the gain function in a reproducible Hilbert kernel space~\cite{radhakrishnan2018feedback}. A comparison of different gain function approximation methods  appears in~\cite{berntorp2018comparison}. 

\medskip
\subsection{Paper outline}
The outline of the remainder of this paper is as follows: The
mathematical problem of the gain function approximation together with
a summary of known results on this topic appears in~\cref{sec:prelim}.  
The diffusion-map based algorithm is described in a self-contained fashion in
\cref{sec:kernel}.  
The main theoretical results of this paper including the bias and variance
estimates appear in \cref{sec:error_kernel}.  Some numerical experiments for
the same appear in \cref{sec:numerics}.  % The paper concludes in
% \cref{sec:conc} with a discussion of future directions.  
All the proofs
appear in the Appendix.    

\subsection{Notation} 
For vectors $x,y\in\Re^d$, the dot product is denoted as $x\cdot y$ and
$|x|:=\sqrt{x\cdot x}$. The space of positive definite $d \times d$ matrices is denoted as $S^d_{++}$. The Borel $\sigma$-algebra on $\Re^d$ is denoted by $\mathcal{B}(\Re^d)$. The indicator function, for a measurable set $A \in \mathbb{B}(\Re^d)$, is denoted as $\mathds{1}_A(\cdot)$. 
The space of measurable functions $f:\Re^d \to \Re$ such that $\|f\|_{L^p(\pr)}:=\left(\int |f(x)|^p \pr(x)\ud x\right)^{1/p} <\infty$ is denoted as $L^p(\pr)$.  The inner product on  $L^2(\pr)$  is defined by $\big
<f,g\big>:=\int f(x)g(x) \pr(x)\ud x$. 
The space  
$H^1(\pr)$ is the
space functions $f\in L^2(\pr)$ whose derivative (defined in the
weak sense) is in $L^2(\pr)$.  For a (weakly) differentiable function $f$, $\|\nabla f\|_{L^p(\pr)}:=\left(\int |\nabla f(x)|^p\pr(x)\ud x\right)^{1/p}$. For an integrable  function $f$, $\hat{f}_\pr:=\int f(x) \rho(x) \ud x$ denotes the
mean.  $L^2_0(\pr) :=\{f \in L^2(\pr) \mid\hat{f}_\pr =0\}$
and $H_0^1(\pr):=\{f \in H^1(\pr) \mid \hat{f}_\pr =0\}$ denote the co-dimension $1$ subspace of functions whose
mean is zero.  $L^\infty(\Omega)$ denotes the
space of bounded functions on $\Omega \subset \Re^d$ with the sup-norm denoted as
$\|\cdot\|_{L^{\infty}(\Omega)}$. The space of continuous and bounded functions on $\Omega \subset \Re^d$ and the space of continuous and smooth functions on $\Omega$ is denoted as $C_b(\Omega)$ and $C^\infty_b(\Omega)$ respectively. For a linear operator $T$, on a Banach space $\mathcal{X}$ with norm $\|\cdot\|_\mathcal{X}$, the operator norm is denoted as $\|T\|_\mathcal{X}$.  The Gaussian distribution with mean $m$ and covariance $\Sigma$ is denoted as $\mathcal{N}(m,\Sigma)$. The variance of the random variable $X$ is denoted as $\text{Var}(X)$. 
\section{Gain function approximation}
\label{sec:prelim}
\subsection{Problem formulation}\label{sec:problem-formulation}
The mathematical problem is to numerically approximate the solution
of the Poisson's equation~\cref{eq:Poisson-intro} introduced in
\cref{sec:intro} and also repeated below:
\begin{equation}
- \Delta_\rho \phi = h -\hat{h}_\pr,
\label{eq:Poisson}
\end{equation}  
where the weighted Laplacian $\Delta_\pr \phi(x):=\frac{1}{\pr(x)}\nabla \cdot(\pr(x)\nabla \phi(x))$; $\rho(x)$ is
an everywhere positive probability density on $\Re^d$; $h(x)$ is a real-valued
function defined on $\Re^d$ and $\hat{h}_\pr:=\int h(x) \rho(x) \ud x$.
The function $\phi$ is referred to as the solution.  Its gradient is referred to as the gain
  function and denoted as $\K(x) := \nabla\phi(x)$.  The
  PDE~\cref{eq:Poisson} is referred to as the Poisson's equation.

The numerical approximation problem is as follows:

\newP{Problem statement:} Given $N$ 
samples $\{X^1,\hdots,X^i,\hdots,X^N\}$, drawn i.i.d. from $\rho$, approximate
the gains $\{\K^1,\hdots,\K^i,\hdots,\K^N\}$, where
$\K^i:=\K(X^i)=\nabla\phi(X^i)$.   
The density $\rho$ is not known in an explicit form.

\subsection{Mathematical preliminaries}
\newP{Assumptions:} The following assumptions are made throughout the paper:

\begin{romannum}
	\item {\bf Assumption A1:} The probability density $\rho$ is of the
	form $\rho(x)=e^{-V(x)}$ where the function $V(x)=\frac{1}{2}(x-m)^\top\Sigma^{-1}(x-m) + w(x)$ for some $m \in \Re^d$, $\Sigma \in S^d_{++}$, and  $w \in C^\infty_b(\Re^d)$; 
	\item {\bf Assumption A2:} The function $h:\Re^d \to \Re$ is (weakly) differentiable with $\|h\|_{L^4(\pr)},\|\nabla h\|_{L^4(\pr)}<\infty$. 
\end{romannum}

\todo{
\begin{remark}\label{rem:assumption}
	 Assumption A1 is used to prove the approximation
         result (\cref{prop:Tepsn-convergence}) and to derive the spectral
         gap (\cref{prop:DV3}) for the diffusion map approximation
         first introduced in 
         \cref{sec:kernel}. In prior literature, a similar assumption has been previously
         used for studying functional inequalities to obtain
         Poincar\'e inequality with a constant that does not depend on
         the dimension~\cite[Ch. 8]{villani2003}. Assumption A1 is
         restrictive, e.g., a mixture of Gaussians does not satisfy
         the assumption.  Based on numerical experiments, it is
         conjectured that Assumption A1 can be relaxed.  A weaker
         assumption would be to assume $\pr = \pr_g * w $, the
         convolution of a Gaussian density $\pr_g$  with a density $w$
         that has a compact support. Proving the theoretical results
         under this weaker assumption is the subject of future work.  
\end{remark}
}

\subsubsection{Spectral representation} Under Assumption (A1), the weighted Laplacian $\Delta_\pr$ has a
discrete spectrum with an ordered sequence of eigenvalues
$0=\lambda_0<\lambda_1\le\lambda_2\le\hdots$ and associated
eigenfunctions $\{e_n\}$ that form a complete orthonormal basis of
$L^2(\pr)$~\cite[Cor. 4.10.9]{bakry2013}.
The trivial eigenfunction $e_0(x) = 1$, and for $f \in L^2_0(\pr)$, the spectral
representation yields:
\begin{equation}
\label{eq:spectral_rep}
-\Delta_\rho f= \sum_{m=1}^\infty \lambda_m \langle e_m,f \rangle e_m.
\end{equation}

The positivity of the smallest non-trivial eigenvalue ($\lambda_1>0$) is
referred to as the Poincar\'e inequality (or the spectral gap
condition)~\cite{bakry08}.  The inequality is equivalently expressed as
\begin{equation*}
\int_{\Re^d} (f-\hat{f}_\pr)^2 \rho \ud x \leq
\frac{1}{\lambda_1}\int_{\Re^d} |\nabla f|^2 \rho \ud x, \quad \quad \forall\; f
\in H^1(\rho),
%\label{eq:poincare} 
\end{equation*} 
where $\hat{f}_\pr=\int f  \pr \ud x$. 

The Poincar\'e inequality is important to show that the Poisson
equation is well-posed and a unique solution exists. The solution to
the Poisson equation is defined using the weak formulation.  

\subsubsection{Weak formulation} A function $\phi \in H_0^1(\rho)$ is
said to be a weak solution of~\cref{eq:Poisson} if
\begin{equation}
\int \nabla \phi (x) \cdot \nabla \psi(x) \rho(x) \ud x = \int  (h(x)-\hat{h}_\pr) \psi(x) \rho(x) \ud x \quad \quad \forall\; \psi
\in H^1(\rho).\label{eq:Poisson-weak}
\end{equation}
Equation~\cref{eq:Poisson-weak} is referred to as the weak-form of
the Poisson's equation. 
The weak-form is expressed
succinctly as $\langle\nabla\phi,\nabla\psi\rangle=\langle
h-\hat{h}_\pr,\psi\rangle$ where $\langle \cdot,\cdot \rangle$ is 
the inner-product in $L^2(\rho)$.
The existence and uniqueness of the solution to the weak-form of the Poisson equation is stated in the following Proposition.
\begin{proposition}
	{\cite[Thm. 2.2.]{laugesen15}}
	\label{prop:existence-weak-form}
Suppose $\rho$ satisfies Assumption~(A1) and $h$ satisfies
Assumption~(A2).  Then there exists a unique function $\phi \in H_0^1(\pr)$ that satisfies the weak-form of the Poisson equation~\cref{eq:Poisson-weak}. The solution satisfies the bound:
\begin{equation*}
\int |\nabla \phi(x)|^2 \pr(x)\ud x \leq \frac{1}{\lambda_1}\int (h(x)-\hat{h}_\pr)^2\pr(x)\ud x.
\end{equation*}	
\end{proposition}   
%The weak-form is expressed
%succinctly as $\langle\nabla\phi,\nabla\psi\rangle=\langle
%h-\hat{h}_\pr,\psi\rangle$ where $\langle \cdot,\cdot \rangle$ is 
%the inner-product in $L^2(\rho)$.  

%The weak formulation has led to the Galerkin algorithm
%presented in the original FPF papers~\cite{yang2016}. A
%special case of the Galerkin solution is the constant gain
%approximation formula~\cref{eq:const-gain-approx}.  The formula is
%obtained as the best least-square approximation (in $L^2(\pr)$) of exact gain function on the  subspace $S:=
%\text{span}\{x_1,x_2,\ldots,x_d\}\subset H^1$, where
%$x_1,x_2,\ldots$ are the coordinate
%functions. 
\todo{
\begin{remark}
[Constant gain approximation]
The weak formulation~\cref{eq:Poisson-weak} has led to the Galerkin
algorithm presented in the original FPF papers~\cite{yang2016}. A
special case of the Galerkin solution is the constant gain
approximation formula~\cref{eq:const-gain-approx}.
The formula is obtained upon choosing the test functions
in~\cref{eq:Poisson-weak} to be the coordinate functions: $\psi_m(x) =
x_m$ for $m=1,2,\ldots,d$. Then,
\begin{equation*}
\int \frac{\partial \phi}{\partial x_m}(x) \pr (x) \ud x  = \int (h(x)-\hat{h}_\pr)x_m \pr(x)\ud x,\quad \text{for}\quad m=1,\ldots,d,
\end{equation*}
which yields the formula~\cref{eq:const-gain-approx}. 
%  choose  The formula is
%obtained as the best least-square approximation (in $L^2(\pr)$) of exact gain function on the  subspace $S:=
%\text{span}\{x_1,x_2,\ldots,x_d\}\subset H^1_0$, where
%$x_1,x_2,\ldots$ are the coordinate
%functions. 
\label{rem:constant-gain}
\end{remark}
}
\medskip

The diffusion map-based algorithm presented in this paper is based on the semigroup
formulation of the Poisson equation.
  
\subsubsection{Semigroup}\label{sec:semigroup} Let $\{P_t\}_{t\geq 0}$ be the semigroup associated
with the weighted Laplacian $\Delta_\pr$.  The semigroup allows for a
probabilistic interpretation which is described next.  Consider the
following 
reversible Markov process $\{S_t\}_{t\geq 0}$ evolving in $\Re^d$:
\[
\ud S_t =  -\nabla V(S_t) \ud t  + \sqrt{2}\ud B_t,
\]
where $V(x) := -\log(\pr(x))$ and $\{B_t\}_{t\geq 0}$ is a standard Weiner
process in $\Re^d$.  Then
\[
P_tf(x)= \Expect[f(S_t)|S_0=x] .
\]
It is straightforward to verify that $P_t:L^2(\pr)\to L^2(\pr)$ is
symmetric, i.e., $\lr{P_tf}{g} =\lr{f}{P_t g}$ for all $f,g \in
L^2(\pr)$ and $\pr(x) = e^{-V(x)}$ is its invariant density.  
The semigroup also admits a kernel representation:
\[
P_tf(x) = \sum_{m=1}^\infty e^{-t\lambda_m}\lr{e_m}{f}e_m(x) =
\int_{\Re^d} \bar{k}_t(x,y)f(y)\pr(y)\ud y,
\] 
where $\bar{k}_t(x,y):=\sum_{m=0}^\infty e^{-t\lambda_m} e_m(x) e_m(y)$.  

The spectral gap implies that
$\|P_t\|_{L^2_0(\pr)}=e^{-t\lambda_1}<1$. Hence,
$P_t$ is a strict contraction on $L^2_0(\pr)$.  
For the
special case of Gaussian density, the
eigenfunctions are given by the Hermite polynomials.  This leads to an
explicit formula for the kernel $\bar{k}_t(x,y)$ in the Gaussian case, as
described in~\cref{apdx:Laplacian-Gaussian}. 

%The semigroup also identifies the solution to the (weighted) heat equation, i.e $u(t,x) = P_t f(x)$ solves the pde:
%\begin{equation*}
%\frac{\partial u}{\partial t} = \Delta_\pr u,\quad u(0,x)=f(x)
%\end{equation*} 

%\newP{Semigroup formulation:} 
Consider the heat equation 
\begin{equation*}
\frac{\partial u}{\partial t} = \Delta_\pr u + (h-\hat{h}_\pr),\quad u(0,x)=f(x).
\end{equation*} 
Its solution is given in terms of the semigroup as follows:
\begin{equation*}
u(t,x) = P_t f(x) + \int_0^t P_{t-s}(h-\hat{h}_\pr)(x)\ud s.
\end{equation*}
Letting $f(x)=\phi(x)$ where $\phi$ solves the Poisson equation~\cref{eq:Poisson} yields the following 
fixed-point equation for $t=\epsilon$:
\begin{align}
\text{(exact fixed-point equation)}\quad\phi = P_\epsilon \phi+
\int_0^\epsilon P_s (h-\hat{h}_\pr) \ud s. \label{eq:Poisson-semigroup} 
\end{align}
%where we used $\Delta_\pr \phi + h -\hat{h}_\pr=0$.
%\medskip
Equation~\cref{eq:Poisson-semigroup} is referred to as the semigroup
form of the Poisson equation~\cref{eq:Poisson}.

%\item \newP{Variational formulation:} The optimization problem is given by
%\begin{equation}\label{eq:opt-Poisson}
%\phi= \argmin_{\psi \in H^1_0(\pr)} ~\Expect_{X\sim \pr}[\frac{1}{2}|\nabla \psi(X)|^2 - \psi(X)h(X)]
%\end{equation}
%This is referred to as the optimization formulation of the Poisson equation~\cref{eq:Poisson}. 
%\end{enumerate}
%One thus obtains the following closed-form formula for the solution $\phi$ of the Poisson equation~\cref{eqn:EL_phi_intro}:  
%\begin{equation*}
%\phi = \sum_{m=1}^N \frac{1}{\lambda_m}\langle e_m,h-\hah \rangle e_m
%\label{eq:spectralSolution}
%\end{equation*}
The following Proposition shows that the  weak
form~\cref{eq:Poisson-weak} and the semigroup
form~\cref{eq:Poisson-semigroup} are equivalent. The proof appears in
the~\cref{apdx:semigoup-equivalent}.  
\begin{proposition}\label{prop:semigroup-equivalent}
Suppose $\rho$ satisfies Assumption~(A1) and $h$ satisfies
Assumption~(A2).   Then the unique solution $\phi \in H_0^1(\pr)$ to the weak form~\cref{eq:Poisson-weak} is also the unique solution to the fixed-point equation~\cref{eq:Poisson-semigroup}.
\end{proposition}

 \medskip

%We present a stochastic representation of the semigroup $\{P_t\}_{t\geq 0}$ in terms of Brownian motion. The representation is derived via a unitary transformation of the weighted Laplacian to the Schr\"{o}dinger operator and application of the Feynman-Kac formula. Later in Sec.~\cref{sec:kernel}, the stochastic representation is used to justify the formula for the approximate semigroup. The proof of the following Proposition appears in the Appendix~\cref{apdx:Feynman-Kac}.
%\begin{proposition}\label{prop:Feynman-Kac}
%Consider the semigroup $\{P_t\}_{t\geq 0}$ associated with the weighted Laplacian $\Delta_\pr$. Then for all $t\geq 0$:
%\begin{equation}\label{eq:Feynman-Kac}
%P_t f(x) = \Expect\left[\frac{ e^{-\int_0^t W(B^x_{2s})\ud s}}{\sqrt{\pr(x)}}\sqrt{\pr(B^x_{2t})} f(B^x_{2t})\right]
%\end{equation}
%where $B^x_t$ is Brownian motion with initial condition $B^x_0=x$, and 
%\begin{equation}
%W(x)\coloneqq\frac{1}{4}|\nabla V(x)|^2-\frac{1}{2}\Delta V(x)
%\end{equation}
%\end{proposition} 

 The semigroup formulation has led to the diffusion-map based algorithm which is the
 main focus of the  remainder of this paper.

\section{Diffusion map-based Algorithm}
\label{sec:kernel}

The diffusion map-based  algorithm is based on a numerical approximation of the
fixed-point equation~\cref{eq:Poisson-semigroup}.  The main technique is to
approximate the semigroup $P_\epsilon$ in the following three
steps:

\begin{enumerate}
\item \newP{Diffusion map approximation:} A family of Markov operators $\{\Teps\}_{\epsilon>0}$ are
  defined as follows: 
% For $f:\Re^d \to \Re$
\begin{equation}
\Teps f(x) := \frac{1}{n_\epsilon(x)}\int_{\Re^d}k_\epsilon(x,y)f(y)\pr(y)\ud y,
\label{eq:Teps-definition}
\end{equation} 
where $n_\epsilon(x):=\int
\keps(x,y)\rho(y) \ud y$ is the normalization factor,
\begin{equation*}
\keps(x,y):= \frac{g_\epsilon(x,y)}{\sqrt{\int
g_\epsilon(x,z) \rho(z) \ud z}\sqrt{\int
g_\epsilon(y,z) \rho(z) \ud z}}, \label{eq:keps}
\end{equation*}
and $g_\epsilon(x,y):=e^{-\frac{|x-y|^2}{4\epsilon}}$ is the Gaussian
  kernel in $\Re$.  For small positive values of $\epsilon$, the
  Markov operator $T_\epsilon$ is referred to as the {\it diffusion map} approximation of
  the exact semigroup
  $P_\epsilon$~\cite{coifman,hein-consistency-2005}. The precise statement of this approximation is contained in~\cref{prop:Tepsn-convergence}. For the
special case of Gaussian density, an
explicit formula for the diffusion map appears in the~\cref{apdx:Laplacian-Gaussian}.   

%\medskip
%
%The nature of
%  the approximation is described by the following Proposition whose
%  proof appears in the Appendix~\cref{proof:prop:Tesp-approx}.
%  \begin{proposition} \label{prop:Teps-approx}
%  \it Suppose the density $\rho$ satisfies Assumption~(A1). Then for $\epsilon$ small enough and for $f \in C_c^\infty(\Re^d)$
%  \begin{equation*}
%  \Teps f = f + \epsilon \Delta_\pr f + \epsilon^2 r_f
%  \end{equation*}
%  where $r_f\in L^2(\pr)$.
%  \end{proposition} 

\item \newP{Empirical approximation:} The operator $\Teps$ is approximated
empirically by $\{\TepsN\}_{\epsilon > 0, N \in \mathbb{N}}$ defined as follows:
\begin{equation}
\TepsN f(x) :=\frac{1}{n_\epsilon^{(N)}(x)}\sum_{j=1}^N \kepsN(x,X^j)f(X^j),
\label{eq:TepsN-definition}
\end{equation} 
where $n_\epsilon^{(N)}(x):=\sum_{i=1}^N\keps(x,X^i)$ is the normalization factor and
\begin{equation*}
\kepsN(x,y) :=  \frac{\geps(x,y) }{\sqrt{\sum_{j=1}^N g_\epsilon(x,X^j)}\sqrt{\sum_{j=1}^N g_\epsilon(y,X^j)}}.
\label{eq:kepsN}
\end{equation*}
Recall that $X^i \overset{\text{i.i.d}}{\sim}\pr$ for
$i=1,\ldots,N$. So, by law of large numbers
(LLN), $\TepsN f$ 
represents an empirical approximation of the diffusion map $\Teps$. The
precise statement of the empirical approximation is contained
in~\cref{prop:TepsN-convergence}.

\item \newP{Approximation as Markov matrix:} 
An $N\times N$ Markov matrix $\Ten$ is defined with 
  $(i,j)$-th element given by
\begin{equation}\label{eq:Ten-definition}
 \Ten_{ij} = \frac{1}{\nepsN(X^i)}\KepsN (X^i,X^j).
\end{equation}
\end{enumerate}

\newP{Finite-dimensional fixed-point equation:} Using the three steps above, the
original infinite-dimensional fixed-point
equation~\cref{eq:Poisson-semigroup} is approximated as a finite dimensional
fixed-point equation
\begin{equation}\label{eq:fixed-pt-finite-N}
{\sf \Phi} =  \Ten {\sf  \Phi} + \epsilon {(\hvec - \pi(h))},
\end{equation}
where ${\hvec} := (h(X^1),\ldots,h(X^N))$ is a
$N\times 1$ column vector, and $\pi(h)=\sum_{i=1}^N \pi_ih(X^i)$ where
the probability vector $\pi_i=\frac{\nepsN(X^i)}{\sum_{j=1}^N\nepsN(X^j)}$ is the unique stationary distribution of the Markov matrix $\Ten$.
The solution $ \sf \Phi$ is used to define an approximation to the solution of the Poisson equation as follows: 
\begin{equation}
\phiepsN(x) := \frac{1}{n_\epsilon^{(N)}(x)}\sum_{j=1}^N \kepsN(x,X^j)
{\sf \Phi}_j + \epsilon  (h(x)-\pi(h)).
\label{eq:phiepsN}
\end{equation}
The approximation for the gain function is as follows:
\begin{equation}
\K_\epsilon^{(N)} (x) = \nabla \left[ \frac{1}{n_\epsilon^{(N)}(x)}\sum_{j=1}^N \kepsN(x,X^j)
({\sf \Phi}_j + \epsilon   \hvec_j)\right].
\label{eq:empirical_formula_for_gain}
\end{equation}
% This approximation is similar to $\nabla \phiepsN$ for small values of
% $\epsilon$ and converges to the constant gain approximation formula
% for large values of $\epsilon$.  
Upon evaluating the gradient in closed-form, the following linear
formula results for the gain function evaluated at particle locations:
\begin{equation}
\K^i :=\K_\epsilon^{(N)}(X^i) = \sum_{j=1}^Ns_{ij}X^j,
\label{eq:gain-linear-form}
%\frac{1}{n_\epsilon^{(N)}(x)}\sum_{i=1}^N \KepsN(x,X^i) \Phi_i + \epsilon h(x)
\end{equation}
where 
\begin{equation}\label{eq:s-r-definition}
s_{ij} := \frac{1}{2\epsilon}\Ten_{ij}(r_j-\sum_{k=1}^N
\Ten_{ik}r_k),\quad r_j := {\sf \Phi}_j + \epsilon {\hvec}_j.
\end{equation}
 The details of the calculation leading to the linear formula appear in the~\cref{apdx:gain-linear-form}.  

\todo{
\begin{remark}[Numerical procedure]  The fixed-point problem~\eqref{eq:fixed-pt-finite-N} is solved in an iterative manner. The vector $\phivec$ is initialized to $\phivec_0=(0,\ldots,0)\in \Re^N$ and updated according to
\begin{equation}\label{eq:iteration}
\phivec_{n+1} = \Ten \phivec_n +  \epsilon {(\hvec - \pi(h))},
\end{equation}
for $n=1,\ldots,L$ 
for a finite number of $L$ iterations. 
The procedure is guaranteed to converge, with a geometric convergence
rate, because $\Ten$ is a strict contraction on $L^2_0(\pi)$ (\cref{prop:N}-(ii)). The overall algorithm is presented
in~\cref{alg:kernel}.      

%In a filtering application of FPF, the iterative
% procedure~\eqref{eq:iteration} is applied at each filtering step. In
% this case, quick convergence is obtained by 
% initializing $\phivec_0$ to the value obtained from the last filtering
% step. The reason is that the change in the solution of the fixed point
% equation~\eqref{eq:fixed-pt-finite-N} is (typically) small from one filtering step
% to the next.  This is because the change in particle locations is
% (typically) small for a small choice of time increment.  

The proposed iterative procedure~\eqref{eq:iteration} is preferred to
other numerical procedures because (i) it is straightforward to
implement and does not require matrix inversion; (ii) it may be
numerically more efficient than solving a system of $N$ linear
equations;  and (iii) it
allows one to use the solution obtained from the previous filter step,
as initialization for the iterative procedure~\eqref{eq:iteration},
resulting in quick convergence -- typically in a few iterations. The
reason for quick convergence is that the change in the solution of the
fixed point equation~\eqref{eq:fixed-pt-finite-N} is (typically) small
from one filtering step to the next.  This is because the change in
particle locations is (typically) small for a small choice of time
increment.  
\end{remark}
}
 % The convergence analysis of this algorithm, as $\epsilon\downarrow 0$
% and $N\rightarrow \infty$, is the subject of \Sec{sec:error_kernel}.

\begin{algorithm}
 \caption{diffusion-map based algorithm for gain function approximation}
 \begin{algorithmic}[1]
     \REQUIRE $\{X^i\}_{i=1}^N$, $\{h(X^i)\}_{i=1}^N$, $\phivec_{\text{prev}}$, $\epsilon$, L
     \ENSURE $\{\K^i\}_{i=1}^N$ \medskip
     \STATE Calculate $g_{ij}:=e^{-\frac{|X^i-X^j|^2}{4\epsilon}}$ for $i,j=1$ to $N$\medskip
     \STATE Calculate $k_{ij}:=\frac{g_{ij}}{\sqrt{\sum_l g_{il}}\sqrt{\sum_l g_{jl}}}$ for $i,j=1$ to $N$
     \STATE Calculate $d_i= \sum_{j} k_{ij}$ for $i=1$ to $N$
     \STATE Calculate $\Ten_{ij}:=\frac{k_{ij}}{d_i}$ for $i,j=1$ to $N$
     \STATE Calculate $\pi_i=\frac{d_i}{\sum_{j}d_j}$ for $i=1$ to $N$
     \STATE Calculate $\hat{\hvec}= \sum_{i=1}^N \pi_jh(X^i)$ \medskip
     \STATE Initialize $\phivec=\phivec_{\text{prev}}$ 
     \medskip
     \FOR {$t=1$ to  L}
     \STATE  $\phivec_i= \sum_{j=1}^N \Ten_{ij} \phivec_j + \epsilon (\hvec-\hat{\hvec})$ for $i=1$ to $N$\medskip
%     \STATE Calculate $\Phi_i = \Phi_i - \sum_{j=1}^N \pi_j\Phi_j$
     \ENDFOR
     \STATE Calculate $r_i = \phivec_i + \epsilon \hvec_i$ for $i=1$ to $N$
     \STATE Calculate $s_{ij} = \frac{1}{2\epsilon}\Ten_{ij}(r_j-\sum_{k=1}^N
     \Ten_{ik}r_k)$ for $i,j=1$ to $N$
     \STATE Calculate $\K^i = \sum_j s_{ij}X^j$ for $i=1$ to $N$
%      \[
%      \K^i = \frac{1}{2\epsilon}\sum_{j=1}^N \left[T_{ij}(\Phi_j + \epsilon(h(X^j) 
%      - \hat{h}_\pr^{(N)}))\left(X^j - \sum_{k=1}^N T_{ik}X^k\right)\right]
%      \]
 \end{algorithmic}
 \label{alg:kernel}
\end{algorithm}

\begin{remark} \normalfont The computational complexity of the diffusion-map based
  algorithm is $O(N^2)$ because of the need to assemble the $N\times
  N$ matrix $\Ten$.  The computational complexity may be reduced using the sparsity structure of the matrix $\Ten$ and sub-sampling techniques. Compared to the Galerkin algorithm with computational complexity of $O(Nd^3)$, the diffusion-map algorithm is advantageous in high-dimensional problems where $d>>N$.
% A subject of future work is to reduce the computational complexity by exploiting the sparsity structure of the matrix $\Ten$. 
\end{remark}

\subsection{Approximation results}

The notation $G_\epsilon (f)(x) := \int \geps(x,y)f(y)\ud y$ is used to
denote the heat semigroup with a Gaussian kernel $\geps(x,y)$, and
\begin{subequations}
\begin{align}
%\pr_g(x;\Sigma)&\coloneqq (2\pi)^{-d} (\det{\Sigma})^{-1/2} \exp(-\frac{1}{2}x^\top \Sigma^{-1}x)\\
%G_Q(f)(x)&\coloneqq \int \pr_g(y-x;2Q)f(y)\ud y
%\\
U_\epsilon &\coloneqq \frac{1}{2}\log(\frac{G_\epsilon(\pr)}{\pr^2}),\quad U \coloneqq- \frac{1}{2}\log (\pr),\label{eq:Ueps-definition}\\ 
W_\epsilon&\coloneqq \frac{1}{\epsilon}\log(e^{U_\epsilon}G_\epsilon (e^{-U_\epsilon})),\quad W\coloneqq|\nabla U|^2-\Delta U.
\label{eq:Weps-definition}
\end{align}	
\end{subequations}

%where $G_\epsilon (f)(x) = \int \geps(x,y)f(y)\ud y$ is the heat semigroup.

The proof of the following proposition appears
in~\cref{apdx:Tepsn-convergence}.   
\begin{proposition} \label{prop:Tepsn-convergence}
	Consider the family of Markov operators
        $\{\Teps\}_{\epsilon>0}$ defined according
        to~\cref{eq:Teps-definition}.  Let $n \in \mathbb{N}$, $t \in
        (0,t_0)$ with $t_0<\infty$, and $\epsilon = \frac{t}{n}$. Then,
	%	Assume there exists positive constants $\alpha$ and $\beta$ such that $W_\epsilon(x)\geq \alpha|x|^2-\beta$.
%	Then there exists positive constants $\lambda$, $b$, $R$, $\delta$, a probability measure $\nu$, and a number $n_0 \in \mathbb{N}$ (all independent of $n$) such that for all $n>n_0$:
%%	Let $\pr(x) = \pr_g(x;\Sigma)e^{-w(x)}$ and $U(x)= - \frac{1}{2}\log(\pr)$. Assume $\Sigma \in S^d_{++}$ and $w \in C_b^\infty$. Define the integral operator:
%	\begin{equation}\label{eq:nTeps}
%	\nTeps f \coloneqq \frac{G_\epsilon  (f e^{-U})}{G_\epsilon(e^{-U})}
%	\end{equation}
%	for all $\epsilon>0$ and define the functions 
%	\begin{align*}
%	W &= |\nabla U|^2 - \Delta U\\
%	W_\epsilon &= \frac{1}{\epsilon}\log(e^{U}G_\epsilon (e^{-U}))
%	\end{align*}
%	Then
	\begin{romannum}
		\item The semigroup $P_t$ and the operator $\Teps^n$ admit the following representations: 
		\begin{align}
		P_tf(x) &= e^{U(x)}\Expect[e^{-\int_0^t W(B^x_{2s})\ud s} e^{-U(B^x_{2t})} f(B^x_{2t})],\label{eq:Pt-Feynman-Kac}\\
		\nTeps^n f(x) &= e^{U_\epsilon(x)}\Expect[e^{-\epsilon\sum_{k=0}^{n-1}W_\epsilon(B^x_{2k\epsilon})}e^{-U_\epsilon(B^x_{2n\epsilon})}f(B^x_{2n\epsilon})],\label{eq:Teps-Feynman-Kac}
		\end{align}
		for all $x\in \Re^d$ where $B^x_t$ is the Brownian motion with initial condition $B_0^x=x$.
		\item In the asymptotic  limit as $\epsilon \to 0$:
		\begin{subequations} 
		\begin{align}
		%|U_\epsilon(x)-U(x)| &\leq \epsilon^2(\CA|x|^2+\CB)\label{eq:Ueps-U}\\
		U_\epsilon(x)&=U(x) + 2\epsilon W(x) + \epsilon \Delta V(x) + \epsilon^2 r^{(1)}_\epsilon(x), \label{eq:Ueps-U}\\ 
%|\nabla U_\epsilon(x)-\nabla U(x)| &\leq \epsilon(\CA|x|+\CB)\label{eq:nabla-Ueps-U}\\
W_\epsilon(x)&=W(x)+ \epsilon r^{(2)}_\epsilon(x),\label{eq:Weps-W}
		\end{align}
		\end{subequations}
	where $|r^{(1)}_\epsilon(x)|,|r^{(2)}_\epsilon(x)| = O(|x|^2)$ and $|\nabla r^{(1)}_\epsilon(x) |= O(|x|)$ as $|x|\to \infty$.
		\item For all functions $f$ such that $f,\nabla f \in L^4(\pr)$:
		\begin{equation}\label{eq:Teps-Peps-limit}
		\|(T_{\frac{t}{n}}^n - P_t)f\|_{L^2(\pr)} \leq \frac{\sqrt{t}}{n}C(\|f\|_{L^4(\pr)}+\|\nabla f\|_{L^4(\pr)}),
		\end{equation}
		where the constant $C$ only depends on $t_0$ and $\pr$.
		%		\item The operator $T_\epsilon$ is a Markov operator with invariant probability distribution 
		%		\begin{equation}
		%		\preps(x) \coloneqq \frac{\neps(x)\pr(x)}{\int \neps(x)\pr(x)\ud x}
		%		\end{equation}
		%		\item The operator $T_\epsilon$ is symmetric on $L^2(\preps)$ and a contraction on $L^2_0(\preps)$
		%		\item The operator $T_\epsilon$ is compact on $L^2(\preps)$ and admits a discrete spectrum, and strict contraction on $L^2_0(\preps)$
%		\item Let $f \in C^\infty(\Re^d)$, such that $\partial^2 f ,\partial^{4}(fe^{-\frac{w}{2}})\in L^\infty$ Then, in the asymptotic limit as $\epsilon \to 0$:
%		\begin{equation}
%		\Teps f(x) = f(x) + \epsilon \Delta_\pr f(x) + O(\epsilon^2)
%		\end{equation}
		%where the remainder term has at most quadratic growth.
	\end{romannum}
\end{proposition} 

% In the following Proposition, we present the convergence of the
% empirical operator $\TepsN$ to $\Teps$ as $N \to \infty$ using law of
% large numbers. 
The proof of the following proposition appears
in~\cref{apdx:TepsN-convergence}.

\begin{proposition}\label{prop:TepsN-convergence}
	Consider the diffusion map kernel $\{\Teps\}_{\epsilon>0}$, and its empirical approximation $\{\TepsN\}_{\epsilon>0,N\in \mathbb{N}}$. Then for any bounded continuous function $f \in C_b(\Re^d)$:
	\begin{romannum}  
		\item (Almost sure convergence) For all $x \in \Re^d$
		\begin{equation*}
		\lim_{N \to \infty} \TepsN f(x) = \Teps f(x),\quad \text{a.s.}
		\end{equation*}
		\item (Convergence rate) For any $\delta \in (0,1)$, in the asymptotic limit as $N \to \infty$,  
		\begin{equation}
		\int |\TepsN f(x) - \Teps f(x)|^2\pr(x)\ud x \leq O(\frac{\log(\frac{N}{\delta})}{N\epsilon^{d}}), 
		\end{equation}
		with probability higher than $1-\delta$.
	\end{romannum}
	%%    empirical approximation of the diffusion 
\end{proposition}

\begin{remark}[Related work]
	The key idea in the proof of the~\cref{prop:Tepsn-convergence}
        is the Feynman-Kac representation of the
        semigroup~\cref{eq:Pt-Feynman-Kac}.  To the best of our
        knowledge, this representation has not been used before in the analysis of the diffusion map approximation.
	Most of the existing results concerning the convergence of the
        diffusion map are based on a Taylor series expansion that
        would lead to a convergence of the form $\lim_{\epsilon \to
          0}\frac{f(x)-\Teps f(x)}{\epsilon} = \Delta_\pr f(x)$ for each
        $x\in \Re^d$~\cite{hein-consistency-2005,coifman,gine2006empirical}.  
	Convergence results of the form 
        $\lim_{n \to \infty} \|T_{\frac{t}{n}}^n f - P_tf\|_{L^2(
        \pr)}=0$
        appear in~\cite{coifman,ting2011analysis}, based on functional
        analytic arguments. The Taylor series
        type arguments typically require the distribution to be supported on a
        compact manifold which not assumed here. 
\end{remark} 

\section{Convergence and error analysis}
\label{sec:error_kernel}

The analysis of the diffusion-map algorithm involves the consideration of the following four fixed point problems:
\begin{align}
\text{(exact)}&\quad& \phi &= P_\epsilon \phi + \int_0^\epsilon P_s(h-\hat{h}_\pr)\ud s,
%- \Delta_\pr^{-1}  (h-\hat{h}_\pr)
& \label{eq:exact-fixed-pt}\\
\text{(diffusion-map approx.)}&\quad& \phieps &= \Teps \phieps + \epsilon (h-\hat{h}_{\pr_\epsilon}),&\label{eq:kernel-approx-fixed-pt} \\
\text{(empirical approx.)}&\quad&\phiepsN & = \TepsN \phiepsN + \epsilon (h-\pi(h)),
%  \frac{1}{n_\epsilon^{(N)}(x)}\sum_{j=1}^N \KepsN(x,X^j)
%{\sf \Phi}_j + \epsilon \TepsN (h-\hat{h}_\pr^{(N)})(x)& 
\label{eq:empirical-fixed-pt}\\
\text{(finite-dim.)}&\quad& {\sf \Phi} & =  \Ten {\sf  \Phi} + \epsilon
({\hvec} - \pi(\hvec)),\label{eq:finite-N-fixed-pt}
\end{align} 
where $\hat{h}_{\pr_\epsilon}:= \int h(x)\preps(x)\ud x$ and
$
\preps(x):=\frac{n_\epsilon(x)\pr(x)}{\int n_\epsilon(x)\pr(x)\ud x}
$
is the density of the invariant probability distribution associated
with the Markov operator $\Teps$.

In practice, the finite-dimensional
problem~\cref{eq:finite-N-fixed-pt} is solved.  The
existence and uniqueness of the solution for this problem is the
subject of the following proposition whose proof appears in~\cref{proof:prop:N}.

\begin{proposition} \label{prop:N}
Consider the finite-dimensional fixed point
equation~\cref{eq:finite-N-fixed-pt}. \\
Then almost surely
\begin{romannum}
\item $\Ten$ is a reversible Markov matrix with a unique stationary
  distribution 
\begin{equation} \label{eq:pi-def}
\pi_i:=\frac{\nepsN(X^i)}{\sum_{j=1}^N \nepsN(X^j)},
\end{equation}
for $i=1,\ldots,N$.
\item $\Ten$ is  a strict contraction on $L^2_0(\pi)=\{v \in \Re^N;
  \sum \pi_i v_i=0\}$. Hence the fixed point
  equation~\cref{eq:finite-N-fixed-pt} has a unique solution $\phivec
  \in L^2_0(\pi)$. 

\item The (empirical approx.) fixed point
  equation~\cref{eq:empirical-fixed-pt} has a unique solution given
  by (see~\cref{eq:phiepsN})
\[
\phiepsN(x) = \frac{1}{n_\epsilon^{(N)}(x)}\sum_{j=1}^N \kepsN(x,X^j)
{\sf \Phi}_j + \epsilon  (h(x)-\pi(h)).
\]
\end{romannum}
\end{proposition}

Based on the results in~\cref{prop:semigroup-equivalent} and \cref{prop:N}, the exact solution
$\phi$ and the numerical solution $\phiepsN$ are both
well-defined.  
The remaining task is to show the convergence of $\phiepsN
\to \phi$ as $N \to \infty$ and $\epsilon \to 0$. We break the convergence analysis into two parts, bias and variance:
\begin{equation*}
\phiepsN \underset{\text{(variance)}}{\overset{N \uparrow \infty}{\longrightarrow}} \phieps
\underset{\text{(bias)}}{\overset{\epsilon \downarrow 0}{\longrightarrow}} \phi .
\end{equation*}

Before describing the general result, it is useful to first introduce an
example that helps illustrate the bias-variance trade-off in this problem.

\subsection{Example - the scalar case}
\label{ex:scalar}

In the scalar case (where $d=1$), the Poisson equation is:
\begin{equation*}
-\frac{1}{\rho(x)}\frac{\ud }{\ud x}(\rho(x) \frac{\ud \phi}{\ud
  x}(x)) = h(x)-\hah.
\end{equation*}
Integrating twice yields the solution explicitly
\begin{equation}
\begin{aligned}
\K_{\text{exact}} (x) = \frac{\ud \phi}{\ud x}(x) &= -\frac{1}{\rho(x)}\int_{-\infty}^x
\rho(z)(h(z)-\hah)\ud z.
%\\
%\phi(x) &= -\int_{-\infty}^{x}\frac{\ud y}{\rho(y)}\int_{-\infty}^y \rho(z)(h(z)-\hah)\ud z  
\end{aligned}
\label{eq:scalar}
\end{equation}

For the choice of $\rho$ as the sum of two Gaussians
${\cal N}(-1,\sigma^2)$ and ${\cal N}(+1,\sigma^2)$ with $\sigma^2=0.2$ and $h(x)=x$, the solution
obtained using~\cref{eq:scalar} is depicted in
\cref{fig:kernel-approx}~(a).  Also depicted is the approximate
solution obtained using the diffusion-map algorithm with $N=200$, for different values of $\epsilon$. The constant gain approximation is evaluated according to the explicit integral formula~\eqref{eq:const-gain-approx}.  
As $\epsilon \to \infty$ the approximate gain converges to the
constant gain approximation. As $\epsilon$ becomes smaller, the
approximation becomes more accurate.
However, for very small values of $\epsilon$ the approximation is poor due to the variance error. 

The bias-variance trade-off while varying the the parameter $\epsilon$ is depicted in \cref{fig:kernel-approx}~(b). The $L^2$ error is computed as a Monte-Carlo average:
%n order to provide a quantitative assessment, the following
%Monte-Carlo average is computed for the $L^2$ error:
\begin{equation}
\text{m.s.e} = \frac{1}{M}\sum_{m=1}^M \frac{1}{N}\sum_{i=1}^N | {\sf
  K}^{(m)}(X^i) - {\sf K}_{\text{exact}}(X^i) |^2.
\label{eq:MC_error}
\end{equation}
\cref{fig:kernel-approx}~(b)
depicts the error obtained from averaging over $M=1000$
simulations as a function of the parameter $\epsilon$. It is observed
that for a fixed number of particles $N$, there is an optimal value of
$\epsilon$ that minimizes the error. 
\begin{figure}[t]
\centering
%\begin{tabular}{cc}
\subfloat[]{
\includegraphics[width=0.455\columnwidth]{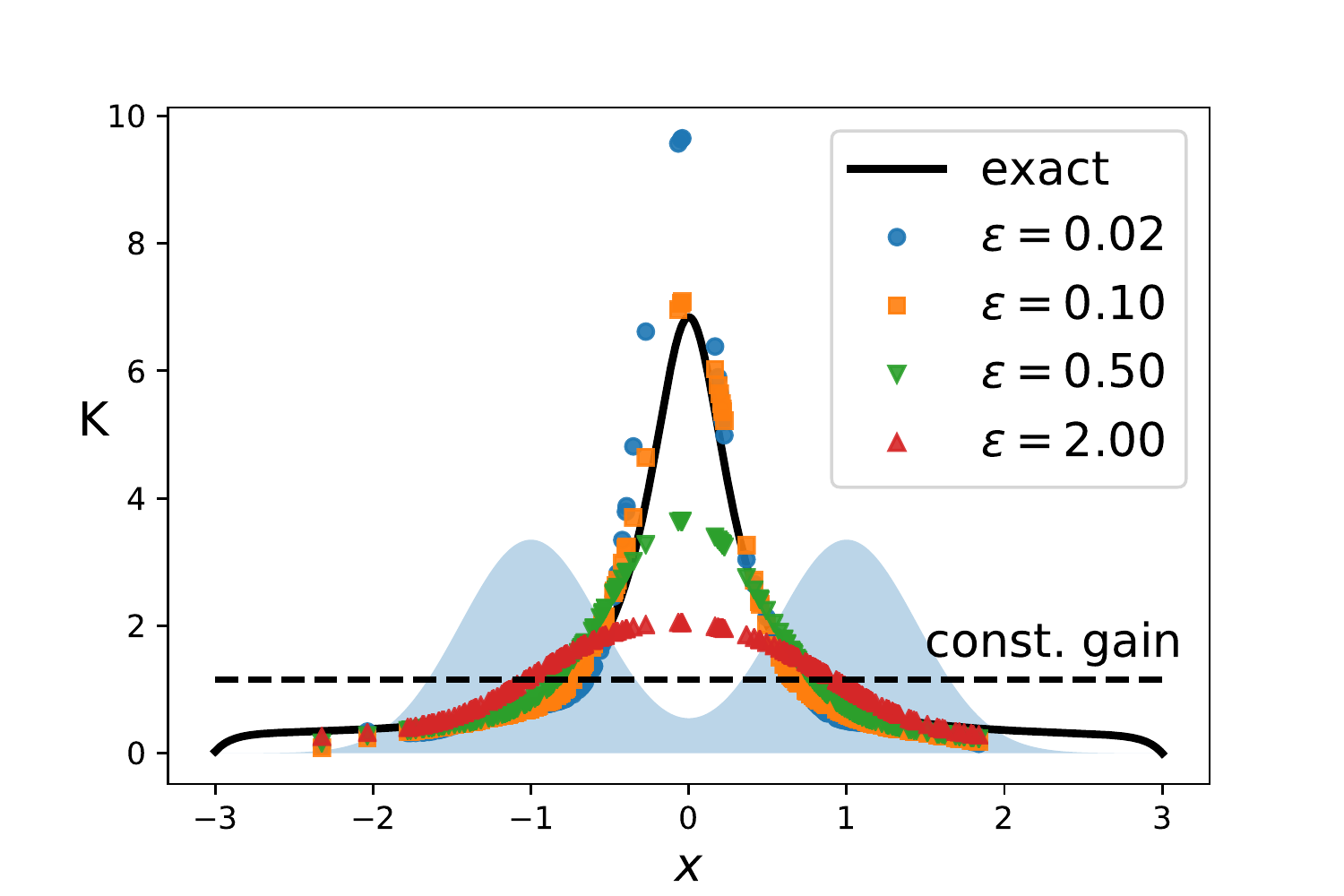}
}
\subfloat[]{
\includegraphics[width=0.445\columnwidth]{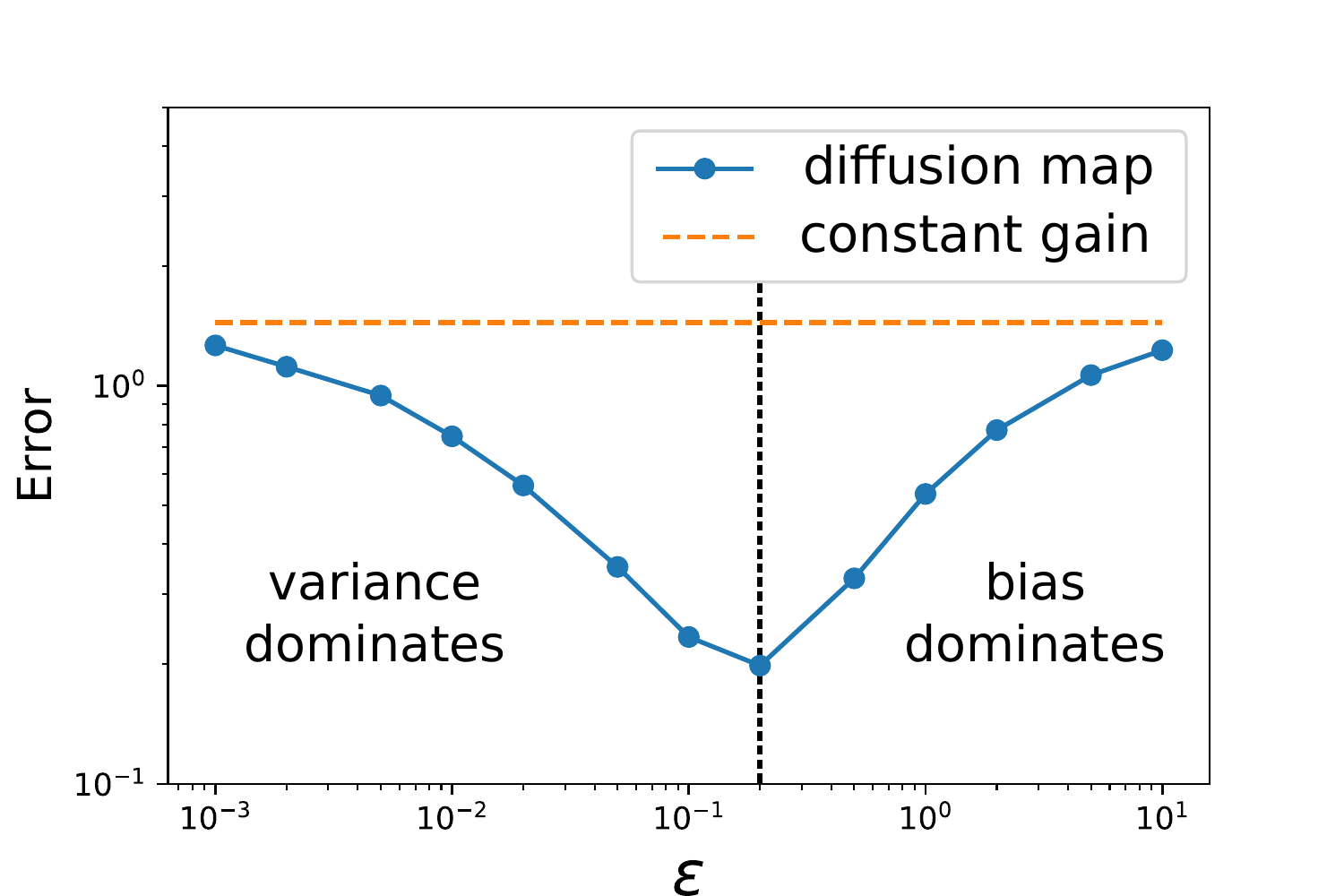}
}
%\end{tabular}
\caption{Simulation results for the diffusion-map algorithm for the
  scalar bimodal example: (a) Approximate gain function for different
  choices of $\epsilon$ compared to the exact gain function (solid
  line).  The shaded area in the background 
  is the bimodal probability density function $\pr$.  The dashed line is the constant
  gain approximation solution; (b) Gain function approximation error
  of the diffusion-map algorithm as a function of the parameter
  $\epsilon$.  All the results are with $N=200$ particles.}
\label{fig:kernel-approx}
\end{figure} 

The vector counterpart of this example appears in \cref{ex:vector}.

\subsection{Bias}

The analysis of bias has two parts:
\begin{enumerate}
\item To show that the (diffusion-map) fixed-point
  equation~\cref{eq:kernel-approx-fixed-pt} admits a unique solution
  $\phieps$ 
  for {\em all} positive choices of $\epsilon$;
\item To show that $\phieps \rightarrow \phi$ as $\epsilon \downarrow
  0$.  
\end{enumerate}

% The analysis of the bias error concerns the existence of the solution
% to the (kernel) fixed-point
% equation~\cref{eq:kernel-approx-fixed-pt} and its convergence to the
% exact solution of the semigroup
% formulation~\cref{eq:exact-fixed-pt}. 

For $n \in \mathbb{N}$, iterate the fixed-point equation~\cref{eq:kernel-approx-fixed-pt} $n$ times to obtain:
\begin{equation}\label{eq:fixed-point-n}
\phi_\epsilon = \Teps^n \phieps + \sum_{k=0}^{n-1} \epsilon \Teps^k(h-\hat{h}_{\pr_\epsilon}).
\end{equation}
We let $\epsilon = \frac{t}{n}$ for some $t>0$ and study the solution
of this fixed-point equation as $n \to \infty$. Note that the solution
to the iterated fixed-point equation~\eqref{eq:fixed-point-n} is
identical to the solution to the fixed-point
equation~\cref{eq:kernel-approx-fixed-pt}.  
%Moreover, the form the iterated fixed-point equation is more similar to the exact fixed-point equation~\cref{eq:exact-fixed-pt}. 

The fixed-point equation~\cref{eq:fixed-point-n}  is the (discrete)
Poisson equation that appears in the theory of Markov chain simulation
\cite{glynn96,MT} and stochastic
control~\cite[Ch. 9]{meyn2008control}. Theory presented in these
references illustrates how bounds on the solution are obtained
under a Foster-Lyapunov drift condition.  A similar strategy is
adopted here.    

In the following proposition, an existence-uniqueness result is
described for the fixed-point equation~\cref{eq:fixed-point-n}.  The technical
step in the proof involves a Foster-Lyapunov condition known as
DV(3)~\cite{konmey12a}.  The proof appears in~\cref{apdx:DV3}.

%Define the following notations:

%The next objective is to show that the operator $T^n_{\frac{t}{n}}$ admits a spectral gap. The spectral gap is shown 
\begin{proposition}
	\label{prop:DV3}
	Consider the family of Markov operators $\{\Teps\}_{\epsilon>0}$ defined in \cref{eq:Teps-definition}.  Let $n \in \mathbb{N}$, $t \in (0,t_0)$, and $\epsilon = \frac{t}{n}$, with $t_0<\infty$.  
	%	Assume there exists positive constants $\alpha$ and $\beta$ such that $W_\epsilon(x)\geq \alpha|x|^2-\beta$.
	Then there exists positive constants $a$, $b$, $R$, $\delta$, a probability measure $\nu$, and a number $n_0 \in \mathbb{N}$  such that for all $n>n_0$:
	\begin{subequations}
		\begin{align}
		\log(e^{-U_\epsilon} T_{\epsilon}^n e^{U_\epsilon}) &\leq -a tU_\epsilon+ bt,\label{eq:DV3}\\
		T_{\epsilon}^n \mathds{1}_A(x)&\geq \delta \nu(A) \mathds{1}_{\|x|\leq R} \quad \forall A\ \in \mathcal{B}(\Re^d).\label{eq:minorization}
		\end{align}
		%		\label{eq:DV3}
	\end{subequations}
	Consequently,
	\begin{romannum}
		\item The chain with transition kernel $T_{\epsilon}^n$ is geometrically ergodic with invariant density 
		\begin{equation}\label{eq:preps}
		\preps(x):=\frac{n_\epsilon(x)\pr(x)}{\int n_\epsilon(x)\pr(x)\ud x}.
		\end{equation}

		\item $T_{\epsilon}^n$ is reversible with respect to the density $\preps$ 
		It admits a spectral gap as a linear operator
                $T_{\epsilon}^n :L^2_0(\pr_{\epsilon})\to
                L^2_0(\pr_{\epsilon})$ that is uniform with respect to $\epsilon$.  The spectral gap is denoted
                as $\lambda$.  
%		 The spectral gap is strictly positive uniform with respect to $n$. 
		
		\item There exists a solution to~\cref{eq:fixed-point-n} with the bound 
		\begin{equation*}
		\|\phieps\|_{L^2(\preps)} \leq \frac{t\|h\|_{L^2(\preps)}}{\lambda}.
		\end{equation*}
%		where $\lambda'$ is the spectral gap for $\Teps^n$
%		For any function $f\in L^{U_\epsilon}_\infty$ there is a solution to Poisson's equation $\xi\in L^{U_\epsilon}_\infty$:
%		\[
%		T_{\epsilon}^n\xi =  \xi - f + \hat{f}_{\epsilon}
%		\]
%		where $\hat{f}_{\epsilon} = \int f(x)  \pr_{\epsilon}(x) \ud x$.   
		%Moreover, if $\xi'\in L_\infty^{U_r}$ is another solution,  then  the difference $\xi-\xi'$ is zero a.e.\  $[\preps]$.  
	\end{romannum}
	%	where $U(x)=-\log(\alpha_\epsilon(x))$ where $\alpha_\epsilon$ is defined in Prop.~\cref{prop:Teps-approx-new}.
\end{proposition}

The proof of the following main result appears
in~\cref{apdx:bias}.  

\begin{theorem} 
Suppose the assumptions (A1)-(A2) hold for the density $\rho$ and the
function $h$, and $\phi$ denotes the exact solution
of~\cref{eq:exact-fixed-pt}. 
%with an associated gain function $\K=\nabla\phi$ (see
%\cref{prop:exact}).  
Consider the approximation of this problem
defined by the (diffusion-map)
fixed-point equation~\cref{eq:kernel-approx-fixed-pt}.  For the approximate problem:
\begin{romannum}
\item{\textbf{Existence-Uniqueness:}} For each fixed $\epsilon>0$, there exists a unique solution
  $\phieps$.
\item{\textbf{Convergence:}} In the asymptotic limit as $\epsilon \to 0$
\begin{equation}\label{eq:bias-bound}
\|\phieps- \phi\|_{L^2(\preps)} \leq O(\epsilon).
\end{equation}
%For the special case where the density $\rho$ is Gaussian, the
%constant $C_h \leq
%\|\nabla h\|_\infty \sigma \sqrt{d} + \|\nabla^2
%h\|_\infty(1+\sigma\sqrt{d})+\|D^3 h\|_\infty$. 
% \item{\textbf{Estimate for bias:}} Suppose the density $\rho$ is a Gaussian ${\cal
%     N}(0,\sigma^2 I)$. In the asymptotic limit as $\epsilon \downarrow 0$
% \[
%  \| \K - \Keps\|_{L^2(\rho)} = C _h\epsilon + O(\epsilon^2)
% \]
% where the constant $C_h = \|\nabla h\|_\infty \sigma \sqrt{d} + \|\nabla^2 h\|_\infty(1+\sigma\sqrt{d})+\|D^3 h\|_\infty$. 
\end{romannum}
\label{thm:bias}
\end{theorem}

%\begin{remark}{}
%	The spectrum of the map $\Teps$ is studied in~\cite{belkin2007convergence} where it was shown that assuming $\pr$ is uniform distribution over a compact space the convergence of the spectrum of $\Teps$ is studied. It is assumed in~\cite{belkin2007convergence} that $\pr$ is uniform on a compact manifold. Here we are assuming non-uniform distribution over an unbounded space. Moreover our approach is new, a probabilistic approach relying on stability theory of Markov chains. 
%\end{remark} 
%\begin{remark}
%The previous work on convergence of the diffusion map
%\end{remark}

\subsection{Variance}

The analysis of the variance concerns the (empirical) fixed-point
equation~\cref{eq:empirical-fixed-pt} whose solution is denoted as
$\phiepsN$.  The parameter $\epsilon$ is assumed to be positive and
fixed and $N$ is assumed to be finite but large.    

The existence-uniqueness of $\phiepsN$ has already been shown as part
of Prop.~\ref{prop:N}.  The convergence has only been shown below only for the
case where the density has a compact support.      

% It is used to define the empirical approximation $\KepsN$ of the gain function
% in~\cref{}.  The existence uniqueness result for $\phiepsN$ has
% already been shown in \cref{prop:N}.  The following characterizes
% convergence   
% n the following Proposition, we present the convergence of the empirical operator $\TepsN$ to $\Teps$ as $N \to \infty$ using law of large numbers. The proof appears in the~\cref{apdx:TepsN-convergence}   

% \begin{proposition}\label{prop:TepsN-convergence}
% 	Consider the diffusion map kernel $\{\Teps\}_{\epsilon>0}$, and its empirical approximation $\{\TepsN\}_{\epsilon>0,N\in \mathbb{N}}$. Then for any bounded continuous function $f \in C_b(\Re^d)$
% 	\begin{romannum}  
% 		\item (almost sure convergence) for all $x \in \Re^d$
% 		\begin{equation*}
% 		\lim_{N \to \infty} \TepsN f(x) = \Teps f(x)\quad \text{a.s}
% 		\end{equation*}
% 		\item (convergence rate) For any $\delta \in (0,1)$, and in the asymptotic limit as $N \to \infty$,  
% 		\begin{equation}
% 		\int |\TepsN f(x) - \Teps f(x)|^2\pr(x)\ud x = O(\frac{\log(\frac{N}{\delta})}{N\epsilon^{d}})
% 		\end{equation}
% 		with probability higher than $1-\delta$.
% 	\end{romannum}
% 	%%    empirical approximation of the diffusion 
% \end{proposition}

\noindent
{\bf Assumption A3:} The distribution $\pr$ has compact support given by $\Omega \subset \Re^d$.
%\medskip

\begin{theorem} \label{thm:conv-variance}
Suppose the assumptions (A2)-(A3) hold for the density $\rho$ and the
function $h$, and $\phieps$ denotes the solution
of the (kernel) fixed-point equation~\cref{eq:kernel-approx-fixed-pt}  for a fixed
positive parameter $\epsilon$.  Consider the approximation
of this problem defined by the (empirical) fixed-point
equation~\cref{eq:empirical-fixed-pt}.  For the approximate problem:
\begin{romannum}
\item{\textbf{Existence-Uniqueness:}} For each finite $N$, there
  exists (almost surely) a unique solution
  $\phiepsN$.
\item{\textbf{Convergence:}} The approximate solution
  $\phiepsN$ converges to the kernel solution $\phieps$   
	\begin{equation}\label{eq:phiepsN_phieps_conv_result}
	\lim_{N \to \infty} \|\phiepsN -\phieps\|_{L^\infty(\Omega)} = 0,\quad \text{a.s.}
	\end{equation}
% \item{\textbf{Estimate for variance:}} Suppose the density $\rho$ is a Gaussian ${\cal
%     N}(m,\Sigma)$. In the asymptotic limit as $\epsilon \downarrow 0$
%   and $N\rightarrow \infty$:
% \[
% \expect[\| \Keps - \KepsN\|_{L^2(\rho)}]  = \frac{C}{N^{1/2}\epsilon^{1+d/4}}+ O(N^{-1})
% \]
% where the constant $C = BLAH$. 
\end{romannum}
\label{thm:variance}
\end{theorem}

\todo{
The proof of the convergence $\phiepsN \to \phieps$ is based on 
classical results in the numerical analysis of integral equations on a
grid~\cite{anselone1971collectively,atkinson1976survey}. It relies on the verification of the following three
conditions:
\begin{romannum}
	\item The family of operators
	$\{\TepsN\}_{N=1}^\infty$ is collectively compact 
	as linear operators on $C_b(\Omega)$.
	\item For any function $f \in C_b(\Omega)$,
	\begin{equation}\label{eq:TepsN-unif-convergence-1}
	\lim_{N\to \infty} \|\TepsN f - \Teps f\|_{L^\infty(\Omega)} = 0,\quad \text{a.s.}
	\end{equation}
	\item The inverse  $(I-\Teps)^{-1}$ exists and it is a bounded on $C_0(\Omega):=\{f\in C_b(\Omega); \hat{f}_\pr = 0\}$.
\end{romannum} 
Once these three conditions have been verified, the
convergence result~\cref{eq:phiepsN_phieps_conv_result} follows from a
standard result in the approximation theory of the numerical solutions of integral
equations~\cite[Thm. 7.6.6]{hutson2005}. The proof appears in \cref{apdx:variance}.

%The result in~\cref{thm:variance} establishes asymptotic convergence of the variance error to zero, but it does not include a convergence rate. Obtaining a convergence rate is rather challenging, because it is difficult to obtain a rate for the almost sure convergence  of $\|\TepsN f - \Teps f\|_\infty$ to zero.  
%Although, we did not obtain a non-asymptotic convergence rate for the variance error, an error bound can be formally concluded. Express the difference $\phieps-\phiepsN$ as
\begin{remark}[Convergence rate] \label{rem:variance-error}
The result in~\cref{thm:variance} establishes asymptotic convergence
of the variance error to zero.  However, it does not provide an
explicit form for the convergence rate.  It is possible to obtain an
explicit form based upon a convergence rate estimate for the uniform
convergence~\eqref{eq:TepsN-unif-convergence-1}.  The latter is
difficult because the existing result in~\cite{gine2006empirical}
holds only under rather strong regularity conditions on $f$ and assumes that
the distribution $\pr$ is uniform. 

Based upon the approximation result \cref{prop:TepsN-convergence},
suppose a convergence rate holds
for~\eqref{eq:TepsN-unif-convergence-1} with order
$O(\frac{1}{N^{1/2}\epsilon^{d/2}})$.  In this case, it is
straightforward to derive the following explicit form of the
convergence rate for the variance:  
\begin{equation*}
\|\phieps - \phiepsN\|_{L^\infty(\Omega)} \leq O(\frac{1}{N^{1/2}\epsilon^{1+d/2}}).
\end{equation*}
The validity and tightness of this bound is studied using numerical experiments in \cref{sec:numerics}. 
\end{remark}
%Although, we did not obtain a non-asymptotic convergence rate for the variance error, an error bound can be formally concluded. Express the difference $\phieps-\phiepsN$ as
%%\begin{remark} (related work)
%The proof of the convergence $\phiepsN \to \phieps$ is based on 
%results in the numerical analysis of integral equations on a
%grid~\cite{anselone1971collectively,atkinson1976survey}. A related 
%approach is used in~\cite{von2008consistency} to show the consistency
%of spectral clustering. 
%\end{remark}

%\begin{remark}
%The result in~\cref{thm:variance} establishes asymptotic convergence of the variance error to zero, but it does not include a convergence rate. Obtaining a convergence rate is rather challenging, because it is difficult to obtain a rate for the almost sure convergence  of $\|\TepsN f - \Teps f\|_\infty$ to zero.  
%Although, we did not obtain a non-asymptotic convergence rate for the variance error, an error bound can be formally concluded. Express the difference $\phieps-\phiepsN$ as
%%\begin{equation*}
%%\phieps- \phiepsN = \Teps (\phieps - \phiepsN)  + (\TepsN\phiepsN - \Teps\phiepsN) + \epsilon (\pi(h)-\hat{h}_{\pr_\epsilon})
%%\end{equation*}
%%As a result
%\begin{equation*}
%\phieps- \phiepsN = (I-\Teps)^{-1}(\TepsN\phiepsN - \Teps\phiepsN+ \epsilon (\pi(h)-\hat{h}_{\pr_\epsilon})) 
%\end{equation*}
%From~\cref{prop:DV3}, we have $\|(I-\Teps)^{-1}\| \leq \frac{1}{1-\|\Teps\|} \leq \frac{1}{\lambda \epsilon}$ and from~\cref{prop:TepsN-convergence}, we have $\|\TepsN \phiepsN -\Teps \phiepsN\|\leq O(\frac{1}{N^{1/2}\epsilon^{d/2}})$. Therefore, 

%\end{remark}

\begin{remark} (Unbounded domain)
Analysis of the variance error for the case where the support of $\pr$
is unbounded has proved to be difficult. In the unbounded case, it is
more appropriate to consider $\Teps$ and $\TepsN$ as linear operators
on $L^2(\pr)$.  
Following the same approach
as used in the proof of~\cref{thm:variance}, one would need to verify the
three conditions noted above.  However, for the unbounded case, we
could not verify the condition (i) that $\{\TepsN\}_{N=1}^\infty$ is
collectively compact on $L^2(\pr)$.  An alternative approach is to
follow the spectral method as outlined
in~\cite{koltchinskii2000random}.  In this approach, one examines the
convergence of empirical matrix $[k(X^i,X^j)]_{i,j=1}^N$ where
$k(\cdot,\cdot)$ is a given symmetric kernel. However, this approach
does not directly apply to the analysis of the empirical operator $\TepsN$.
This is because the form of the kernel $\kepsN(\cdot,\cdot)$, as it is
used in the definition of $\TepsN$, is not explicitly given.  It too
must be empirically approximated as a ratio whose convergence analysis
has proved to be rather challenging. 
\end{remark}
}

% In order to make these results applicable, we make the following
% assumption:

%\paragraph{Move to rest to Appendix.}

\subsection{Relationship to the constant gain approximation}

Although the convergence and error analysis pertains to the
$\epsilon\downarrow 0$ limit, an important property of
the  diffusion-map approximation is that the numerical procedure yields a
unique solution for arbitrary values of $\epsilon$ (see
\cref{prop:N}).  In fact, more can be said: one recovers the constant gain
approximation formula in the $\epsilon\rightarrow\infty$ limit.  

Before stating the result, it is useful to recall the three
formulae for the gain: 
\begin{romannum}
\item{\textbf{Exact formula:}} $\K = \nabla\phi$ is defined using the
  exact solution $\phi$;
\item{\textbf{Kernel formula:}} $\Keps$ is defined using the solution
  $\phieps$ to the (diffusion-map) approximation fixed-point equation:
\begin{equation}\label{eq:Keps-def}
\K_\epsilon(x) := \nabla_x \left[
\frac{1}{n_\epsilon(x)}\int  \keps(x,y)
(\phieps(y) + \epsilon  h(y) ) \rho_\epsilon(y) \ud y\right] .
\end{equation}

\item{\textbf{Empirical formula:}} $\KepsN$ is the empirical version of
  the kernel formula.  It was defined
  in~\cref{eq:empirical_formula_for_gain} using the solution ${\sf \Phi}$
of the finite-dimensional fixed-point
problem.  
\end{romannum}

\medskip

The proof of the following Proposition appears in the~\cref{Sec:prop:convergence_to_constant_gain}.         

\begin{proposition}\label{prop:convergence_to_constant_gain}
Consider the fixed-point problems~\cref{eq:kernel-approx-fixed-pt} and~\cref{eq:empirical-fixed-pt} in the limit
as $\epsilon\rightarrow\infty$.     
% \begin{align*}
% \lim_{\epsilon\rightarrow\infty} \;\;\K_{\epsilon} &= \int
% (h(x)-\hat{h}_\pr) \rho(x) \ud x = {\sf E}[K],\\
% \lim_{\epsilon\rightarrow\infty} \;\;\K_{\epsilon}^{(N)} &=
% \frac{1}{N}\sum_{i=1}^N(h(X^i)-\hat{h}_\pr^{(N)})X^i\quad \text{a.s.}
% \end{align*}
\begin{romannum}
\item The kernel formula of the gain is given by
\[
\lim_{\epsilon\rightarrow\infty} \K_{\epsilon} = \int (h(x)-\hat{h}_\pr)x \rho(x) \ud x. % = {\sf E}_{\rho}[\K]
\]
\item For any finite $N$, the empirical formula of the gain is given by
\[
\lim_{\epsilon\rightarrow\infty} \K_{\epsilon}^{(N)} =
\frac{1}{N}\sum_{i=1}^N(h(X^i)-\hat{h}_\pr^{(N)})X^i \quad \text{a.s.}
\]
\end{romannum}
%\qed
\label{prop:eps-infty}
\end{proposition}

This result serves to highlight the connection between the FPF and the
EnKF: With the diffusion map approximation of the gain, the FPF approaches EnKF in the limit of large $\epsilon$.  The parameter
$\epsilon$ can then be regarded as the tuning parameter to ``improve''
the gain.  Of course, for any finite value of $N$, this can only be
done up to a point -- where variance becomes dominant (see \cref{fig:kernel-approx}).

\section{Numerics}
\label{sec:numerics}

%In this section, results are presented to numerically illustrate the
%effect on the error of the following parameters: (i) the kernel
%bandwidth $\epsilon$, (ii) the number of particle $N$, (iii) and the
%dimension $d$. 

\subsection{Example - the vector case} 
\label{ex:vector}

A vector generalization of
the scalar example in \cref{ex:scalar} is obtained by considering the 
following form of the probability density function in $d$-dimensions:
\begin{equation*}
\pr(x) = \rho_{b}(x_1) \prod_{n=2}^d\rho_g(x_n),\quad \text{for}\quad x=(x_1,x_2,\ldots,x_d) \in \Re^d,
\end{equation*}
where $\rho_b$ is the bimodal distribution
$\frac{1}{2}\mathcal{N}(-1,\sigma^2)
+\frac{1}{2}\mathcal{N}(+1,\sigma^2)$ introduced in \cref{ex:scalar}, and $\rho_g$ is the Gaussian distribution~$\mathcal{N}(0,\sigma^2)$. 
%In other words, for the random variable $X=(X_1,\ldots,X_d) \sim \rho$, the first component $X_1$ has bimodal distribution and other components are independent and Gaussians. 
Also suppose the function $h(x)=x_1$. The simple example is
illustrative of realistic application scenarios where the density has
non-Gaussian features along certain (not necessarily apriori known)
low-dimensional subspace.  The  directions orthogonal to this
subspace are modelled here as Gaussian noise.  

For this problem,  the exact gain function is easily obtained as
\[
{\sf K}_{\text{exact}}(x) = ({\sf K}_{\text{exact}}(x_1),0,\ldots,0),
\] 
where the function ${\sf K}_{\text{exact}}(x_1)$ is given by the
formula~\cref{eq:scalar} in \cref{ex:scalar}.  The exact solution is
used to compute error properties as dimension increases.  

% This non-Gaussian problem admits an explicitly known solution but
% nevertheless exhibits non-trivial error properties particularly as
% dimension increases. 

The diffusion-map algorithm (\cref{alg:kernel}) is simulated to
approximate the gain function for this problem. The number of iterations in~\cref{alg:kernel} set to  $L=10^3$. For each particle
$X^i=(X^i_1,\ldots,X^i_d)$, the first coordinate $X^i_1 \overset{\text{i.i.d}}{\sim}
\frac{1}{2}\mathcal{N}(-1,\sigma^2)
+\frac{1}{2}\mathcal{N}(+1,\sigma^2)$ and other the coordinates $X^i_n
\overset{\text{i.i.d}}{\sim} \mathcal{N}(0,\sigma^2)$ for $n=2,\hdots,d$.  The constant gain approximation is evaluated according to the explicit integral formula~\eqref{eq:const-gain-approx}.  

 \cref{fig:numerics-all} depicts the m.s.e~\cref{eq:MC_error} computed from running $M=100$ simulations.
 A summary of these results is as follows:

\begin{enumerate}
	\item \cref{fig:numerics-all}-(a) depicts the error as
          a function of the parameters $\epsilon$ and $d$ for a fixed
          number of particles $N=1000$.  Also depicted is the error
          with the constant gain approximation.  The constant gain
          error serves here as baseline.

          For large values of $\epsilon$, the bias error is dominant, and as $\epsilon \to \infty$ the  error
          asymptotes to 
          the error for the constant-gain approximation.  This is because (see
          \cref{prop:convergence_to_constant_gain}) the diffusion map gain
          approaches the constant gain as $\epsilon\rightarrow\infty$.  \todo{
          For small values of $\epsilon$, the variance error dominates. According to~\cref{rem:variance-error}, the upper-bound for m.s.e is expected te be of the order $O(\frac{1}{N \epsilon^{d+2}})$. 
          However, the numerical error in~\cref{fig:numerics-all}-(a)
          is observed to be $O(\frac{1}{N
            \epsilon^{0.16d+0.3}})$. Therefore, the upper-bound
          in~\cref{rem:variance-error} is not tight for this specific
          problem.} 
% This
          % explains the fact that the error asymptotes to the constant
          % gain error for large $\epsilon$. 

% The other aspect to note
%          is the bias-variance trade-off first illustrated in
%          \cref{fig:kernel-approx} for the scalar case.  As the
%          dimension increases, the error due to the variance becomes
%          dominant at relatively larger values of $\epsilon$.  
	\item \cref{fig:numerics-all}-(b) depicts the
          bias-variance trade-off as a
          function of number of particles $N$ for the fixed $d=1$.  It is not a surprise
          that the error gets better, for all choices of $\epsilon$, as
          the number of particles increase.  However, the optimal
          value of $\epsilon$ -- at which the error is the smallest --
          is relatively insensitive to changes in $N$. 
%          A common choice of the kernel-bandwidth in practice is given proportional to $\epsilon  
%           Figure~BLAH
%          depicts this critical value -- determined using simulations
%          -- as a function a $N$ and $d$.       
	\item \cref{fig:numerics-all}-(c) depicts the error as
          function of $N$ for different values of $\epsilon$.  The
          dimension $d=1$ is fixed.  The
          error goes down as $O(\frac{1}{N})$ and asymptotes to
          the $O(\epsilon)$ bias.  The $O(\frac{1}{N})$ is due to the variance error obtained in~\cref{prop:TepsN-convergence} and $O(\epsilon)$ bias error is
          consistent with the conclusion of the \cref{thm:bias}.   
	\item \cref{fig:numerics-all}-(d)  depicts the run time
          comparison between the diffusion-map algorithm and the constant gain
          algorithm.  The scaling for the diffusion-map  algorithm is $O(N^2)$
          which is significantly more expensive than the $O(N)$
          scaling of the constant gain approximation.  
\end{enumerate}

\begin{figure}[t]
	\centering
	\begin{tabular}{cc}
		\subfloat[]{
			\includegraphics[width=0.45\columnwidth]{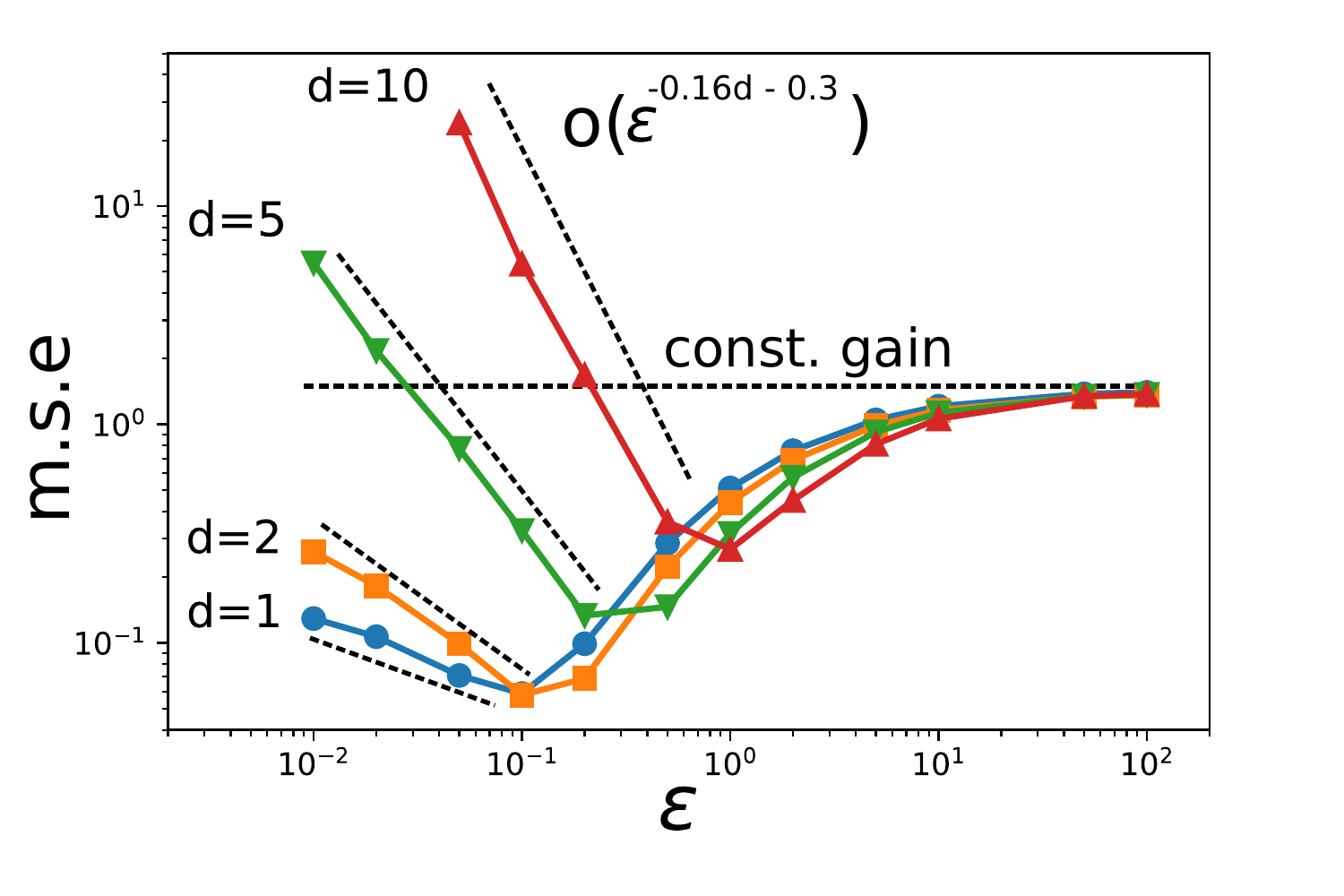}
		}&
		\subfloat[]{
			\includegraphics[width=0.45\columnwidth]{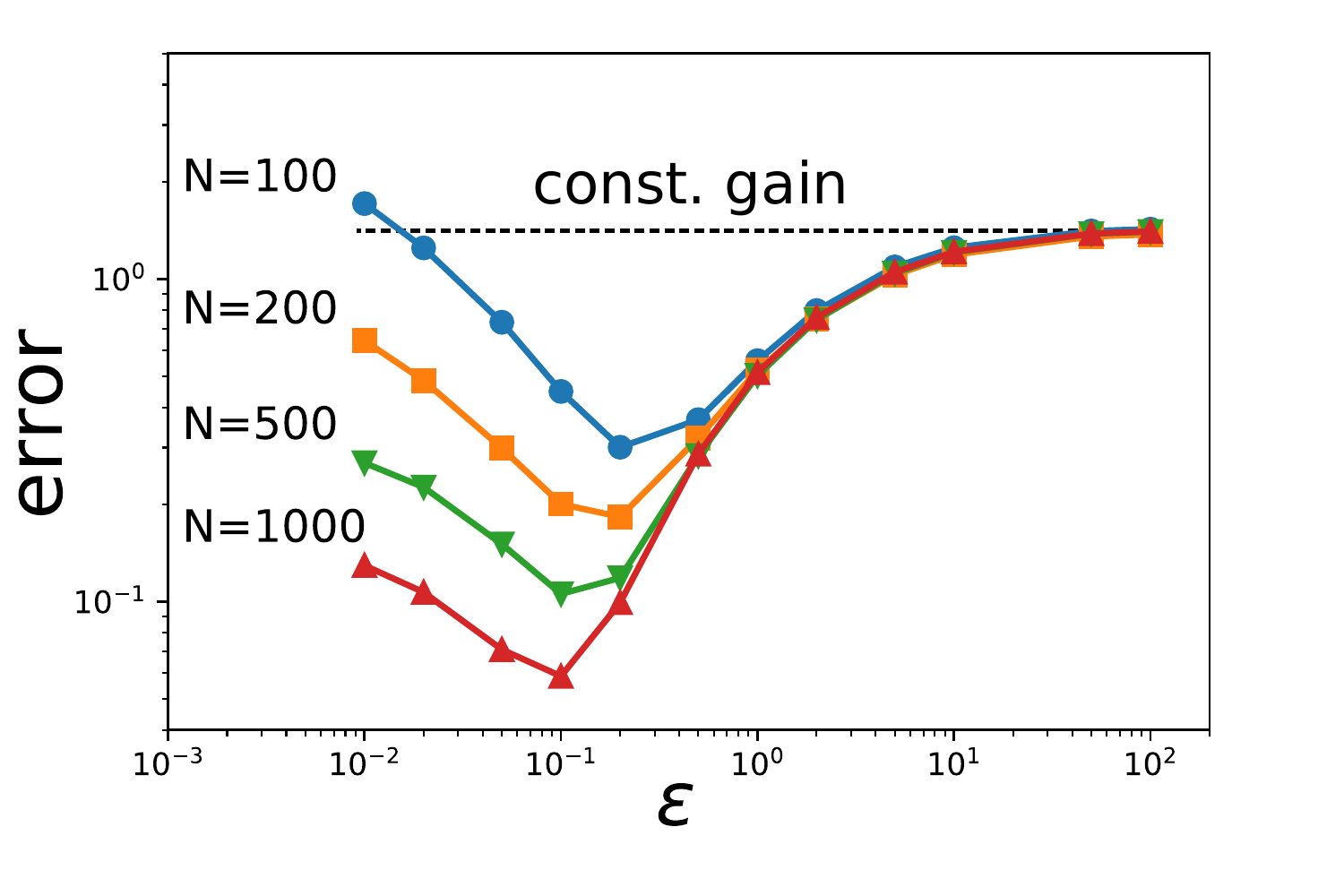}
		}\\
		\subfloat[]{
			\includegraphics[width=0.45\columnwidth]{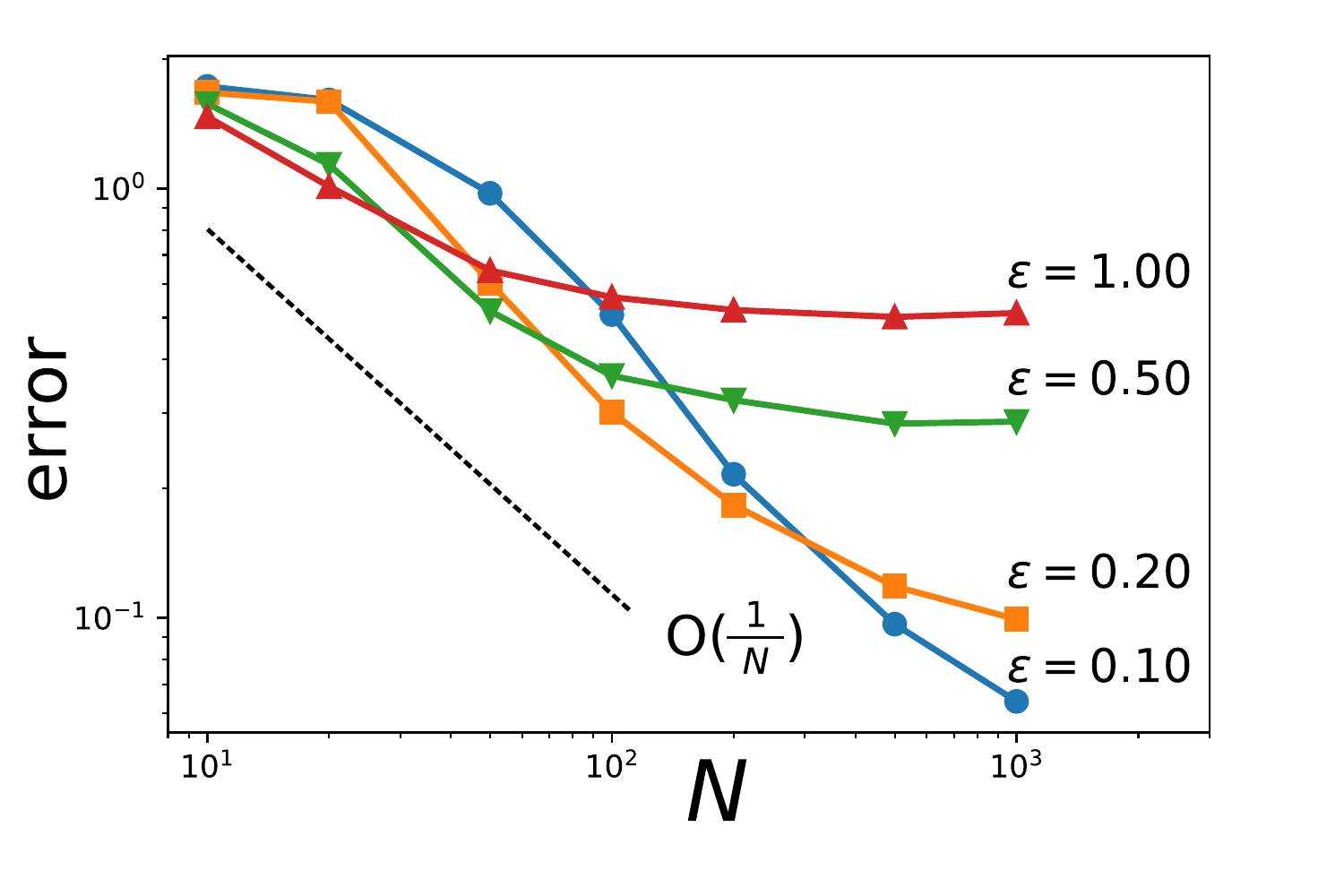}
		}&
		\subfloat[]{
			\includegraphics[width=0.45\columnwidth]{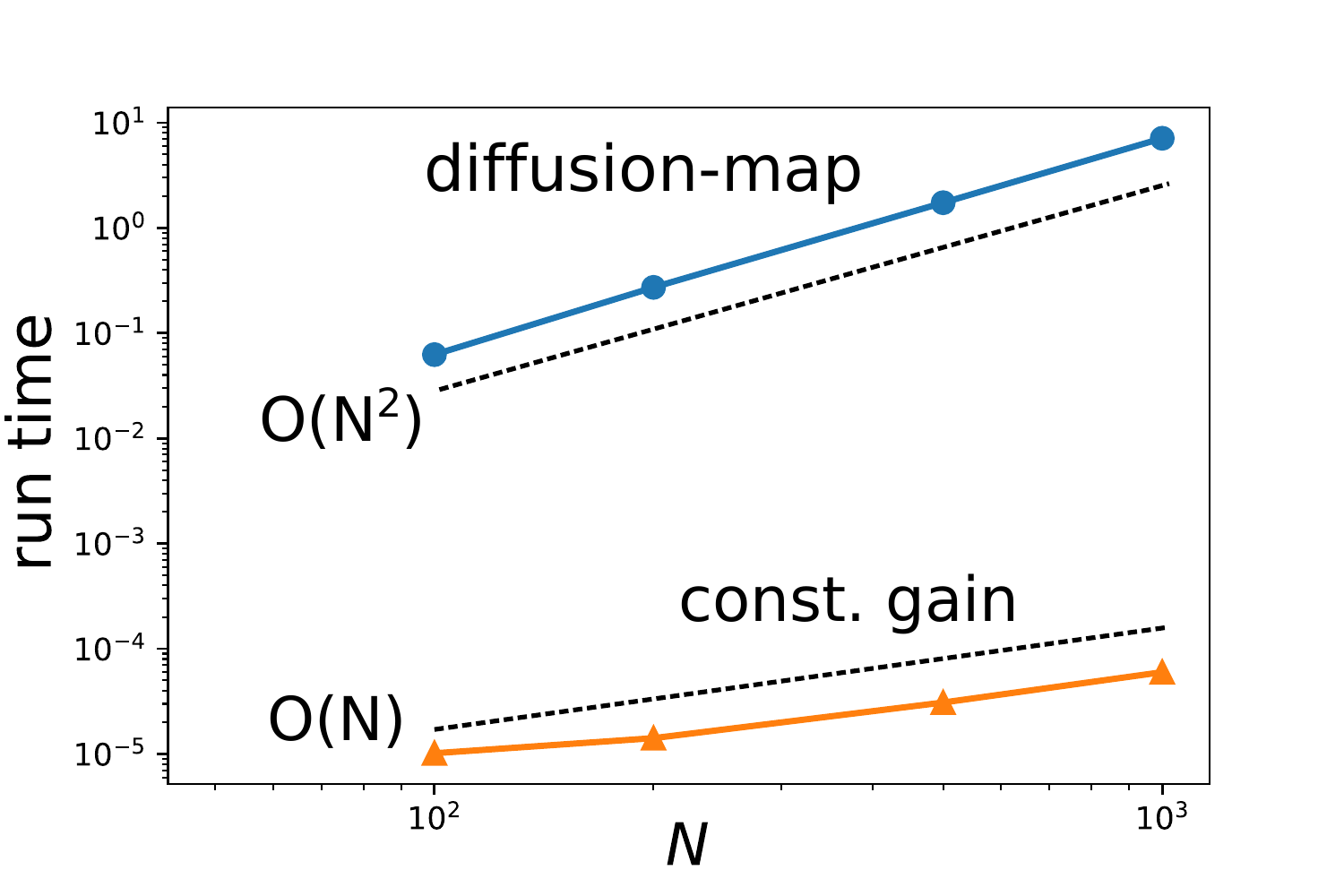}
		}		
	\end{tabular}
	\caption{Simulation results for the diffusion-map algorithm
          for the vector bimodal example: (a) Gain function
          approximation error as a function of
          $\epsilon$ for $d\in\{1,2,5,10\}$. (b) Error as a function of $\epsilon$ for  $N\in\{100,200,500,1000\}$. (c) Error as a function of $N$ for $\epsilon \in \{0.1,0.2,0.5,1.0\}$; (d) Comparison of the run-time  
	%the running time as a function of $N$ for kernel algorithm and constant gain approximation. Kernel algorithm for the bimodal example problem with $N=200$ particles.
	}
	\label{fig:numerics-all}
\end{figure} 

\todo{
\begin{remark}[Selection of $\epsilon$] The numerical results in
  \cref{fig:numerics-all} suggest that there is an optimal value
  of $\epsilon$ such that the error is smallest. Given the fact that
  the constant gain approximation results in the limit as
  $\epsilon\to\infty$, an optimal choice of $\epsilon$ may be possible
  more generally.  At the optimal value, one optimally trades-off
  the errors due to variance and bias.  The difficulty, of course, is that
  the formula for this optimal choice is not known and may not even be
  possible in general settings.  Instead, in the literature involving kernel methods, a popular heuristic is to set 
	$\epsilon = \frac{4\text{(med)}^2}{\log(N)}$ where
        $\text{(med)}$ is the median value of all pairwise distances
        $\{|X^i-X^j|\}_{i\neq j}$~\cite{chaudhuri2017mean}. The justification is that, with such a choice, 
	%is that the summation $\sum_{i=1}^N \geps(X^i,X^j) \approx N e^{-\frac{\text{med}^2}{4\epsilon}} = 1 $, and as a result, 
	the matrix
        $[\geps(X^i,X^j)]_{i,j=1}^N$ is not close to the identity
        matrix (which represents the degenerate case).  

	\label{rem:epsilon}
\end{remark}
}
%\begin{remark}[Selection of Num\_Iters and $\phi_\text{prev}$]
%\end{remark}
\todo{
\begin{remark}It is worthwhile to also examine the limit as $\epsilon
  \to 0$ while $N$ is fixed at a finite value. In this limit, the
  Markov matrix $\Ten$ converges to the identity matrix. As a result,
  the solution $\phivec$ to the fixed-point
  problem~\eqref{eq:finite-N-fixed-pt} is unbounded. However, in
  practice, value of $\phivec$ is large but finite, because the
  equation~\eqref{eq:finite-N-fixed-pt} is solved in an iterative
  manner with finite number of iterations. With a finite value of
  $\phivec$ and $\Ten$ equal to identity, the gain function given by
  the formula~\eqref{eq:gain-linear-form} is zero. Consequently, the
  feedback correction for each particle is zero. 
\end{remark}
}
\subsection{Filtering example}\label{sec:filtering-example}
Consider the following filtering problem:
\begin{align*}
\ud X_t  &= 0,\quad X_0\sim p_0,\\
\ud Z_t &= h(X_t)\ud t + \sigma_w \ud W_t,
\end{align*}
where $X_t\in \Re$, $Z_t \in \Re$, $\sigma_W>0$, and $\{W_t\}$ is standard Brownian motion, independent of $X_t$. The prior distribution $p_0$ is Gaussian $\mathcal{N}(0,1)$ and the observation function $h(x)=|x|$. For the static filtering problem, the posterior distribution is explicitly given by:
\begin{equation*}
p^\star_t(x) = \text{(const.)}p_0(x)e^{\frac{1}{\sigma_w^2}(h(x)Z_t - \frac{t}{2}h^2(x))}.
\end{equation*} 

Three filtering algorithms are implemented for this problem: (i) the FPF algorithm with the diffusion-map
gain approximation; (ii)  the FPF algorithm with the the constant gain
approximation (similar to EnKF); (iii) a sequential importance
resampling (SIR) particle filter~\cite{doucet09}. The simulation parameters are as follows: The measurement
noise $\sigma_w= 0.1$. The simulation is carried out for $T=500$
time-steps with step-size $\Delta t = 0.001$. Both the algorithms use
$N=200$ particles with identical initialization. For the diffusion-map
approximation, the kernel bandwidth was set to $\epsilon = 0.1$, and number of iterations in~\ref{alg:kernel} is set to $L = 100$.

The numerical results are depicted in
Figure~\ref{fig:filtering-example-1}. The distribution of the particles along with the exact posterior distribution
are depicted in Figure~\ref{fig:filtering-example-1}-(a). It is
observed that the FPF algorithm with the diffusion map approximation
provides a more accurate approximation of the posterior distribution.
In contrast, the constant-gain approximation fails to reproduce the
bimodal nature of the posterior distribution.

A quantitative estimate of the performance is provided in terms of a
mean squared error (m.s.e.). in estimating the conditional expectation
of the function $\psi(x)= x\mathds{1}_{x\leq 0}$.  A Monte Carlo
estimate of the m.s.e. is depicted in
Figure~\ref{fig:filtering-example-1}-(b) with $M=100$ runs.  At time $t$, it is
calculated according to
\begin{equation*}
\text{m.s.e.}_t=\frac{1}{M}\sum_{m=1}^M \left(\frac{1}{N}\sum_{i=1}^N \psi(X^{m,i}_t) - \int \psi(x)p^\star_t(x)\ud x\right)^2.
\end{equation*}

At time $t=0$, the empirical distribution of the particles is an
accurate approximation of the prior distribution, because the
particles are sampled i.i.d. from the prior distribution. 
Therefore, the m.s.e at $t=0$ is small.  As time progress, the
difference between the empirical distribution and the exact posterior
becomes larger because the filter update is not exact.  For FPF, as the
time-step $\Delta
t$ is small, the main
source of the m.s.e. error is due to the error in the gain function
approximation.  Therefore, the diffusion map FPF with its more
accurate approximation of the gain yields better m.s.e., compared to
the EnKF using the constant gain approximation.  The particle filter, like FPF with diffusion map approximation, is able to capture the bi-modal distribution. However, due to  the stochastic noise, introduced from the resampling step, it  admits larger error.

\begin{figure}[t]
\centering
\begin{tabular}{cl}
\subfloat[]{\includegraphics[width = 0.6\columnwidth]{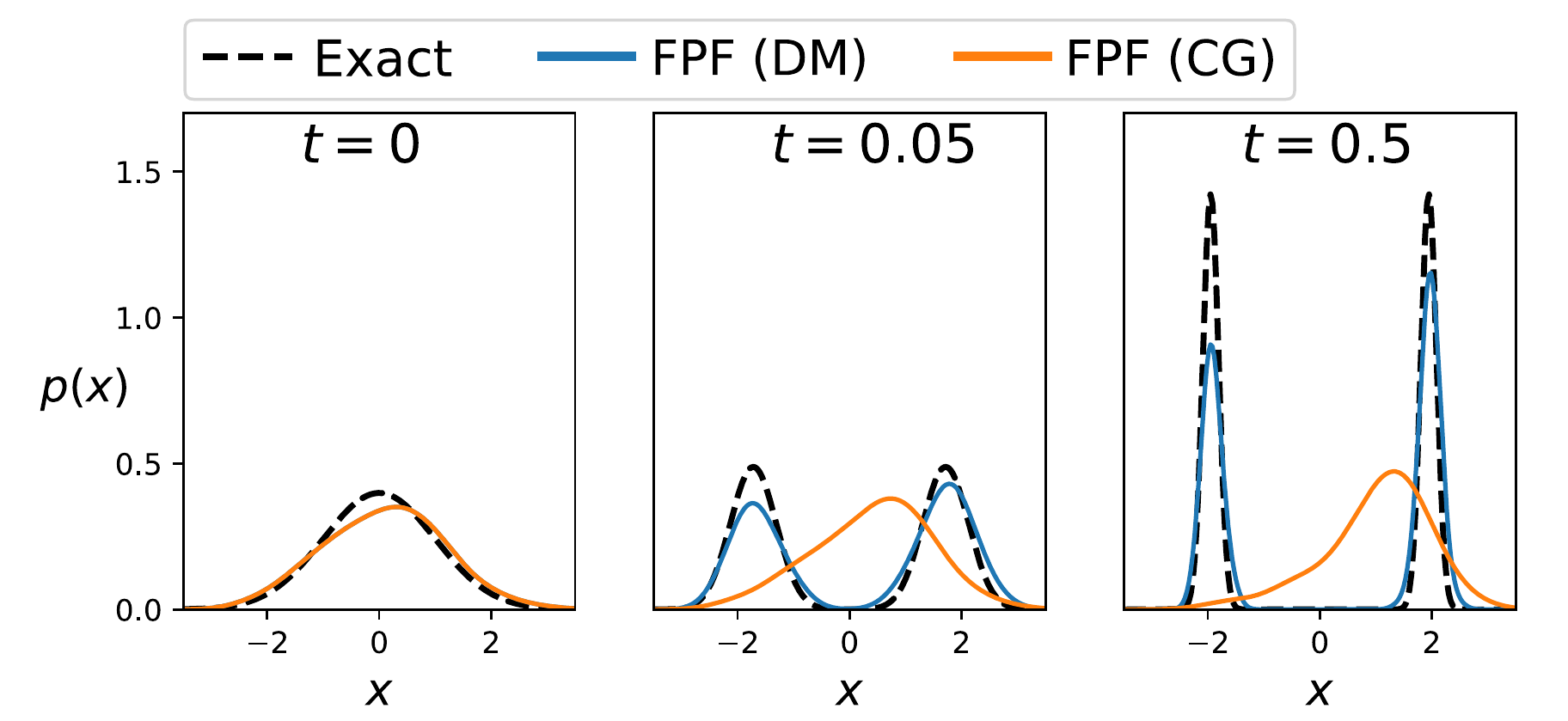}}
&
%\subfloat[]{\includegraphics[width = 0.33\columnwidth]{figures/filtering-example-1-mse.pdf}}
\subfloat[]{\includegraphics[width = 0.33\columnwidth]{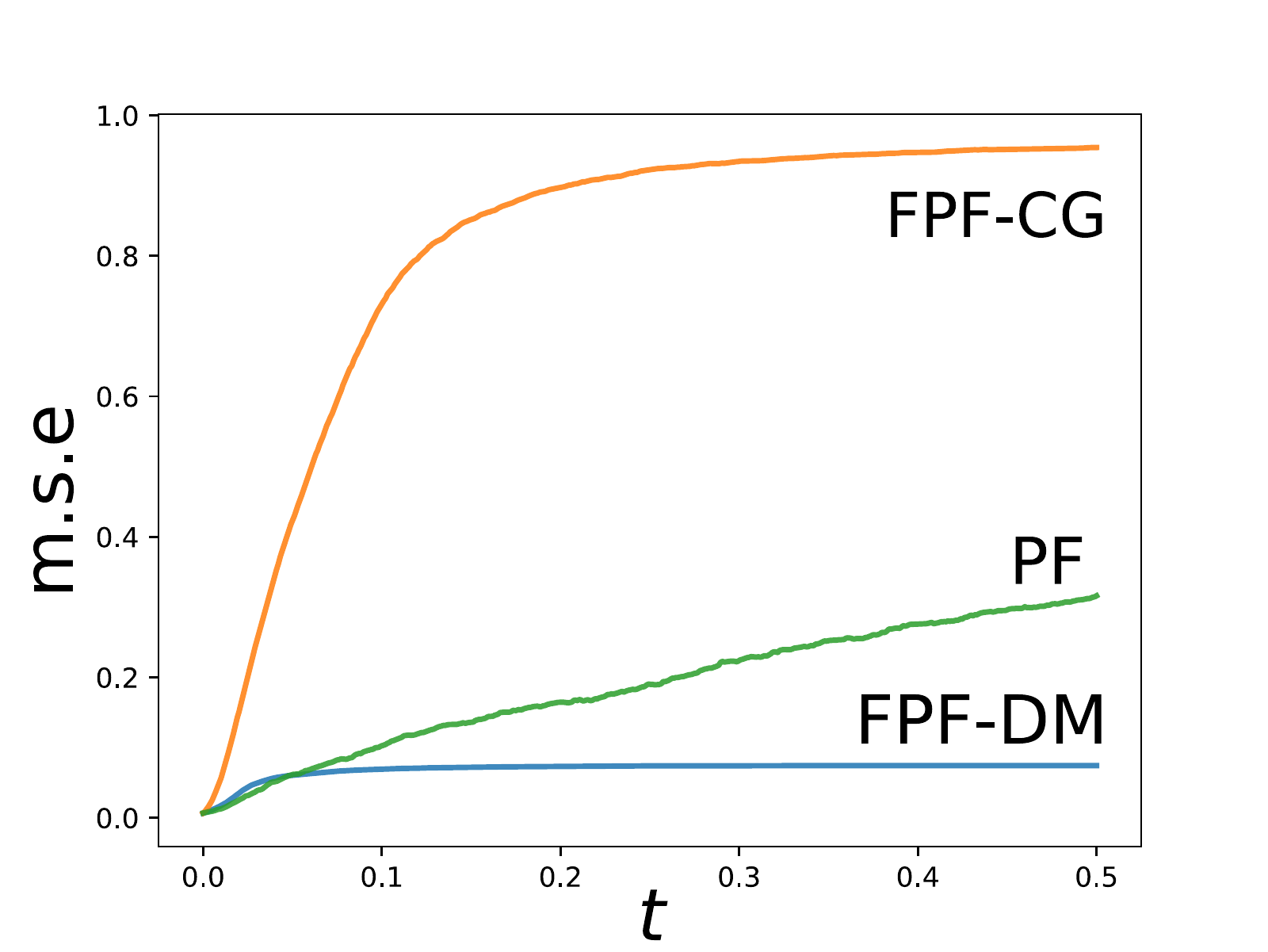}}
\end{tabular}
\caption{Simulation results for the FPF algorithm for the filtering example: (a) The distribution of the
  particles obtained using the diffusion-map approximation and the
  constant gain approximation as compared to the exact distribution (dashed line); (b) Plot
  of the mean squared error 
  in estimating the conditional expectation of the function
  $\psi(x)=x\mathds{1}_{x<0}$.}
\label{fig:filtering-example-1}
\end{figure} 

\todo{
\subsection{Benes filter}
Consider the following filtering problem:
\begin{align*}
\ud X_t &= \mu \sigma_B \tanh(\frac{\mu}{\sigma_B}X_t)\ud t + \sigma_B \ud B_t,\quad X_0=x_0,\\
\ud Z_t &= (h_1 X_t + h_1h_2)\ud t + \ud W_t,
\end{align*}
where $\{X_t\},\{Z_t\} \in \Re$ are one-dimensional stochastic
processes, $\{B_t\}$ and $\{W_t\}$ are one-dimensional, independent,
Brownian motions, $x_0$ is a known initial condition, and the
constants $\mu,\sigma_B,h_1,h_2\in \Re$.  This filtering problem has a
finite-dimensional analytical solution given by a mixture of two
Gaussians~\cite{bain2009}:
 \[
w_t \mathcal{N}(a_t-b_t,\sigma_t^2) + (1-w_t)
\mathcal{N}(a_t+b_t,\sigma_t^2),
\] 
where
\begin{align*}
a_t&= \sigma_B\Psi_t\tanh(h_1\sigma_Bt) + \frac{h_2 + x_0}{\cosh(h_1\sigma_Bt)} -h_2,\quad b_t = \frac{\mu}{h_1}\tanh(h_1\sigma_Bt),\\
\sigma_t^2 &= \frac{\sigma_B}{h_1}\tanh(h_1\sigma_Bt),\quad \Psi_t = \int_0^t \frac{\sinh(h_1\sigma_B s)}{\sinh(h_1\sigma_B t)}\ud Z_s,\quad w_t=\frac{1}{1+e^{\frac{2a_tb_t}{\sigma_B}\coth(h_1\sigma_B t)}}.
\end{align*}

The three filtering algorithms, as in the previous example, are
also implemented and evaluated for this problem.  The simulation
parameters are chosen according to the values used in~\cite{crisan2013kusuoka}:
$\mu=0.5$, $h_1= 0.4$, $h_2=0$, $\sigma_B = 0.8$, $x_0=1.0$.  The
simulations are carried out over the time
horizon $T=10$.  The stochastic integrals are approximated with a
first-order Euler scheme using the discretization step-size $\Delta t
= 0.01$.  For FPF with DM gain approximation, the kernel
bandwidth $\epsilon$ is selected according to the rule described in
\cref{rem:epsilon} and number of iterations in~\cref{alg:kernel} is $L=100$. 

The numerical results are depicted in \cref{fig:Benes-filter}. It is
observed that the FPF with DM and constant gain
approximations admit almost the same accuracy. The reason is that the
exact bimodal posterior distribution quickly converges to an almost
uni-modal distribution.  This is because the weight of one of the
mixture modes converges to zero. 
% and the difference between the means of the two
% models  is small compared to their variance. 
The accuracy of the SIR 
particle filter is poor because of the stochastic noise introduced
from  resampling step. 
}
\begin{figure}[t]
	\centering
	\begin{tabular}{cl}
		\subfloat[]{\includegraphics[width = 0.6\columnwidth]{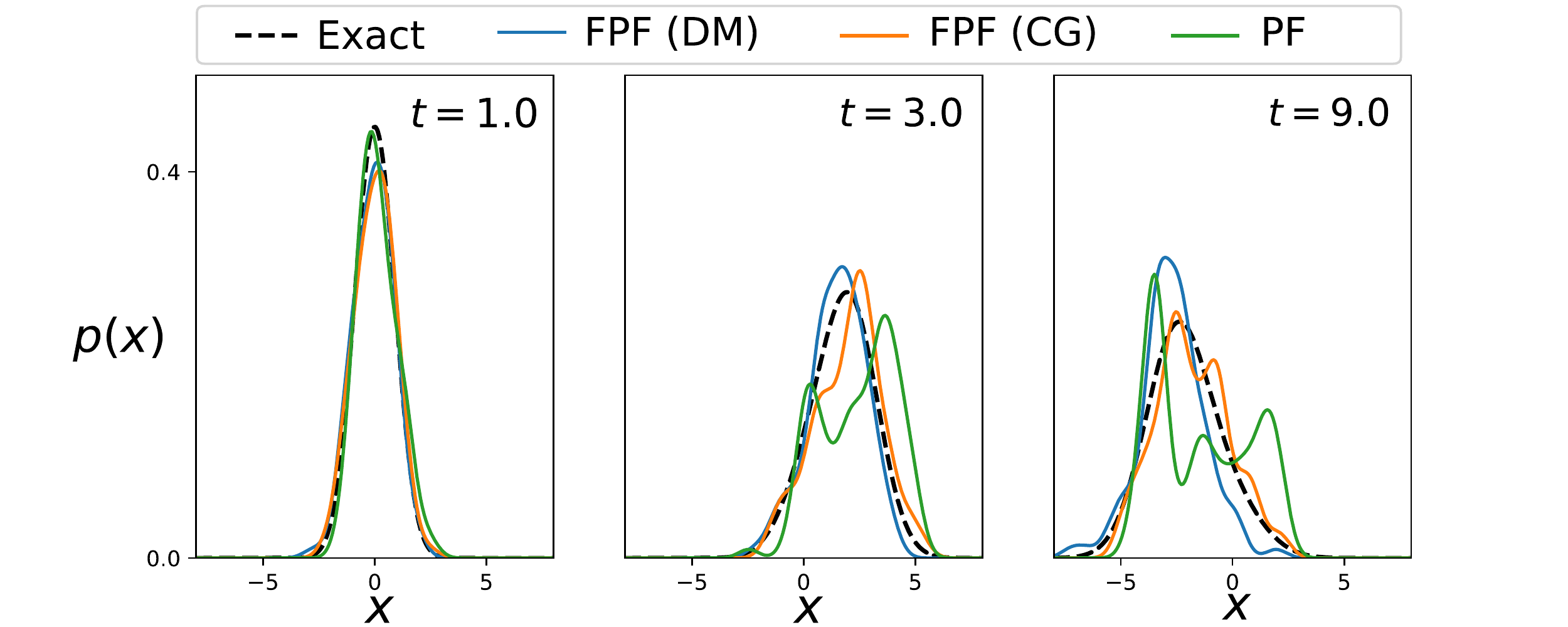}}
		&
		\subfloat[]{\includegraphics[width = 0.33\columnwidth]{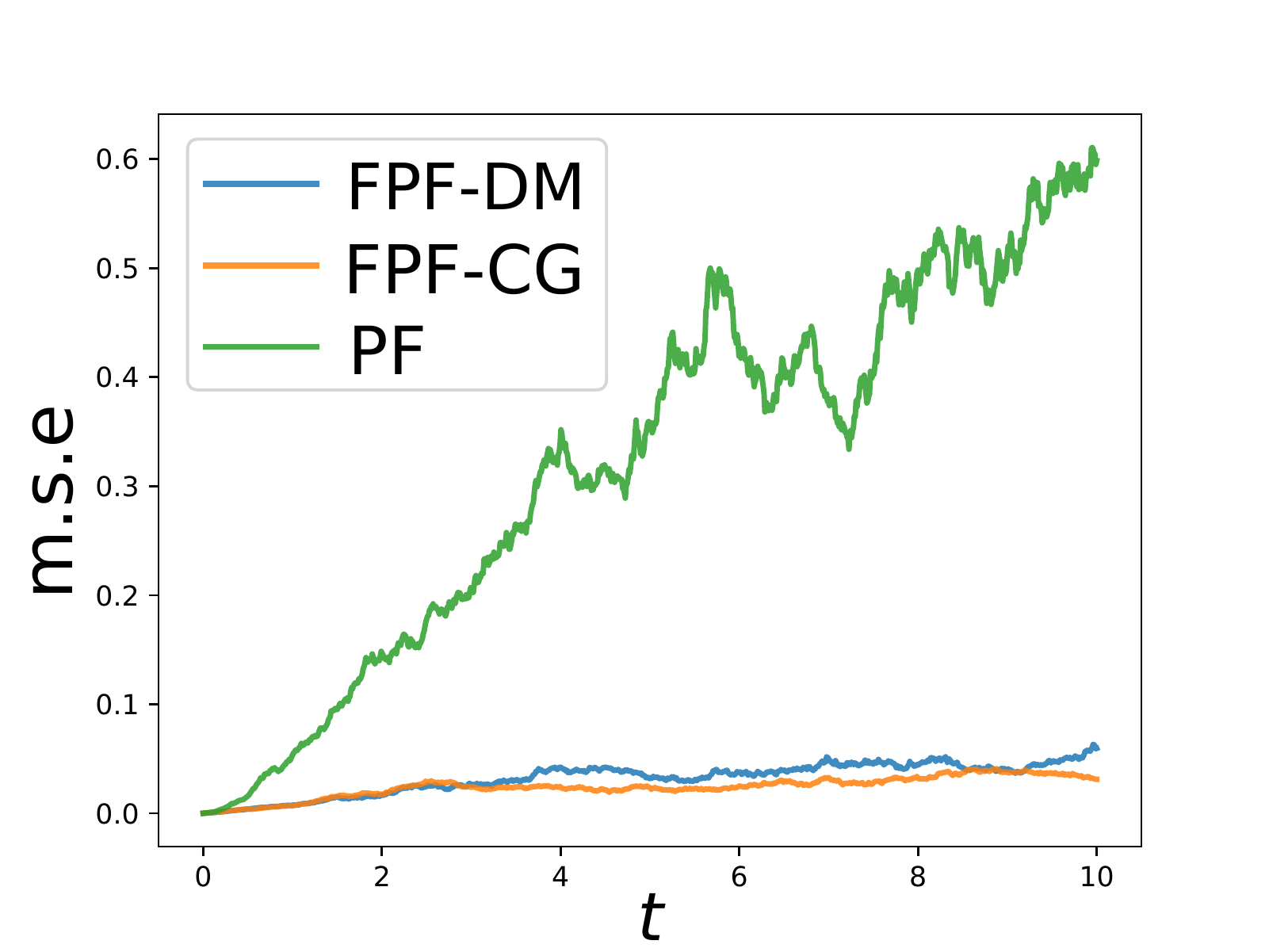}}
	\end{tabular}
	\caption{Simulation results for the FPF algorithm for the Benes filter example: (a) The distribution of the
		particles obtained using the diffusion-map approximation and the
		constant gain approximation as compared to the exact distribution (dashed line); (b) Plot
		of the mean squared error 
		in estimating the conditional expectation of the function
		$\psi(x)=x$.}
	\label{fig:Benes-filter}
\end{figure} 
% \label{sec:conc}

\todo{
\section{Conclusions and Directions for Future Work}

In this paper, the diffusion map (DM) algorithm was presented for the
problem of gain function approximation in the FPF.  It was
shown that the approximation error converges to zero in the limit as
the number of particles $N \to \infty$ and the kernel bandwidth
parameter $\epsilon\to0$ (\cref{thm:bias,thm:variance}).  In the limit
as $\epsilon \to \infty$, the gain obtained using the DM algorithm was
shown to converge to the constant gain approximation
(\cref{prop:convergence_to_constant_gain}). Consequently, in this limit, the FPF using the DM algorithm reduces to
an EnKF.  This is an important property because it suggests a path to
improve the performance of an EnKF algorithm by choosing an appropriate (finite) value of the parameter
$\epsilon$.  The bounds, scalings and the numerical experiments
described in this paper provide guidance on how to choose the
parameter $\epsilon$ for large but finite $N$.       
Some directions for future work are as follows:
\begin{enumerate}
	\item Relaxing the assumptions: The analysis is based on
          Assumption A1 which is restrictive because it does not
          include the mixture of Gaussians. Relaxing this assumption, 
          possibly as suggested in \cref{rem:assumption}, is one possible avenue
          of future work. 
	\item Error analysis for the FPF: The error analysis in this
          paper concerns primarily the convergence of function $\phiepsN$
          to the exact solution $\phi$.  Extending these results to include 
          the convergence analysis of the gain $\KepsN= \nabla
          \phiepsN$ to the exact gain $\K = \nabla \phi$ is important
          for the complete error analysis of the FPF with finitely
          many particles. 
	% \item Error analysis of the FPF: Finally, the ultimate objective is consider the FPF algorithm with the diffusion map-based algorithm for gain approximation, and then analyze the effect of the gain approximation error on the distribution of particles, and the accuracy of the FPF algorithm in approximating the posterior distribution. 
	\end{enumerate}

}
\bibliographystyle{siamplain}
\bibliography{../bibfiles/SIAM-Gain}
%\bibliography{bibfiles/fpfbib,bibfiles/ref,bibfiles/fpfbib2,bibfiles/Optimization,bibfiles/meanfield,bibfiles/meanfield_v2} 

\newpage
\appendix

\section{Exact semigroup and and its diffusion map approximation for the Gaussian case}
\label{apdx:Laplacian-Gaussian}
In this section, we provide explicit formulae for the exact semigroup $P_t$ and its
diffusion map approximation $T_\epsilon$, for the special case when the
density $\rho$ is a Gaussian $\mathcal{N}(m,\Sigma)$. 
For the Gaussian case, the semigroup is the
Ornstein-Uhlenbeck semigroup~\cite[Sec. 2.7.1]{bakry2013} an its
spectral representation is obtained in terms of the Hermite polynomials.  For
notational ease, after an appropriate change of coordinates, we
assume $m=0$ and $\Sigma =\text{diag}(\sigma_1^2,\ldots,\sigma_d^2)$ where
$\sigma_1^2\geq\sigma_2^2\geq \ldots \geq \sigma_d^2>0$ are
ordered eigenvalues of $\Sigma$.

%
%\newP{Exact Semigroup:} For the Gaussian case, the semigroup is the
%Ornstein-Uhlenbeck semigroup~\cite[Sec. 2.7.1]{bakry2013}.  
%%For the Gaussian case, the weighted Laplacian $\Delta_\pr$ is given by:
%%\begin{equation*}
%%\Delta_\pr f(x) = \Delta f(x) - x^\top \Sigma^{-1} \nabla f(x)
%%\end{equation*}
%That is, the Markov process $S_t$ is described by the
%Ornstein-Uhlenbeck sde:
%\begin{equation*}
%\ud S_t = -\Sigma^{-1}S_t\ud t + \sqrt{2}\ud B_t,\quad S_0=x
%\end{equation*} 
%whose explicit solution is given by
%\[
%S_t = e^{-t\Sigma^{-1}}x + \int_0^te^{-\Sigma^{-1}s}\ud B_s
%\]
%Therefore, using the probabilistic representation
%$P_tf(x)=\Expect[f(S_t)]$, we have the following integral
%representation of the operator $P_t$:
%\begin{equation}
%P_t f(x) = \int_{\Re^d} \prod_{j=1}^d \frac{1}{(2\pi \sigma_j^2(1-e^{-2t\sigma_j^{-2}}))^{1/2}}
%\exp(-\frac{|y_j-e^{-t\sigma_j^{-2}}x_j|^2}{2\sigma_j^2(1-e^{-2t\sigma_j^{-2}})})f(y)\ud y 
%\label{eq:exact-semigroup-Gaussian}
%\end{equation}

%One can also write down an explicit form of the spectral
%representation because the eigenfunctions are known in
%terms of the Hermite polynomials.  

\begin{definition}
The {\em Hermite polynomials} are recursively defined as
\begin{equation*}
\hslash_{n+1}(x) = x \, \hslash_{n}(x) - \hslash'_n(x),\quad \hslash_0(x) = 1,
\end{equation*} 
where the prime $'$ denotes the derivative.
%\qed
\end{definition} 

%For the exact semigroup $P_t$, the eigenvalues and eigenfunctions are
%given by
%\begin{equation*}
%\lambda_n = \prod_{j=1}^de^{-t\frac{n_j}{\sigma_j^2}}, \quad e_n (x) = \prod_{j=1}^d \hslash_{n_j}(\frac{x_j}{\sigma_j}),\quad \text{for}\quad n=(n_1,\ldots,n_d) \in \mathbb{Z}_+^d, 
%\end{equation*} 
%The spectral gap $\lambda_1 = \frac{1}{\sigma_1^2}>0$ and the operator norm $\|P_t\|_{L^2_0(\pr)}=e^{-\frac{t}{\sigma_1^2}}<1$. 

%\newP{Diffusion map approximation:} The explicit form of the kernel
%approximation $T_\epsilon$ (see~\cref{eq:Teps-definition} for
%the definition) of the exact semigroup $P_{\epsilon}$ is presented in
%the following Proposition.

% $\Teps$
\begin{proposition}\label{prop:Teps-Gaussian}
Suppose the density $\rho$ is
  Gaussian ${\cal N}(0,\Sigma)$ with the variance $\Sigma =
  \text{diag}(\sigma_1^2,\ldots,\sigma_d^2)$ and 
$\sigma_1^2\geq\sigma_2^2\geq \ldots \geq \sigma_d^2>0$.  Then 
\begin{romannum}
\item The exact semigroup $P_t$ and the diffusion map $T_\epsilon$ admit the following
  integral representations:
\begin{align}
P_t f(x) &= \int_{\Re^d} \prod_{j=1}^d \frac{1}{(2\pi \sigma_j^2(1-e^{-2t\sigma_j^{-2}}))^{1/2}}
e^{-\frac{|y_j-e^{-t\sigma_j^{-2}}x_j|^2}{2\sigma_j^2(1-e^{-2t\sigma_j^{-2}})}}f(y)\ud y, 
\label{eq:exact-semigroup-Gaussian}\\
\Teps f(x) &= \int_{\Re^d}\prod_{j=1}^d\frac{1}{(4\pi\epsilon(1-\delta_j))^{1/2}}
e^{-\frac{|y_j-(1-\delta_j)x_j|^2}{4\epsilon(1-\delta_j)}} f(y)\ud y,
\label{eq:approximate-semigroup-Gaussian}
\end{align} 
where $\delta_j:=\epsilon\frac{\sigma_j^2+4\epsilon}{\sigma_j^4+3\sigma_j^2\epsilon + 4\epsilon^2}$ for $j=1,\ldots,d$. 
\item The operators $P_t$ and $T_\epsilon$ each have a unique invariant
  Gaussian density given by $\mathcal{N}(0,\Sigma)$ and 
  $\mathcal{N}(0,\Sigma_\epsilon)$, respectively, where $\Sigma_\epsilon =
  \text{diag}(\sigma_{\epsilon,1}^2,\ldots,\sigma_{\epsilon,d}^2)$
  with $\sigma_{\epsilon,j}^2=\frac{2\epsilon(1-\delta_j)}{\delta_j(2-\delta_j)}$
  for $j=1,\ldots,d$. 
\item The eigenvalues and the associated eigenfunctions are as follows:
\begin{align*}
\text{Spectrum of the semigroup}\;\;P_t:&\quad \lambda_n = \prod_{j=1}^de^{-t\frac{n_j}{\sigma_j^2}}, \quad e_n (x) = \prod_{j=1}^d \hslash_{n_j}(\frac{x_j}{\sigma_j}), \\
 \text{Spectrum of the diffusion map}\;\;T_\epsilon:&\quad \lambda_n = \prod_{j=1}^d(1-\delta_j)^{n_j}, \quad e_n (x) =
 \prod_{j=1}^d \hslash_{n_j}(\frac{x_j}{\sigma_{\epsilon,j}}),
\end{align*} 
for $n=(n_1,\ldots,n_d) \in \mathbb{Z}_+^d$. 
\item The operator norm $\|P_t\|_{L^2(\pr)}= e^{-\frac{t}{\sigma_1^2}}$ and $\|\Teps\|_{L^2(\preps)}= 1-\delta_1$.
\end{romannum}
\end{proposition}
\begin{proof}Omitted. See~\cite[Prop. 1]{Amir_ACC17}.
\end{proof}

\section{Proof of~\cref{prop:semigroup-equivalent}}
\label{apdx:semigoup-equivalent}

Based on the use of the spectral
representation~\cref{eq:spectral_rep}, the weak solution of the
Poisson equation is readily seen to be
\begin{equation}\label{eq:spectral_sol}
\phi = \sum_{m=1}^\infty \frac{1}{\lambda_m}\lr{e_m}{h}e_m.
\end{equation}
This solution~\cref{eq:spectral_sol} also satisfies the fixed-point
equation~\cref{eq:Poisson-semigroup} because 
\begin{align*}
P_t \phi + \int_0^t P_s (h-\hat{h}_\pr) \ud s&=\sum_{m=1}^\infty e^{-t\lambda_m}\lr{e_m}{\phi}e_m + \int_0^t \sum_{m=1}^\infty e^{-s\lambda_m}\lr{e_m}{h}e_m \ud s\\
% &=\sum_{m=1}^\infty e^{-t\lambda_m}\lr{e_m}{\phi}e_m +\sum_{m=1}^\infty \
% \frac{1-e^{-t\lambda_m}}{\lambda_m}\lr{e_m}{h}e_m \\
&= \sum_{m=1}^\infty \frac{e^{-t\lambda_m}}{\lambda_m}\lr{e_m}{h}e_m +\sum_{m=1}^\infty 
\frac{1-e^{-t\lambda_m}}{\lambda_m}\lr{e_m}{h}e_m = \phi.
%\\
% &=\sum_{m=1}^\infty 
% \frac{1}{\lambda_m}\lr{e_m}{h}e_m= \phi
\end{align*}
The uniqueness of the solution to the fixed-point
equation~\cref{eq:Poisson-semigroup} follows from the contraction
mapping principle because $\|P_t\|_{L^2_0(\pr)}= e^{-t\lambda_1}<1$.

\section{Derivation of the linear form of the gain~\cref{eq:gain-linear-form}}\label{apdx:gain-linear-form}
By a direct calculation,
\begin{align*}
\nabla_x \frac{\kepsN(x,X^j)}{\sum_{l=1}^N \kepsN(x,X^l)} &= \frac{\frac{X^j-x}{2\epsilon} \kepsN(x,X^j)}{\sum_{l}\kepsN(x,X^l)} - \frac{\sum_{l=1}^N \frac{X^l-x}{2\epsilon}\kepsN(x,X^l)}{\sum_{l} \kepsN(x,X^l)} \frac{\kepsN(x,X^j)}{\sum_{l} \kepsN(x,X^l)},
\end{align*}
which evaluated at $x=X^i$ yields
\begin{align*}
\nabla_x\left. \left(\frac{\kepsN(x,X^i)}{\sum_{j=1}^N \keps(x,X^j)}\right)\right|_{x=X^i} 
%&= \frac{1}{2\epsilon}[(X^j-X^i){\sf T}_{ij}- \sum_{l=1}^N (X^l-X^i){\sf T}_{il}{\sf T}_{ij}]\\
=&\frac{1}{2\epsilon}\left( X^j{\sf T}_{ij}- \sum_{l=1}^N X^l{\sf
   T}_{il}{\sf T}_{ij} \right).
\end{align*}
Using the definitions \cref{eq:empirical_formula_for_gain} for $\K_\epsilon^{(N)}$, and
\cref{eq:s-r-definition} for $r$ and $s$,
\begin{align*}
\K_\epsilon^{(N)} (X^i) &= \nabla_x\left. \left( \frac{1}{n_\epsilon^{(N)}(x)}\sum_{j=1}^N \kepsN(x,X^j)
({\sf \Phi}_j + \epsilon   \hvec_j)\right) \right|_{x=X^i}\\
&= \nabla_x\left. \left(\frac{\sum_{j=1}^N \kepsN(x,X^j)
	r_j}{\sum_{j=1}^N \kepsN(x,X^j)
	}\right) \right|_{x=X^i}\\
%&=\nabla\left[\frac{\sum_{j=1}^N \frac{\geps(x,X^j)}{\sqrt{q_\epsilon^{(N)}(X^j)}}
%	({\sf \Phi}_j + \epsilon   \hvec_j)}{\sum_{j=1}^N  \frac{\geps(x,X^j)}{\sqrt{q_\epsilon^{(N)}(X^j)}
%}}\right]
%&=\frac{1}{2\epsilon}\left[\sum_{j=1}^N X^j{\sf T}_{ij}r_j- \sum_{j=1}^N\sum_{l=1}^N X^l{\sf T}_{il}{\sf T}_{ij}r_j \right]\\
&=\frac{1}{2\epsilon}\left( \sum_{j=1}^N X^j{\sf T}_{ij}(r_j- \sum_{l=1}^N {\sf T}_{il}r_l) \right)=\sum_{j=1}^N s_{ij}X^j.
\end{align*}

%\section{Proof of~\cref{prop:Teps-approx-new}}\label{apdx:Teps-approx-new}

\def\rd#1{{\color{red}#1}}

\section{Proof of~\cref{prop:N}}
\label{proof:prop:N}

\begin{romannum}
	\item $\Ten$ is a Markov matrix because
          $\Ten_{ij}=\frac{1}{\nepsN(X^i)}\kepsN(X^i,X^j)>0$ a.s. and \[\sum_{j=1}^N \Ten_{ij} = \frac{1}{\nepsN(X^i)}\sum_{i=1}^N \kepsN(X^i,X^j) = \frac{\nepsN{(X^i)}}{\nepsN(X^i)}=1\]
	 The stationary distribution is $\pi$ because
	\begin{align*}
	\sum_{i=1}^N \pi_i \Ten_{ij} &= \sum_{i=1}^N \frac{\nepsN(X^i)}{\sum_{k=1}^N \nepsN(X^k)} \frac{ \kepsN(X^i,X^j) }{\nepsN(X^i)}\\&=  \frac{\sum_{i=1}^N \kepsN(X^i,X^j)}{\sum_{k=1}^N \nepsN(X^k)} = \frac{\nepsN(X^j)}{\sum_{k=1}^N \nepsN(X^k)} = \pi_j.
	\end{align*}
	All entries of the Markov matrix are positive. Hence the Markov chain is irreducible and aperiodic. Therefore, the stationary distribution is unique.
	It is reversible because 
	\begin{align*}
	\pi_i \Ten_{ij} &= \frac{\nepsN(X^i)}{\sum_{k=1}^N \nepsN(X^k)} \frac{ \kepsN(X^i,X^j) }{\nepsN(X^i)}=\frac{\kepsN(X^j,X^i)}{\sum_{k=1}^N \nepsN(X^k)}\\&=\frac{\nepsN(X^j)}{\sum_{k=1}^N \nepsN(X^k)} \frac{ \kepsN(X^j,X^i) }{\nepsN(X^j)} =\pi_j \Ten_{ji}.
	\end{align*}
	\item Denote $\delta:=\min_{ij} \Ten_{ij}$.  Then $\delta>0$
          a.s. Therefore, $\|\Ten\|_{L^2_0(\pi)}\leq
          1-\frac{N\delta}{2}<1$, and is thus a contraction on
          $L^2_0(\pi)$~\cite[Ch. 5]{stroock2013introduction}). It
          follows, from the contraction mapping principle, that the
          fixed point equation~\cref{eq:fixed-pt-finite-N} has a
          unique solution.
	\item Evaluating the definition \cref{eq:phiepsN} at $x=X^i$ concludes $\phiepsN(X^i)= \phivec_i$ because,
	\begin{equation*}
	\begin{aligned}
	\phiepsN(X^i) &= \frac{1}{\nepsN(X^i)}\sum_{j=1}^n
        \kepsN(X^i,X^j) \phivec_j + \epsilon(h(X^i) - \pi(h)) \\&= \sum_{j=1}^N \Ten_{ij}\phivec_j +  \epsilon(\hvec_i - \pi(h))  = \phivec_i.
%\\&=\sum_{j=1}^N\Ten_{ij} \phivec_j + \epsilon (\hvec_i-\pi(h)) = \phivec_i
	\end{aligned}
%	\label{eq:phiepsN-phivec}
	\end{equation*}
	Therefore $\phiepsN$ solves the fixed-point equation~\cref{eq:empirical-fixed-pt}, because
	\begin{align*}
	\TepsN \phiepsN(x)  &= \frac{1}{\nepsN(x)}\sum_{j=1}^n \kepsN(x,X^j) \phiepsN(X^j)\\
	&=\frac{1}{\nepsN(x)}\sum_{j=1}^n \kepsN(x,X^j) \phivec_j\\
	&\overset{\cref{eq:phiepsN}}{=} \phiepsN(x) - \epsilon(h(x)-\pi(h)).
	\end{align*}
%	i.e., $\phiepsN$ solves the fixed-point equation~\cref{eq:empirical-fixed-pt}.  
\end{romannum}

\section{Proof of the~\cref{prop:Tepsn-convergence}}\label{apdx:Tepsn-convergence}
\begin{proof}
	\begin{romannum}
		\item Let $U =- \frac{1}{2}\log (\pr)$ and
                  $W=|\nabla U|^2-\Delta U$ as defined in
                  ~\cref{eq:Ueps-definition,eq:Weps-definition}.  To obtain the representation~\cref{eq:Pt-Feynman-Kac} for the semigroup $P_t$, consider the unitary transformation~\cite[Sec. 1.15.7]{bakry2013}:
		\begin{equation}
		e^{-U} \Delta_\pr \: e^{U} = \Delta  - W.
		\end{equation}  
		Therefore, for any function $f \in C_b(\Re^d)$,
		\begin{align*}
		e^{-U}~P_t \: e^{U} (f)  &= e^{t(\Delta-W)} (f) = \Expect[e^{-\int_0^t W(B^x_{2s})\ud s}f(B^x_{2t})],
		\end{align*}
where the stochastic representation (second equality) follows from the
Feynman-Kac formula; $B_t^x$ is a Brownian motion initialized at $x$.  Setting $f(x)=e^{-U(x)}g(x)$, 
		\begin{align*}
		P_t g(x)&=e^{U(x)}e^{t(\Delta-W)}(e^{-U}g)(x)=e^{U(x)}\Expect[e^{-\int_0^t W(B^x_{2s})\ud s}e^{-U(B^x_{2t})}g(B^x_{2t})],
		\end{align*}
		which is the representation~\cref{eq:Pt-Feynman-Kac}.
		
\medskip

		Next, the representation~\cref{eq:Teps-Feynman-Kac} is
                obtained.  Using the
                definitions, \eqref{eq:Teps-definition} of $\Teps$
                and \cref{eq:Ueps-definition,eq:Weps-definition} of
                $U_\epsilon$ and $W_\epsilon$,
		\begin{align*}
		\nTeps f (x) = \frac{G_\epsilon
                  (fe^{-U_\epsilon})(x)}{G_\epsilon
                  (e^{-U_\epsilon})(x)}&=e^{U_\epsilon(x)-\epsilon
                  W_\epsilon(x)}G_\epsilon (e^{-U_\epsilon}f)(x) \\
                &= e^{U_\epsilon(x)-\epsilon W_\epsilon(x)}\Expect[e^{-U_\epsilon(B^x_{2\epsilon})}f(B^x_{2\epsilon})],
		\end{align*}
                where the final equality follows 
		from using the stochastic representation of the heat semigroup $G_\epsilon$.
		The representation~\cref{eq:Teps-Feynman-Kac} is
                obtained by iterating this formula $n$ times. 
\medskip
		\item Without loss of generality, upon a change of
                  coordinates, assume $m=0$ and
                  $\Sigma=\text{diag}(\sigma_1^2,\ldots,\sigma_d^2)$
                  in Assumption A1.  Using the definitions	
		\begin{align}
		U_\epsilon(x) &= U(x) -\frac{1}{2}\log(\pr(x)) + \frac{1}{2}\log(\Geps (\pr) (x) ).\label{eq:ueps_u_log-Log}
		\end{align}

Now, $\log(\pr(x)) = \log (\pr_g(x;\Sigma)) + w(x)$.  So, the main
calculation is to approximate $\log(G_\epsilon(\pr))$. Using the definition
		\begin{align*}
		\Geps (\pr) (x) &=\int_{\Re^d} \geps(x,y)\pr_g(y;\Sigma)e^{-w(y)}\ud y \\
		&=
		\int_{\Re^d} \frac{e^{-\sum_{n=1}^d \frac{|x_n-y_n|^2}{4\epsilon}}}{(4\pi \epsilon)^{d/2}}\frac{e^{-\sum_{n=1}^d \frac{|y_n|^2}{2\sigma_n^2}-w(y)}}{\prod_{n=1}^d(2\pi \sigma_n^2)^{1/2}}\ud y\\
		&= \frac{^{-\frac{1}{2} \sum_{n=1}^d \frac{|x_n|^2}{2(\sigma_n^2+2\epsilon)}}}{\prod_{n=1}^d(2\pi (\sigma_n^2+2\epsilon))^{1/2}}  \int_{\Re^d} \frac{e^{-\sum_{n=1}^d \frac{|y_n-(1-\delta_n)x_n|^2}{4\epsilon(1-\delta_n)}}}{\prod_{n=1}^d(4\pi\epsilon(1-\delta_n))^{1/2}}e^{-w(y)}\ud y\\
		& = \pr_g(x;\Sigma + 2\epsilon I)  G^{(\delta)}_{\epsilon} (e^{-w})((I-\delta)x),
		\end{align*}
		where $\delta_n =
                \frac{2\epsilon}{\sigma_n^2+2\epsilon}$, $\delta =
                \text{diag}(\delta_1,\ldots,\delta_d)$ and
                $G^{(\delta)}_\epsilon$ is the semigroup
                associated with the PDE $\frac{\partial}{\partial
                  t}G^{(\delta)}_{t} f = G^{(\delta)}_t(
                \text{tr}((I-\delta)\nabla^2 f))$.

		The Taylor expansion of $\Gdelta_{\epsilon}(e^{-w})$,
                about $\epsilon=0$, is expressed as
		\begin{align*}
		\Gdelta_{\epsilon}(e^{-w})(x) &= e^{-w(x)} + \epsilon \text{Tr}((I-\delta)\nabla^2 e^{-w})(x) \\&+\underbrace{\int_0^\epsilon \int_0^\tau (\sum_{m,n=1}^d (1-\delta_m)^2(1-\delta_n)^2\Gdelta_{s}\partial^2_{n}\partial^2_m e^{-w})(x)\ud s\ud \tau }_{\epsilon^2r_{\epsilon}(x)},
		%	\\
		%	&= e^{-w(x)} + \epsilon\Delta e^{-w}(x) + O(\epsilon^2)
		\end{align*}
                where $\partial^2_m:=\frac{\partial^2}{\partial
                  x_n^2}$, and $\nabla^2 e^{-w}$ denotes is the Hessian matrix of $e^{-w}$.

		Using the property that $\Gdelta_s \partial_nf
                = \partial_n \Gdelta f$,  $\|\Gdelta_t(f)\|_{L^\infty}
                \leq \|f\|_{L^\infty}$ and the assumption (A1) that $w \in C^\infty_b(\Re^d)$, we conclude that $r_\epsilon \in C^\infty_b(\Re^d)$. 
		Therefore,
		\begin{align*}
		\log(G_{\epsilon}\pr(x)) 
		=&\log(\pr_g(x;\Sigma+2\epsilon I)) -\\&\underbrace{(w  -\log(1+ \epsilon\text{tr}((I-\delta)e^{w}\nabla^2e^{-w})+\epsilon^2e^{w}r_\epsilon))\vert_{(I-\delta)x}}_{w^{(1)}_\epsilon(x)}.
		\end{align*}
		The asymptotic expansion of $w_\epsilon^{(1)}$, as $\epsilon
                \to 0$, is obtained as 
		\begin{align*}
		w^{(1)}_\epsilon(x) &= w(x) - 2\epsilon x^\top \Sigma^{-1} \nabla w(x)  - \epsilon e^{w}\Delta e^{-w}(x) +O(\epsilon^2),
		\end{align*}
		where the remainder term has at most linear growth as $|x|\to \infty$.
		
		Substituting the asymptotic expression for
                $\log(G_{\epsilon}\pr(x))$ in~\cref{eq:ueps_u_log-Log},
		\begin{align*}
		U_\epsilon(x) &= U(x) -\frac{1}{2}\log (\pr_g(x;\Sigma)) +\frac{1}{2}w(x) + \frac{1}{2}\log (\pr_g(x;\Sigma+2\epsilon I)) -\frac{1}{2}w^{(1)}_\epsilon(x)\\
		&=U(x) + \frac{\epsilon}{2}x^\top
                  \Sigma^{-1}(\Sigma+2\epsilon I)^{-1}x -
                  \frac{\epsilon}{2}\text{Tr}(\Sigma^{-1})\\&\quad \quad+\epsilon x^\top \Sigma^{-1}\nabla w(x) + \frac{\epsilon}{2}(|\nabla w(x)|^2-\Delta w(x)) + O(\epsilon^2) \\
		&=U(x) + \underbrace{\frac{\epsilon}{2}|\Sigma^{-1}x + \nabla w(x)|^2 - \frac{\epsilon}{2}(\trace(\Sigma^{1})+\Delta w(x))}_{2\epsilon W(x) +\frac{\epsilon}{2} \Delta V(x)} + O(\epsilon^2),
		\end{align*}
		where the remainder $O(\epsilon^2)$ error term has at
                most quadratic growth as $|x| \to \infty$. This
                concludes the proof of approximation~\cref{eq:Ueps-U}.

Based on this above calculation, the following
                estimate for an upper bound of the function
                $U$ is obtained (it is used in the proof of~\cref{prop:DV3}): 
	\begin{align}
	U_\epsilon(x) 
	&\leq \frac{1}{4} x^\top \Sigma^{-1}x + \epsilon( | \Sigma^{-1}x|^2 + \|\nabla w\|_{L^\infty}^2 + \|\Delta V\|_{L^\infty} ) + \epsilon^2(C_1|x|^2+C_2) \nonumber\\
	&\leq \frac{1}{8\sigma_1^2}|x|^2 + \frac{\sigma_1^2}{8}(  \|\nabla w\|_{L^\infty}^2 + \|\Delta V\|_{L^\infty} + \frac{C_2\sigma_1^2}{8}), \label{eq:ueps_upperbound}
	\end{align} 
where recall $\sigma_1^2=\lambda_\text{min}(\Sigma)$.
\medskip
		
		Next, the approximation~\cref{eq:Weps-W} is derived.
                Using the definition
		\begin{equation*}
		\epsilon W_\epsilon(x) = U_\epsilon(x) + \log(G_\epsilon e^{-U_\epsilon}(x)).
		\end{equation*}
		% Next, we form an approximation to $\log(G_\epsilon(e^{-\Ueps}))$. We begin by expressing $e^{-\Ueps}$ as
% 		\begin{align*}
% 		e^{-U_\epsilon(x)}&=\frac{\pr(x)}{\sqrt{G_\epsilon(\pr)(x)}}
% 		=
% %		\frac{\pr_g(x;\Sigma)e^{-w(x)}}{\sqrt{\pr_g(x;\Sigma+2\epsilon)e^{-w^{(1)}_\epsilon(x)}}}\\&=  
% C\pr_g(x;2\Sigma(I+\delta)^{-1})e^{-(w(x)-\frac{1}{2}w^{(1)}_\epsilon(x))}
% 		\end{align*} 
% 		where $C$ is a constant, $\delta_n = \frac{2\epsilon}{\sigma_n^2+2\epsilon}$ and $\delta = \text{diag}(\delta_1,\ldots,\delta_d)$. 
		By repeating the steps, just used to approximate
                $\log(G_\epsilon(\pr))$, it is shown
		\begin{align*}
		\log(G_\epsilon(e^{-\Ueps})) = 	\log(\pr_g(x;2\Sigma(I+\delta)^{-1}+2\epsilon I)) - w^{(2)}_\epsilon(x),	
		\end{align*}
		where
		\begin{align*}
		w^{(2)}_\epsilon(x) 
		&= w(x)-\frac{1}{2}w^{(1)}_\epsilon(x) -\frac{\epsilon}{2} x^\top\Sigma^{-1}\nabla w(x) -\epsilon(\frac{1}{4}|\nabla w(x)|^2-\frac{1}{2}\Delta w(x))+  O(\epsilon^2).
		\end{align*}
		Therefore,
		\begin{align*}
		\epsilon W_\epsilon(x)&= -\log(\pr_g(x;2\Sigma(I+\delta)^{-1}))  + \log(\pr_g(x;2\Sigma(I+\delta)^{-1}+2\epsilon I)) \\ &\quad\quad+w(x)- \frac{1}{2}w^{(1)}_\epsilon(x) - w^{(2)}_\epsilon(x) + O(\epsilon^2)\\&= \frac{2\epsilon}{4}x^\top (I+\delta)\Sigma^{-1}(2\Sigma(I+\delta)^{-1}+2\epsilon I)^{-1} x -\frac{\epsilon}{2} \text{Tr}(\Sigma^{-1}) \\
		&\quad\quad+\epsilon x^\top\Sigma^{-1}\nabla w(x) +\frac{\epsilon}{2}(\frac{1}{4}|\nabla w(x)|^2-\frac{1}{2}\Delta w(x))+  O(\epsilon^2)\\
		&= \epsilon\underbrace{(\frac{1}{4}|\Sigma^{-1}x + \nabla w(x)|^2 -\frac{1}{2}(\text{Tr}(\Sigma^{-1}) + \Delta w(x)))}_{W(x)} + O(\epsilon^2),
		\end{align*}
		where the error term has at most quadratic growth as
                $|x|\to \infty$. This concludes the proof of the
                approximation~\cref{eq:Weps-W}.

                Based on this above calculation, the following
                estimate for a lower bound of the function
                $W_\epsilon$ is obtained (it is used in the proof of ~\cref{prop:DV3}): 
	\begin{align}
	W_\epsilon(x) %&= W(x) + \epsilon r^{(2)}_\epsilon(x) \\
	&= \frac{1}{4}|\Sigma^{-1}x+\nabla
          w(x)|^2-\frac{1}{2}(\text{Tr}(\Sigma^{-1}) + \Delta w(x))  +
          \epsilon r^{(2)}_\epsilon(x) \nonumber \\
	&\geq \frac{1}{8}|\Sigma^{-1}x|^2 - \frac{1}{2}(\|\nabla
          w\|_{L^\infty}^2  +\text{Tr}(\Sigma^{-1}) +\| \Delta w\|_{L^\infty}
          ) - \epsilon (C_1|x|^2 +C_2) \nonumber \\
	&\geq \alpha |x|^2 - \beta,  \label{eq:weps_lower_bound}
	\end{align}
	where $\alpha = \frac{1}{16\sigma_d^4}$, $\beta =
        \frac{1}{2}(\|\nabla w\|_{L^\infty}^2 +\text{Tr}(\Sigma^{-1}) +\|
        \Delta w\|_{L^\infty} + \frac{C_2}{8\sigma_1^2}) $ and
        $\epsilon\leq \frac{1}{16C_1\sigma_d^4}$ (where recall $\sigma_d^2 = \lambda_\text{max}(\Sigma)$). 
		\item Let $\tilde{P}_t$ denote the semigroup
                 for the weighted Laplacian $\Delta_q$
                  with the density $q(x)=e^{-2U_\epsilon(x)}$.  We break
                  the error into two parts:
		\begin{equation*}
		\|\Teps^n f - P_t f\|_{L^2(\pr)} \leq 	\|\Teps^n f - \tilde{P}_t f\|_{L^2(\pr)} + 	\| \tilde{P}_t f-P_t f\|_{L^2(\pr)}.
		\end{equation*} 
                The bounds for the two terms on the right-hand side are
                derived in the following two steps:  

                \textbf{Step 1.} Using the stochastic
                representation~\cref{eq:Pt-Feynman-Kac}-\cref{eq:Teps-Feynman-Kac},
%Subtracting $\tilde{P}_t f(x)$ given by~\cref{eq:Pt-Feynman-Kac} (with $U$ replaced by $U_\epsilon$) from $T_{\epsilon}^nf(x)$ given by~\cref{eq:Teps-Feynman-Kac} yields
		\begin{align*}
		(T_{\epsilon}^n -\tilde{P}_t)f(x)  = e^{U_\epsilon(x)} \Expect\left[e^{-U_\epsilon(B^x_{2t})}f(B^x_{2t})\;\zeta_t \right],
		\end{align*}
		where $\zeta_t := e^{-\epsilon\sum_{k=0}^{n-1}W_{\epsilon}(B^x_{2k\epsilon})} - e^{ -\int_0^t W(B^x_{2s})\ud s }$.
		By the Cauchy-Schwartz inequality
		\begin{align*}
		|(T_{\epsilon}^n -\tilde{P}_t)f(x)|  &\leq
                                                       e^{U_\epsilon(x)}
                                                       \;\Expect[|f(B^x_{2t})|^2e^{-2U_\epsilon(B^x_{2t})}]^{\frac{1}{2}}
                                                       \; \Expect\left[|\zeta_t|^2\right]^{\frac{1}{2}}.
		\end{align*}
		Next we obtain a bound for $\zeta_t$. Upon using the inequality $|e^{-x}-e^{-y}| \leq e^{-\min(x,y)}|x-y|$, 
		\begin{align}
		|\zeta_t|
%&\leq e^{-C} 
%		\left|\epsilon\sum_{k=0}^{n-1}W_{\epsilon}(B^x_{2k\epsilon}) -\int_0^t W(B^x_{2s})\ud s\right|\\&
&\leq
  e^{-C}\left|\sum_{k=0}^{n-1}\epsilon(W_{\epsilon}(B^x_{2k\epsilon})-W(B^x_{2k\epsilon})
  )\right|+e^{-C}\left|\int_0^t W(B^x_{2s})\ud s -
  \sum_{k=0}^{n-1}\epsilon W(B^x_{2k\epsilon})\right|,
\label{eq:zetatbound}
		\end{align}
		where  $C = t\min(\min_{x\in\Re^d} W(x),
                \min_{x\in\Re^d} W_\epsilon(x))$. Now, $C$ is finite
                because, as $|x|\to \infty$, $W(x)\rightarrow\infty$  (Assumption A1) and
                $W_\epsilon(x) \rightarrow\infty$
                (by~\eqref{eq:Weps-W}).  As a result, by triangle inequality, $\Expect[|\zeta_t|^2]^{\frac{1}{2}}\leq e^{-C}\Expect[| \text{(first term)} |^2]^{\frac{1}{2}}  + e^{-C}\Expect[|\text{(second term)}|^2]^{\frac{1}{2}}$. 
                %		 in the form of
		
                The expectation of the first term  is bounded
                as follows:
		\begin{align*}
		\Expect\left[\left| \sum_{k=0}^{n-1}\epsilon(W_{\epsilon}(B^x_{2k\epsilon})-W(B^x_{2k\epsilon}) )\right|^2\right]^{\frac{1}{2}}&\leq\sum_{k=0}^{n-1}\epsilon \Expect[|W_\epsilon(B^x_{2k\epsilon})-W(B^x_{2k\epsilon})|^2]^{\frac{1}{2}} \\&\leq \sum_{k=0}^{n-1} \epsilon^2\Expect[(\CA |x+B^x_{2k\epsilon}|^2+\CB)^2]^{\frac{1}{2}}\\
		&\leq \sum_{k=0}^{n-1}\epsilon^2(2\CA |x|^2+2\CA\Expect[|B^{x}_{2k\epsilon}|^4]^{\frac{1}{2}}+\CB)\\
		&\leq \epsilon t\left[2\CA|x|^2+ 6\CA t +\CB   \right], 
		\end{align*}
		where the second inequality follows from 
                the bound $|W_\epsilon(x)-W(x)|=\epsilon
                |r^{(2)}_\epsilon(x)|\leq \epsilon (C_1|x|^2+C_2)$ for
                some constants $C_1,C_2$ (see~\cref{eq:Weps-W}). 
		
		The expectation of the second term in~\cref{eq:zetatbound} is bounded
                as follows:
		\begin{align*}
		\Expect\left[\left|\int_0^t W(B^x_{2s})\ud s - \sum_{k=0}^{n-1}\epsilon W(B^x_{2k\epsilon})\right|^2\right]^{\frac{1}{2}}&\leq \epsilon  (\Expect\bigg[\int_0^t |\nabla W(B^x_{2s})|^2\ud s \bigg]^{\frac{1}{2}}+t\|\Delta W\|_\infty)\\&\leq \epsilon  (\Expect\bigg[\int_0^t |C_3|x+B_s|+C_4|^2\ud s \bigg]^{\frac{1}{2}}+tC_5)\\
		&\leq \epsilon t^{\frac{1}{2}}(C_3 |x|+C_3 t + C_4) + \epsilon C_5 t , 
		\end{align*}
		where the Taylor expansion of $W(x)$ is used to obtain
                the first inequality, and for the second inequality,
                Assumption (A1) is used to bound $|\nabla W(x)| \leq
                \|\Sigma^{-1}\| |x| + \|\nabla w\|_{L^\infty}= C_3|x|+C_4$
                and $\|\Delta W\|_{L^\infty} \leq \|\Sigma^{-1}\| +
                \|\Delta w\|_{L^\infty} = C_5$. 

Putting together the two expectation bounds ,
		\begin{equation*}
		|(T_{\epsilon}^n -\tilde{P}_t)f(x)|  \leq e^{U_\epsilon(x)} \Expect[f^2(B^x_{2t})e^{-2U_\epsilon(B^x_{2t})}]^{\frac{1}{2}} C\epsilon t^{\frac{1}{2}}(|x|^2+1),
		\end{equation*}
		where $C$ is a constant that only depends on $t_0$.
		Upon taking the $L^2(\pr)$ norm
		\begin{align*}
		\|T_{\epsilon}^n -&\tilde{P}_tf\|_{L^2(\pr)}^2  \\&\leq C\epsilon^2 t
		\int \int f^2(y) \pr_g(x-y;2t)(|x|^2+1)^2e^{2U_\epsilon(x)-2U_\epsilon(y)} \pr(x)\ud y\ud x\\
		&\leq C\epsilon^2 t
		\int \int f^2(y) \pr_g(x-y;2t)(|x|^4+1)e^{2\epsilon(2W(x)+\frac{1}{2}\Delta V(x) + \epsilon r_\epsilon^{(1)}(x))}\pr(y) \ud y\ud x\\
		&\leq C\epsilon^2  t 
		\int f^2(y) (|y|^4+12 t^2 + 1)e^{8\epsilon W(y)+O(\epsilon^2)} \pr(y)\ud y
		\\
		&\leq C\epsilon^2 t\left[\int (|x|^4+ 12 t^2 +1)^2e^{8\epsilon W(x)+O(\epsilon^2)}\pr(x)\ud x\right]^{1/2}\left[\int f^4(x)\pr(x)\ud x\right]^{1/2}\\
		&\leq C \epsilon^2 t \|f\|^2_{L^4(\pr)}.
		\end{align*}

		\textbf{Step 2.} Because $P_t$ and $\tilde{P}_t$ are semigroups with generators $\Delta_\pr$ and $\Delta_q$, respectively, we have the identity:
		$
		P_t f  - \tilde{P}_t f = \int_0^t P_{t-s}(\Delta_\pr - \Delta_q) \tilde{P}_sf \ud s$. 
		Upon taking the $L^2(\pr)$ norm of both sides, using
                the triangle inequality, because $P_t$ is contraction on $L^2(\pr)$,
		\begin{align*}
		\|P_t f  - \tilde{P}_t f\|_{L^2(\pr)} \leq  \int_0^t \|(\Delta_\pr - \Delta_q) \tilde{P}_sf\|_{L^2(\pr)} \ud s.
		\end{align*}
		Now,
		\begin{align*}
		&\|(\Delta_\pr - \Delta_q) \tilde{P}_sf\|^2_{L^2(\pr)} = 4\int |(\nabla U(x) - \nabla U_\epsilon(x)) \cdot \nabla (\tilde{P}_s f)(x)|^2 \pr(x)\ud x\\
		&\quad\leq 4\left[\int |\nabla U(x)-\nabla U_\epsilon(x)|^4\frac{\pr(x)}{q(x)} \pr(x) \ud x\right]^{1/2} \left[\int |\nabla \tilde{P}_s f(x)|^4 q(x)\ud x\right]^{1/2}\\
		&\quad\leq 4\epsilon^2\left[\int |C_1|x|+C_2|^4e^{-2U(x)+2U_\epsilon(x)} \pr(x) \ud x\right]^{1/2} \|\nabla f\|^2_{L^4(q)}\\
		&\quad \leq C\epsilon^2 \|\nabla f\|^2_{L^4(\pr)},
		\end{align*}
		where the identity $\Delta_\pr f - \Delta_q f= 2\nabla
                U \cdot \nabla f  -  2 \nabla U_\epsilon \cdot \nabla
                f$ is used in the first step, the Cauchy-Schwartz
                inequality in the second step, and the bounds $|\nabla
                U_\epsilon(x) - \nabla U(x)| \leq \epsilon
                (C_1|x|+C_2)$ and $\|\nabla \tilde{P}_s
                f\|_{L^4(q)}\leq \|\nabla f\|_{L^4(q)}$ in the third
                step.
% ,  and the  finiteness of the integral and $q(x)=e^{2U_\epsilon(x)-2U(x)}\pr(x)<C\pr(x)$ in the last step. Putting the bounds from term I and term II together gives the approximation result~\cref{eq:Teps-Peps-limit}.
		
		%		\begin{align*}
		%\|P_t f  - \tilde{P}_t f\|_{L^2(\pr)} \leq  C\epsilon t \|\nabla f\|_{L^4(\pr)}
		%\end{align*}	

\medskip

	Combining the two sets of bounds in steps 1 and 2, one obtains~\cref{eq:Teps-Peps-limit}.
	\end{romannum}
	
\end{proof}

\section{Proof of the~\cref{prop:DV3}}\label{apdx:DV3}
\begin{proof}
%	Suppose the Lyapunov condition~\cref{eq:DV3} is true.  Then
	
	(i)  The Lyapunov condition~\cref{eq:DV3}, known as DV(3) of~\cite{konmey12a}, is the necessary and sufficient condition for geometric ergodicity  (and in fact the stronger $U_\epsilon$-uniform ergodicity)~\cite[Thm. 15.0.1]{MT}. The distribution $\preps$ is invariant because $\forall f \in C_b(\Re^d)$,
	\begin{align*}
\int \Teps f(x)\preps(x)\ud x & =\int \int   \frac{1}{\neps(x)}\keps(x,y)f(y)\pr(y)\ud y \frac{\neps(x)\pr(x)}{\int \neps(z)\pr(z)\ud z}\ud x\\&=
	\frac{1}{\int \neps(z)\pr(z)\ud z}  \int \int \keps(x,y) \pr(x) \ud x f(y)\pr(y)\ud y\\
	&= \frac{1}{\int \neps(z)\pr(z)\ud z} \int f(y) \neps(y)\pr(y) \ud y
	=\int f(x)\preps(x)\ud x.
	\end{align*}
	
	(ii)  The invariant density $\preps$ is reversible because $\forall f,g\in C_b(\Re^d)$
	\begin{align*}
\int g(x)\Teps f(x) \preps(x)\ud x &= \int \int g(x) \frac{\keps(x,y)}{\neps(x)}f(y)\pr(y) \frac{\neps(x)\pr(x)}{\int \neps(z)\pr(z)\ud z}\ud y\ud x\\
	&=\frac{1}{\int \neps(z)\pr(z)\ud z} \int \Teps g(y) \neps(y) f(y) \pr(y)\ud y \\&=
	 \int  f(y) \Teps g(y) \preps(y)\ud y .
%= \lr{f}{\Teps g}
	\end{align*} The spectral gap follows from Lyapunov
        condition~\cref{eq:DV3} and the fact that the chain is
        reversible~\cite[Thm 2.1]{roberts1997geometric}.  The spectral
        gap is denoted as $\lambda$.  
%	Note that (i) does not imply (ii) without additional assumptions \cite{konmey12a}. 
	
	(iii) The solution $\phieps$ satisfies the bound:
	\begin{align*}
	\|\phi\|_{L^2(\preps)}&\leq
                                \frac{\|\sum_{k=0}^{n-1}\epsilon\Teps^k
                                (h-\hat{h}_{\pr_\epsilon})\|_{L^2(\preps)}}{1-\|\Teps^n\|_{L^2_0(\preps)}}
%\leq \frac{\sum_{k=0}^{n-1}\epsilon\|\Teps^k
%h\|_{L^2_0(\preps)}}{1-\|\Teps^n\|_{L^2_0(\preps)}}\\&
\leq \frac{\epsilon n \|h\|_{L^2(\preps)}}{1-\|\Teps^n\|_{L^2_0(\preps)}}\leq \frac{t\|h\|_{L^2(\preps)}}{\lambda}.
	\end{align*}
%	where we used the fact that $\Teps$ is contraction on $L^2_0(\preps)$.
%	 This also follows from (V4):  Theorem~17.4.2 of \cite{MT} implies existence, and
%	Prop.~17.4.1 gives uniqueness.  
	
	It remains to verify the Lyapunov
        condition~\cref{eq:DV3}: Using~\cref{eq:Teps-Feynman-Kac}
	\begin{align*}
	e^{-U_\epsilon}T_{\epsilon}^n e^{U_\epsilon}(x)
          &=\Expect[e^{-\epsilon\sum_{k=0}^{n-1}W_{\epsilon}(B^x_{2k\epsilon})}] \leq \Expect[e^{-\epsilon\sum_{k=0}^{n-1}(\alpha|B^x_{2k\epsilon}|^2-\beta)}],
	\end{align*}
        where the second inequality follows from using the lower bound $W_\epsilon(x) \geq \alpha |x|^2 - \beta$
        derived in~\cref{eq:weps_lower_bound}.

%{eq:ueps_upperbound}
	We now claim that 
	\begin{equation}\label{eq:gamma-claim}
	\Expect[e^{-\epsilon\sum_{k=0}^{m-1}(\alpha|B^x_{2k\epsilon}|^2-\beta)}]=e^{-\alpha_m|x|^2+\beta_m},
	\end{equation}
	for $m=1,\ldots,n$
	where $\{\alpha_m\}_{m=1}^n$ and  $\{\beta_m\}_{m=1}^n$ are
        defined using the recursions:
	\begin{align*}
	\alpha_{m+1}&=\alpha \epsilon + \frac{\alpha_m}{1+4\epsilon\alpha_m},\quad\alpha_1=\alpha \epsilon \\\beta_{m+1} &= \beta_m + \beta \epsilon - \frac{1}{2}\log(1+4\epsilon\alpha_m),\quad \beta_1=\beta \epsilon.
	\end{align*}
	Assuming for now that the claim is true 
	\begin{align*}
	\log(e^{-U}T_{\epsilon}^n e^U(x)) &\leq
                                              \log(\Expect[e^{-\epsilon\sum_{k=0}^{n-1}(\alpha|B^x_{2k\epsilon}|^2-\beta)}])
                                              = -\alpha_n|x|^2+\beta_n. 
	\end{align*}
	An upper-bound for $\beta_n$ and a lower-bound
        for $\alpha_n$ are obtained as follows: 
\begin{enumerate}
\item For the sequence $\{\beta_m\}_{m=1}^n$,
	\begin{equation*}
	\beta_{m+1}\leq \beta_m + \beta \epsilon,\quad \Rightarrow \quad \beta_n \leq \beta_1 + (n-1) \beta \epsilon  = \beta t.
	\end{equation*} 
\item For the sequence $\{\alpha_m\}_{m=1}^n$,
	\begin{equation*}
	\alpha_{m+1} \leq \alpha_m + \alpha \epsilon,\quad \Rightarrow\quad \alpha_m \leq \alpha_1+(n-1)\alpha\epsilon = \alpha t.
	\end{equation*}
Therefore,
\[
\alpha_{m+1}\geq \frac{\alpha_m}{1+4\epsilon\alpha t} + \alpha \epsilon,\quad\alpha_1=\alpha \epsilon.
\]
It then follows
\[
\alpha_n \geq \alpha t e^{-4\alpha t^2}.
\]
\end{enumerate}
	% The lower-bound for $\alpha_n$ is obtained as follows: First note the upper-bound
	% \begin{equation*}
	% \alpha_m \leq \alpha \epsilon + \alpha_m,\quad \Rightarrow\quad \alpha_n \leq \alpha_1+(n-1)\alpha\epsilon = \alpha t
	% \end{equation*}
	% Then use this to conclude 
	% \begin{align*}
	% \alpha_{m+1}\geq \alpha \epsilon + \alpha_m(1-4\epsilon \alpha t),
	% %\quad \Rightarrow\quad \alpha_n \geq \frac{1-(1-4\epsilon\alpha t)^n}{4\epsilon \alpha t}\alpha\epsilon = \frac{1-(1-\frac{4\alpha t^2}{n})^n}{4t}
	% \end{align*}
	% where we used $\frac{1}{1+x}\geq 1-x$. Therefore,
	% \begin{align*}
	% %\alpha_{m+1}\geq \alpha_\epsilon + \alpha_m(1-4\epsilon \alpha t)
	% %,\quad \Rightarrow\quad 
	% \alpha_n \geq \frac{1-(1-4\epsilon\alpha t)^n}{4\epsilon \alpha t}\alpha\epsilon = \frac{1-(1-\frac{4\alpha t^2}{n})^n}{4t} \geq \frac{1-e^{-4\alpha t^2}}{4t}\geq \alpha t e^{-4\alpha t^2}
	% \end{align*}
	% where we used $(1-x)^n\leq e^{-nx}$ for $x\geq 0$ in the
        % second inequality and $e^{-x}\leq 1-xe^{-x}$ for $x\geq 0$
        % in the last inequality.  This gives 
Upon using the two bounds
	\begin{align*}
	\log(e^{-U_\epsilon}T_{\epsilon}^n e^{U_\epsilon}(x))
	%&\leq -\alpha_n|x|^2+\beta_n \\ 
	&\leq -\alpha te^{-4\alpha t^2} |x|^2+ \beta t
	\leq -a t U_\epsilon(x) + bt,
	\end{align*}
	where the second inequality follows from using the upper
        bound $U_\epsilon(x)  \leq \frac{1}{8\sigma_1^2}|x|^2 + \frac{\sigma_1^2}{8}C$
        derived in~\cref{eq:ueps_upperbound}.  The following estimates
        are obtained for constants
\[
a= 8\sigma_1^2\alpha e^{-4\alpha t_0^2},\quad\quad b = \beta +
C \sigma_1^4\alpha e^{-4\alpha t_0}.
\]

% a quadratic upper-bound on $U_\epsilon$. The upper-bound is obtained as follows:
% 	\begin{align*}
% 	U_\epsilon(x) &= U(x) + \epsilon (2W(x) + \frac{1}{2}\Delta V(x)) + \epsilon^2 r^{(1)}(x)\\
% 	&\leq \frac{1}{4} x^\top \Sigma^{-1}x + \epsilon( | \Sigma^{-1}x|^2 + \|\nabla w\|_\infty^2 + \|\Delta V\|_\infty ) + \epsilon^2(C_1|x|^2+C_2) \\
% 	&\leq \frac{1}{8\sigma_1^2}|x|^2 + \frac{\sigma_1^2}{8}(  \|\nabla w\|_\infty^2 + \|\Delta V\|_\infty + \frac{C_2\sigma_1^2}{8} ) 
% 	\end{align*} 
% 	where we used the approximation~\cref{eq:Ueps-U}, the explicit formula for $U(x)$ from Assumption A1,  and we take $\epsilon$ small enough such that $\frac{\epsilon}{\sigma_1^4} + \epsilon^2C_1\leq \frac{1}{8\sigma_1^2}$. The resulting constants in the Lyapunov condition~\cref{eq:DV3} are
%	 Using the quadratic upper-bound on $U_\epsilon$:
%		\begin{align*}
%		\log(e^{-U_\epsilon}T_{\epsilon_n}^n e^{U_\epsilon}(x))
		%&\leq -\alpha_n|x|^2+\beta_n \\ 
%		&\leq -\alpha te^{-4\alpha t^2} |x|^2+ \beta t\\
%		&\leq -\lambda t U_\epsilon(x) + bt
%		\end{align*}
%		where 
%		$\lambda = 8\sigma_1^2\alpha e^{-4\alpha t_0}$ and $b = \beta + \sigma_1^4\alpha e^{-4\alpha t_0}$. 
		
%$e^{-4\alpha t^2} \geq e^{-4\alpha t_0^2}$, and 
%$\lambda
% U_\epsilon(x) \leq  \alpha e^{-4\alpha t^2} |x|^2 +  b- \beta$ for some constants $b$ and $\lambda$ from approximation~\cref{eq:Ueps-U} and $U(x)\leq \frac{1}{4}(x-m)^T\Sigma^{-1} (x-m)+ \|w\|_\infty $ from Assumption A1.
	%Therefore
	%\begin{align*}
	%e^{-U}T_{\epsilon_n}^n e^U(x) =\Expect[e^{-\epsilon\sum_{k=0}^{n-1}W_{\epsilon_n}(B^x_{2k\epsilon_n})}]
	%\end{align*}
	
It remains to prove the claim~\cref{eq:gamma-claim}.  The constants
$\alpha_1$ and $\beta_1$ 
for $m=1$ are easily verified by direct evaluation and for $m>1$,
	\begin{align*}
	\Expect[e^{-\epsilon\sum_{k=0}^{m}(\alpha|B^x_{2k\epsilon}|^2-\beta)}]&=\Expect[e^{-\epsilon \alpha|x|^2+\beta\epsilon}e^{-\alpha_m|B_{2\epsilon}|^2+\beta_m}]\\
%	&= e^{\epsilon\beta + \beta_m-\epsilon\alpha|x|^2}\Expect[e^{-\alpha_m|B_{2\epsilon}|^2}]\\
	&= e^{-\epsilon\alpha|x|^2-\frac{\alpha_m}{1+4\epsilon\alpha_m}
          |x|^2 + \epsilon\beta +
          \beta_m - \frac{1}{2}\log(1+4\epsilon\alpha_m)}.
	\end{align*}

	The minorization inequality~\cref{eq:minorization} is obtained next. For $|x|\leq R$:
	\begin{align*}
	T_{\epsilon}^n\mathds{1}_A(x) &= e^{U_\epsilon(x)}
                           \Expect[e^{-\epsilon\sum_{k=0}^{n-1}W_\epsilon(B^x_{2k\epsilon})}e^{-U_\epsilon(B^x_{2t})}\mathds{1}_{B^x_{2t}\in
                           A}]\\&\geq 
\frac{e^{\min_{|x|\leq R}
                                  U_\epsilon(x)}}{e^{\max_{|x|\leq R+10} (U_\epsilon(x)+tW_\epsilon(x))}} \P([B^x_{2t} \in A]\cap[\underset{s \in[0,2t]}{\sup}|B_{s}|\leq 10])%\\
	\geq \delta \nu(A),
	\end{align*}
	where 
	\begin{align*}
	\nu(A) &= \P(\{B^x_{2t} \in A\}|\{\underset{s \in[0,2t]}{\sup}|B_{s}|\leq 10\})\\
	\delta &=\frac{e^{\min_{|x|\leq R,\epsilon \in(0,1)}
	                                  U_\epsilon(x)}}{e^{\max_{|x|\leq R+10,\epsilon \in(0,1)} (U_\epsilon(x)+tW_\epsilon(x))}} (1-2e^{-\frac{50}{t_0}}),
	\end{align*}
	because $\P(\underset{s \in[0,2t]}{\sup} B_s\geq 10) \leq
        e^{-\frac{100}{2t}}\leq e^{-\frac{50}{t_0}}$. 
%        [AMIR: why the
%        half?  How did you swap $U_\epsilon$ to U?]
\end{proof}

\section{Proof of the~\cref{thm:bias}}\label{apdx:bias}
\begin{proof} 
	\begin{romannum}
		\item The existence of the solution is proved in~\cref{prop:DV3}.
%		 Note that the solution to~\cref{eq:fixed-point-n} and \cref{eq:kernel-approx-fixed-pt} are equal.
		\item We break the error into two parts: 
		\begin{equation*}
		\|\phieps-\phi\|_{L^2(\preps)}\leq \|\phieps-\tilde{\phi}\|_{L^2(\preps)} +\|\tilde{\phi}-\phi\|_{L^2(\preps)},
		\end{equation*}
		where $\tilde{\phi}$ is the solution to the fixed
                point equation $\tilde{\phi} = P_\epsilon \tilde{\phi}
                + \epsilon(h-\hat{h}_\pr)$ with the exact semigroup
                $P_\epsilon$. The bounds for the two terms on the
                right-hand side are derived in the following two steps:

\textbf{Step 1.} Iterating the  formula $\tilde{\phi} = P_\epsilon \tilde{\phi} + \epsilon(h-\hat{h}_\pr)$ for  $n=\lfloor \frac{1}{\epsilon}\rfloor$ times yields,
		\begin{equation*}
		\tilde{\phi} = P_{\epsilon}^n \tilde{\phi} + \sum_{k=0}^{n-1} \epsilon P_\epsilon^k (h-\hat{h}_\pr),
		\end{equation*}
		and subtracting this from~\cref{eq:fixed-point-n} gives
		\begin{align*}
		\phieps - \tilde{\phi} = \Teps^n (\phieps - \tilde{\phi}) + (\Teps^n - P_\epsilon^n)\tilde{\phi} + \sum_{k=0}^{n-1} \epsilon (\Teps^k - P_\epsilon^k) h + t(\hat{h}_\pr-\hat{h}_{\pr_\epsilon}). 
		\end{align*}
		This forms a (discrete) Poisson equation whose
                solution exists and is bounded according
                to~\cref{prop:DV3}:
		\begin{equation}
		\begin{aligned}
		\|\phieps - &\tilde{\phi}\|_{L^2(\preps)} \\&\leq
                \frac{n\epsilon}{\lambda} \left( \|{(\Teps^n -
                  P_\epsilon^n)\tilde{\phi}}\|_{L^2(\preps)} +
                \|{\sum_{k=0}^{n-1} \epsilon (\Teps^k - P_\epsilon^k)
                  h}\|_{L^2(\preps)}+
                n\epsilon|\hat{h}_\pr-\hat{h}_{\pr_\epsilon}| \right) \\
		&\leq \frac{C n\epsilon}{\lambda} \left( \|{(\Teps^n -
                    P_\epsilon^n)\tilde{\phi}}\|_{L^2(\pr)} +
                  \|{\sum_{k=0}^{n-1} \epsilon (\Teps^k -
                    P_\epsilon^k) h}\|_{L^2(\pr)}+
                  n\epsilon|\hat{h}_\pr-\hat{h}_{\pr_\epsilon}| \right),
		\end{aligned}
	\label{eq:phieps-phitilde}
	\end{equation}
		where we used $\|\cdot\|_{L^2(\preps)}\leq
                C\|\cdot\|_{L^2(\pr)}$ in the second step. This is
                true because $\preps(x) = e^{-U_\epsilon(x)}G_\epsilon
                (e^{-U_\epsilon})(x) = \pr(x) e^{-3\epsilon W(x) -
                  \epsilon \Delta V(x) + O(\epsilon^2)}\leq C\pr(x)$
                using the formula~\cref{eq:Ueps-U}.
% 		\begin{align*}
% 		\preps(x)
% %		= C \pr(x) \neps(x) 
% %		&= \pr(x)\int \frac{\geps(x,y)}{\sqrt{\geps * \pr (x)}\sqrt{\geps * \pr(y)}} \pr(y) \ud y \\
% 		&= e^{-U_\epsilon(x)}G_\epsilon (e^{-U_\epsilon})(x)\\
% 		&= e^{-2U_\epsilon(x) + \epsilon W_\epsilon(x)}\\
% 		&= e^{-2 (U(x) + 2 \epsilon W(x) + \frac{\epsilon}{2}\Delta V(x)) + \epsilon W(x) + O(\epsilon^2) }\\
% 		&= \pr(x) e^{-3\epsilon W(x) - \epsilon \Delta V(x) + O(\epsilon^2)}\leq C\pr(x)
% 		\end{align*}
% 		where we used the approximation~\cref{eq:Ueps-U}.
%		\begin{align*}
%		\pr(x) \neps(x) &= e^{-2 (U(x) + 2 \epsilon W(x) + \frac{\epsilon}{2}\Delta V(x)) + \epsilon W(x) + O(\epsilon^2) }\\
%		&= \pr(x) e^{-3\epsilon W(x) - \epsilon \Delta V(x) + O(\epsilon^2)}
%		\end{align*}
%		Next, we bound the three terms inside the bracket in~\cref{eq:phieps-phitilde}. Two of them are bounded as follows:
%		We use the following upper-bounds to bound  bound the
%		three terms. The first term is bounded by

It remains to bound the three terms inside the bracket in~\cref{eq:phieps-phitilde}:
		\begin{align*}
		\|\Teps^n \tilde{\phi} - P_{n\epsilon} \tilde{\phi} \|_{L^2(\pr)} 
		&\leq C \epsilon \sqrt{n\epsilon}(\|\tilde{\phi}\|_{L^4(\pr)} + \|\nabla \tilde{\phi} \|_{L^4(\pr)})\\
		\|\sum_{k=0}^{n-1} \epsilon (\Teps^k - P_\epsilon^k) h\|_{L^2(\pr)}
		%&\leq   \sum_{k=0}^{n-1} C \epsilon^2 \sqrt{k \epsilon} (\|h\|_{L^4(\pr)} + \|\nabla h \|_{L^4(\pr)})\\
		&\leq  C \epsilon (n\epsilon)\sqrt{n\epsilon}
                  (\|h\|_{L^4(\pr)} + \|\nabla h \|_{L^4(\pr)})\\
                  |\hat{h}_{\pr_\epsilon} - \hat{h}_\pr| &\leq \int
                                                 |h(x)|\pr(x)|e^{-3\epsilon
                                                 W(x) - \epsilon
                                                 \Delta V(x) +
                                                 O(\epsilon^2)}-1|\ud
                                                 x \leq \epsilon C \|h\|_{L^2(\pr)},
		\end{align*}
		by using the error estimates~\cref{prop:Tepsn-convergence}-(iii). % The third term $|\hat{h}_{\pr_\epsilon} - \hat{h}_\pr|$ is bounded as follows:
		% \begin{align*}
		% |\hat{h}_{\pr_\epsilon} - \hat{h}_\pr| &\leq \int |h(x)||\pr(x) - \preps(x)|\ud x\\
		% &\leq \int |h(x)|\pr(x)|e^{-3\epsilon W(x) - \epsilon \Delta V(x) + O(\epsilon^2)}-1|\ud x%\\
		% % &\leq \int |h(x)|\pr(x) C \epsilon (3W(x) + \Delta V(x))\pr(x)\ud x\\
		% % &\leq \epsilon C \left[\int h(x)^2\pr(x)\ud x\right]^{1/2} \left[\int (3W(x)+\Delta V(x))^2\pr(x)\ud x \right]^{1/2}\\
		% \leq \epsilon C \|h\|_{L^2(\pr)}
		% \end{align*}
%		the approximation result from~\cref{prop:Tepsn-convergence}-(iii) in the second step. The second term is bounded by:
%		\begin{align*}
%		\|\sum_{k=0}^{n-1} \epsilon (\Teps^k - P_\epsilon^k) h\|_{L^2(\pr)}
%		&\leq   \sum_{k=0}^{n-1} C \epsilon^2 \sqrt{k \epsilon} (\|h\|_{L^4(\pr)} + \|\nabla h \|_{L^4(\pr)})\\
%		&\leq  C \epsilon (n\epsilon)\sqrt{n\epsilon} (\|h\|_{L^4(\pr)} + \|\nabla h \|_{L^4(\pr)})
%		\end{align*}
%		where we again used~\cref{lem:preps} and~\cref{prop:Tepsn-convergence}-(iii). For the third term $|\hat{h}_{\pr_\epsilon} - \hat{h}_\pr| \leq C \epsilon \|h\|_{L^2(\pr)}$ from \cref{lem:preps}. 
%Putting the bounds on the three terms together and using
%$n\epsilon<1$:
Therefore,
		\begin{align*}
		\|\phieps - \tilde{\phi}\|_{L^2(\preps)} \leq  \epsilon C (\|h\|_{L^4(\pr)} + \|\nabla h\|_{L^4(\pr)} + \|\tilde{\phi}\|_{L^4(\pr)} + \|\nabla \tilde{\phi} \|_{L^4(\pr)}).
		\end{align*}
		
\textbf{Step 2.} Both $\phi$ and $\tilde{\phi}$ are solutions with the
exact semigroup $P_\epsilon$.  Using the spectral representation~\eqref{eq:spectral_rep},
		\begin{align*}
		\phi = \sum_{m=1}^\infty \frac{1}{\lambda_m} \lr{h}{e_m}e_m,\quad \tilde{\phi} = \sum_{m=1}^\infty \frac{\epsilon}{1-e^{-\epsilon\lambda_m}} \lr{h}{e_m}e_m.
		\end{align*} 
		Therefore,
		\begin{align*}
		\|\tilde{\phi} -\phi\|_{L^2(\pr)}^2= \epsilon^2 \sum_{m=1}^\infty ( \frac{1}{1-e^{-\epsilon\lambda_m}} - \frac{1}{\epsilon\lambda_m})^2 |\lr{h}{e_m}|^2
		\leq \epsilon^2 \|h\|_{L^2(\pr)}^2,
		\end{align*} 
and thus $\|\tilde{\phi} -\phi\|_{L^2(\preps)}\leq C \|\tilde{\phi}
-\phi\|_{L^2(\pr)}\leq \epsilon^2 C \|h\|_{L^2(\pr)}^2$. 
		
\medskip

Combining the estimates from steps 1 and 2,
			\begin{align*}
			\|\phieps - \phi\|_{L^2(\preps)} \leq  \epsilon C (\|h\|_{L^4(\pr)} + \|\nabla h\|_{L^4(\pr)} + \|\tilde{\phi}\|_{L^4(\pr)} + \|\nabla \tilde{\phi} \|_{L^4(\pr)}).
			\end{align*}	
	\end{romannum}
\end{proof}

\section{Proof of the~\cref{prop:TepsN-convergence}}\label{apdx:TepsN-convergence}
\begin{proof}
Denote $\eta_j=(\sqrt{\frac{(\geps * \pr)(X^j)}{\frac{1}{N}\sum_{l=1}^N \geps(X^j,X^l)}}-1)$ and express:
\begin{align*}
\TepsN f(x) = \frac{\int \keps(x,y) f(y)\pr(y)\ud y + \xi^{(N)}_1 + \zeta^{(N)}_1}{\neps(x)  + \xi^{(N)}_2 + \zeta^{(N)}_2},
\end{align*}
where
\begin{align*}
\xi_1^{(N)} &= \frac{1}{N}\sum_{j=1}^N \keps(x,X^j)f(X^j) -
              \Expect[\keps(x,X^j)f(X ^j)], \;\; \zeta_1^{(N)} =
                                                                  \frac{1}{N}\sum_{j=1}^N
                                                                  \keps(x,X^j)f(X^j)\eta_j
\\
\xi_2^{(N)} &= \frac{1}{N}\sum_{j=1}^N \keps(x,X^j)-
              \Expect[\keps(x,X^j)], \quad
\zeta_2^{(N)} = \frac{1}{N}\sum_{j=1}^N \keps(x,X^j)\eta_j.
\end{align*}
%where $\eta_j=(\sqrt{\frac{(\geps * \pr)(X^j)}{\frac{1}{N}\sum_{l=1}^N \geps(X^j,X^l)}}-1)$.

\begin{romannum}
\item 
To prove the part-(i) of the~\cref{prop:TepsN-convergence}, the
strategy is to show that as $N\to\infty$ the  stochastic terms
$\xi_1^{(N)},\xi^{(N)}_2,\zeta_1^{(N)},\zeta_2^{(N)}$ converge to zero
almost surely.  We do this in two steps below,
$\xi_1^{(N)},\xi^{(N)}_2$ in step 1, and $\zeta_1^{(N)},\zeta_2^{(N)}$
in step 2.  

\noindent
{\bf Step 1:}  Convergence of $\xi_1^{(N)}$ and $\xi_1^{(N)}$ follows from direct application of the strong law of large numbers (SLLN). The SLLN applies because the summand for $\xi_1^{(N)}$ and $\xi_2^{(N)}$ are independent and identically distributed (i.i.d) and moreover have finite variance:
\begin{align}
\text{Var}\left(\keps(x,X)f(X)\right)&\leq  \frac{C}{\epsilon^{d/2}} \frac{\|f\|_\infty^2\pr(x)}{(g_\epsilon * \pr)^2(x)}\label{eq:var-keps-f}\\
\text{Var}\left(\keps(x,X)\right)&\leq  \frac{C}{\epsilon^{d/2}} \frac{\pr(x)}{(g_\epsilon * \pr)^2(x)},\label{eq:var-keps}
\end{align}
%where $C$ is a constant. 
where we used $\geps^2(x,y)\leq C\epsilon^{-d/2}g_{\epsilon/2}(x,y)$.

\noindent
{\bf Step 2:}
% Similarly $\xi^{(N)}_2$ converges to zero by SLLN because
%\begin{equation*}
%\Expect[|\xi_2^{(N)}|^2]\leq \frac{1}{N}\int \keps^2(x,y)\pr(y)\ud \leq 
%\frac{C}{N\epsilon^{d/2}} \frac{\pr(x)}{(g_\epsilon * \pr)^2(x)}
%%\frac{C}{N\epsilon^{d/2}}\frac{1}{g_\epsilon *\pr(x)}
%\end{equation*}
In order to show the almost sure convergence of $\zeta_1^{(N)}$ and $\zeta_1^{(N)}$ to zero, we first show that in the limit as $N\to \infty$,
	\begin{align}\label{eq:eta-bound}
|\eta_i| &\leq C \sqrt{\frac{\log(\frac{N}{\delta})}{N\epsilon^{d/2}\qeps(X^i)}}  ,\quad\forall i=1,\ldots,N,\quad 
%\xi_i &= \frac{C}{\sqrt{N}\epsilon^{d/4}}\frac{\keps(x,X^i)}{\sqrt{\preps(X^i)}} + o(\frac{1}{\sqrt{N}\epsilon^{d/4}})
\end{align}
with probability larger than $1-\delta$ for any arbitrary choice of $\delta \in
(0,1)$. Assuming for now that the claim is true, it then follows
%follows from convergence of, use~\cref{lem:eta-bound} to conclude
\begin{align}
\zeta_1^{(N)}& \leq 
%\frac{1}{N}\sum_{j=1}^N \keps (x,X^j) |f(X^j)|\sqrt{\frac{C\log(\frac{N}{\delta})}{N\epsilon^{d/2} g_\epsilon*\pr (X^j)}}\\
%&\leq 
\sqrt{\frac{C\log(\frac{N}{\delta})}{N\epsilon^{d/2}}}\left( \frac{1}{N}\sum_{j=1}^N \keps(x,X^j)\frac{|f(X^j)|}{\sqrt{\geps * \pr(X^j)}}\right),\label{eq:var-zeta}
\end{align}
with probability larger than $1-\delta$. The
term inside the bracket converges almost surely to its limit
$\Expect[\keps(x,X)\frac{|f(X)|}{\sqrt{\geps * \pr (X)}}]$, by SLLN, because  
%The term in the bracket is bounded with probability one as $N \to \infty$ because 
\begin{align*}
\Expect\left( \keps(x,X)\frac{|f(X)|}{\sqrt{\geps * \pr (X)}}\right)
&\leq\frac{C\|f\|_\infty \pr(x)}{(\geps * \pr)^{3/2} (x)}.
\end{align*}
The proof that $\zeta_1^{(N)}\overset{\text{a.s.}}{\longrightarrow}0$ is completed by an application of the Borel-Cantelli lemma. Indeed, choose a sequence $\{\delta_N\}_{N=1}^\infty$ given by $\delta_N = \frac{1}{N^2}$. Then $\sum_{N=1}^\infty \P(\zeta^{(N)}_1>\epsilon_N)\leq \sum_{N=1}^\infty\delta_N <\infty$  where $\epsilon_N=\sqrt{\frac{C\log(N^3)}{N\epsilon^{d/2}}}$. Because $\epsilon_N \to 0$, then $\zeta_1^{(N)}\overset{\text{a.s}}{\to}0$. The proof of $\zeta_2^{(N)}\overset{\text{a.s}}{\to}0$ is identical.

It remains to prove the claim~\cref{eq:eta-bound}, which can be established using the Bernstein inequality as follows. We have for any $a>0$: 
\begin{align*}
\P\left(\eta_i \geq a \right) &=
\P\left(\sqrt{\frac{(\geps * \pr)(X^j)}{\frac{1}{N}\sum_{l=1}^N \geps(X^j,X^l)}}\geq 1+a\right)\\&\leq
\P\left(\frac{(\geps * \pr)(X^j)}{\frac{1}{N}\sum_{l=1}^N \geps(X^j,X^l)}\geq 1+a\right)
%		\\&= \P\left(\frac{\qeps(X^i)-\qepsN(X^i)}{\qeps(X^i)}\geq \frac{a}{1+a}\right)
%	 \P\left(\frac{\qepsN(X^i)-\qeps(X^i)}{\qepsN(X^i)}\geq a\right) \\
%	&\leq \P\left(\frac{|\frac{1}{N}\sum_{j=1}^N \geps(X^i,X^j)-\qeps(X^i)|}{\frac{1}{N}\sum_{j=1}^N \geps(X^i,X^j)}\geq a\right) \\
\\&=\P\left(\frac{(\geps *\pr )(X^j)-\frac{1}{N}\sum_{l=1}^N \geps(X^j,X^l)}{(\geps * \pr )(X^j)}\geq \frac{a}{1+a}\right),
\end{align*}  The random variables $\geps(X^i,X^j)$ are i.i.d,
bounded by ${(4\pi\epsilon)^{-\frac{d}{2}}}$, and the variance
\begin{align*}
\Expect\left[|\geps(X^i,X^j)|^2|X^j\right] 
%&=  \int \geps(X^j,y)^2 \pr(y)\ud y \\
%&=\frac{1}{(8\pi\epsilon)^{d/2}}\int g_{\epsilon/2}(X^j,y) \pr(y)\ud y\\
&\leq 
\frac{1}{(8\pi\epsilon)^{d/2}}(g_{\epsilon/2} * \pr)(X^j). 
%	\\&= \frac{q_{\epsilon/2}(X^i)}{(8\pi\epsilon)^{d/2}}
\end{align*}
Therefore by Bernstein inequality,
%	\begin{equation*}
%	\P\left(\frac{\qeps(X^i)-\frac{1}{N}\sum_{j=1}^N \geps(X^i,X^j)}{\qeps(X^i)}\geq \frac{a}{1+a}\right) \leq \exp(-\frac{N(\frac{a\qeps(X^i)}{1+a})^2}{2(\frac{q_{\epsilon/2}(X^i)}{(8\pi\epsilon)^{d/2}}+\frac{a\qeps(X^i)}{1+a}\frac{1}{(4\pi\epsilon)^{\frac{d}{2}}})})
%	\end{equation*} 
%	which implies
\begin{equation*}
|\eta_i| \leq C \sqrt{\frac{(g_{\epsilon/2}* \pr)(X^j)\log(\frac{2}{\delta})}{N(8\pi\epsilon)^{d/2}(\geps * \pr)(X^j)^2}}, 
\end{equation*}
with probability higher than $1-\delta$. The result is obtained by union bound for $i=1,\ldots,N$ and $\|\frac{g_{\epsilon/2} * \pr}{g_{\epsilon} * \pr}\|_\infty<\infty$.

%and apply the Borel-Cantelli lemma to conclude the almost sure convergence $\zeta_1^{(N)}$ to zero. Similarly $\zeta_2^{(N)}$ converges to zero. 
%Therefore for all $x\in \Re^d$
%\begin{equation*}
%\lim_{N \to \infty}\TepsN f(x) = \Teps f(x),\quad \text{a.s}
%\end{equation*} 

\item Collecting the estimates~\cref{eq:var-keps-f,eq:var-keps,eq:var-zeta} and application of the Bernstein inequality yields:
\begin{align*}
|\xi_1^{(N)} |&\leq \sqrt{\frac{C\|f\|^2_{L^\infty}\log{(\frac{1}{\delta})}\pr(x)}{N\epsilon^{d/2}(\geps * \pr)^2 (x)}},\quad 
|\xi_2^{(N)}| \leq \sqrt{\frac{C\log{(\frac{1}{\delta})}\pr(x)}{N\epsilon^{d/2}(\geps * \pr)^2 (x)}}\\
|\zeta_1^{(N)}| &\leq \sqrt{\frac{C\|f\|^2_{L^\infty}\log(\frac{N}{\delta})\pr^2(x)}{N\epsilon^{d/2}(g_\epsilon * \pr)^3(x)}},\quad 
|\zeta_2^{(N)}| \leq \sqrt{\frac{C \log(\frac{N}{\delta})\pr^2(x)}{N\epsilon^{d/2}(g_\epsilon * \pr)^3(x)}},
\end{align*}
 with probability larger than $1-4\delta$. Therefore one obtains the bound:
\begin{align*}
|\TepsN f(x)-\Teps f(x)| 
%&= |\frac{\int \keps(x,y) f(y)\pr(y)\ud y + \xi^{(N)}_1 + \zeta^{(N)}_1}{\int \keps(x,y)\pr(y)\ud y  + \xi^{(N)}_2 + \zeta^{(N)}_2} - \Teps f(x)| \\
% &= |\frac{\int \keps(x,y) f(y)\pr(y)\ud y + \xi^{(N)}_1 + \zeta^{(N)}_1}{\int \keps(x,y)\pr(y)\ud y  + \xi^{(N)}_2 + \zeta^{(N)}_2} - \Teps f(x)| \\
%& = |\frac{\xi_1^{(N)}+ \zeta_1^{(N)}-\Teps f(x)(\xi_2^{(N)}+\zeta_2^{(N)})}{\int \keps(x,y)\pr(y)\ud y  + \xi_2^{(N)}+\zeta_2^{(N)}} \\
&\leq \sqrt{\frac{C\log(\frac{N}{\delta})\pr(x)}{N\epsilon^{d/2}(g_\epsilon * \pr)^2(x) \neps^2(x)}},
\end{align*}
with probability larger than $1-4\delta$.
%
%\begin{align*}
%|\TepsN f(x)-\Teps f(x)| &= |\frac{\int \keps(x,y) f(y)\pr(y)\ud y + \xi^{(N)}_1 + \zeta^{(N)}_1}{\int \keps(x,y)\pr(y)\ud y  + \xi^{(N)}_2 + \zeta^{(N)}_2} - \Teps f(x)| \\
%& = |\frac{\xi_1^{(N)}+ \zeta_1^{(N)}}{\neps(x)} - \Teps f(x) \frac{\xi_2^{(N)}+\zeta_2^{(N)}}{\neps(x)} |+ o(\sqrt{\frac{\log(N)}{N\epsilon^{d/2}}})\\
%&\leq O(\sqrt{\frac{\log(\frac{N}{\delta})}{N\epsilon^{d/2}g_\epsilon * \pr(x) \neps^2(x)}})
%\end{align*}
%with high probability. 
Upon squaring and integrating both sides with respect to $\pr(x)$ proves the rate:
\begin{align*}
\|\TepsN f - \Teps f&\|_{L^2(\pr)} \leq \sqrt{\frac{C\log(\frac{N}{\delta})}{N\epsilon^{d/2}}} \left(\int \frac{\pr(x)}{(\geps * \pr)^2 (x) \neps^2(x)}\pr(x) \ud x\right)^{1/2}\\
&\leq\sqrt{\frac{C\log(\frac{N}{\delta})}{N\epsilon^{d/2}}}
  \left(\int e^{-2\epsilon|\nabla V(x)|^2 + \frac{3}{2}\epsilon|\nabla
  V(x)|^2}\ud x\right)^{1/2} \leq \sqrt{\frac{C\log(\frac{N}{\delta})}{N\epsilon^{d}}}.  
\end{align*}
\end{romannum}

\end{proof}

\section{Proof of the~\cref{thm:variance}} \label{apdx:variance} 
In the proof of~\cref{thm:variance}, the function space of interest is
$C_b(\Omega)$, the Banach space of continuous bounded functions on (a compact
set) $\Omega\subset \Re^d$
equipped with the $\|\cdot\|_{L^\infty(\Omega)}$ norm. Also, define the space $C_0(\Omega):=\{f
\in C(\Omega) \mid \int f \rho_\epsilon  =0\}$, as subspace of functions in $C_b(\Omega)$ with zero mean.  Consider $\Teps$ and
$\TepsN$ as linear operators from $C_b(\Omega)$ to $C_b(\Omega)$.

Part-(i) has already been proved as part of the \cref{prop:N}.  The
proof of part~(ii) relies on the verification of the following three
conditions:
	\begin{romannum}
		\item The family of operators
                  $\{\TepsN\}_{N=1}^\infty$ is collectively compact,
                  as linear operators on $C_b(\Omega)$.
		\item For any function $f \in C_b(\Omega)$,
		\begin{equation}\label{eq:TepsN-unif-convergence}
		\lim_{N\to \infty} \|\TepsN f - \Teps f\|_{L^\infty(\Omega)}  = 0,\quad \text{a.s.}
		\end{equation}
		\item The operator $(I-\Teps)^{-1}$ is a bounded
                  operator on $C_0(\Omega)$.
	\end{romannum}

Once these three conditions have been verified, the
convergence result~\cref{eq:phiepsN_phieps_conv_result} follows from a
standard result in the approximation theory of the numerical solutions of integral
equations~\cite[Thm. 7.6.6]{hutson2005}.

\medskip

The proof of the three conditions is as follows:
	\begin{romannum}
		\item Condition (i) holds if the set
                  $S=\{\TepsN f;~\forall f \in
                  C_b(\Omega),\|f\|_\infty\leq 1,N\in \mathbb{N}\}$ is
                  relatively compact. Relative compactness follows
                  from an application of the Arzela-Ascoli theorem.
                  In order to apply Arzela-Ascoli theorem, we need to
                  show that $S$ is uniformly bounded and
                  equicontinuous. The two conditions hold because 
		\begin{align}
		&\text{(unif. boundedness)}\quad|\TepsN f(x)|\leq \|f\|_\infty  \frac{\sum_{i=1}^N \kepsN(x,X^i)}{\sum_{i=1}^N \kepsN(x,X^i)}\leq 1,\nonumber\\
	&\text{(equicontinuous)}\quad	|\TepsN f(x) - \TepsN f(x')| \leq \frac{L}{\epsilon}|x-x'|e^{\frac{L}{2\epsilon}|x-x'|},    \label{eq:equicon_inequality}
		\end{align}
		for all $x,x'\in \Omega$ and $f$ such that $\|f\|_{L^\infty}\leq 1$. The detailed calculation to obtain the second inequality appears at the end of the proof.

		\item Fix a function $f\in C_b(\Omega)$. From~\cref{prop:TepsN-convergence}-(i), we know that $\TepsN f(x)$ converges to $\Teps f(x)$ almost surely pointwise for all $x\in \Omega$. Because $\Omega$ is compact and $\{\TepsN f\}$ is equicontinuous, pointwise convergence implies  uniform convergence~\cref{eq:TepsN-unif-convergence}.

		\item From parts~(i) and (ii) above, it can be
                  concluded that $\Teps$ is a compact operator.  
		Therefore, using the Fredholm alternative theorem, in
                order to show $(I-\Teps)^{-1}$ is bounded, it is
                enough to show that $I-\Teps$ is injective. The
                injectivity property is shown by contradiction. Suppose there exists a function $f \in C_0(\Omega)$ such that $f-\Teps f=0$. Let $x_0 \in \Omega$ be a point that achieves the maximum of the function $f$. Such a point exists because $f$ is continuous and $\Omega$ is compact.  Evaluating $f-\Teps f=0$ at $x=x_0$ yields
		\begin{align*}
		0=f(x_0)-\Teps f(x_0) = \frac{1}{\neps(x_0)}\int \keps(x_0,y)(f(x_0)-f(y))\ud y . 
		\end{align*}		
		 Because $\keps(x_0,y)>0$ and $f(y)\leq f(x_0)$, this
                 implies $f(y)=f(x_0)$ for all $y\in
                 \Omega$. Therefore, the function $f$ is a
                 constant. But the only constant function in
                 $C_0(\Omega)$ is zero. Hence $I-\Teps$ is injective and its inverse $(I-\Teps)^{-1}$ is bounded. 
		%Hence $f \not \in C_0(\Omega)$
	\end{romannum}

It remains to prove the equicontinuity
inequality~\cref{eq:equicon_inequality} which is done next:
	\begin{align*}
		|\TepsN f(x) - \TepsN f(x')| &\leq |\frac{\sum_{i=1}^N
                                               \kepsN(x,X^i)f(X^i)}{\sum_{i=1}^N
                                               \kepsN(x,X^i)}-\frac{\sum_{i=1}^N
                                               \kepsN(x',X^i)f(X^i)}{\sum_{i=1}^N
                                               \kepsN(x',X^i)}|\\
		&\leq 2\|f\|_{L^\infty(\Omega)} \frac{\sum_{i=1}^N \kepsN(x,X^i)|1-\frac{\keps(x',X^i)}{\keps(x,X^i)}|}{\sum_{i=1}^N \keps(x,X^i)}\\
		&\leq 2\max_{i=1,\ldots,N} |1-\frac{\keps(x',X^i)}{\keps(x,X^i)}|\leq \frac{L}{\epsilon}|x-x'|e^{\frac{L}{2\epsilon}|x-x'|}, 
		\end{align*} 
		where the last inequality is obtained as follows
		\begin{align*}
		|1-\frac{\keps(x',X^i)}{\keps(x,X^i)}| =
                  |1-\frac{\geps(x',X^i)}{\geps(x,X^i)}| = | 1-
                  e^{-\frac{(x'-x)\cdot(x'+x-2X^i)}{4\epsilon}}| \leq \frac{R}{2\epsilon}|x-x'|e^{\frac{R}{2\epsilon}|x-x'|},  
		\end{align*}
		where $R=\max_{x,y\in \Omega}|x-y|$ is the diameter of
                $\Omega$. 
  \section{Proof of \cref{prop:convergence_to_constant_gain}}
\label{Sec:prop:convergence_to_constant_gain}
\begin{romannum}
\item Consider first the finite-$N$ case. In the asymptotic limit as $\epsilon\rightarrow\infty$, we have  $(2\pi\epsilon)^{d/2}\geps(x,y)=1 + O(\frac{1}{\epsilon})$. Therefore,
% Express $\TepsNf(x)$ as
\begin{align*}
& \kepsN(x,y) = \frac{\geps(x,y)}{\sqrt{\frac{1}{N}\sum_{j=1}^N \geps(x,X^j)}\sqrt{\frac{1}{N}\sum_{j=1}^N \geps(y,X^j)}} = 1 + O(\frac{1}{\epsilon})\\
&  \nepsN(x) =  \frac{1}{N}\sum_{i=1}^N \kepsN(x,X^i) = 1 + O(\frac{1}{\epsilon}),
\end{align*}
and
\begin{equation*}
\TepsN f(x) = \frac{\frac{1}{N}\sum_{j=1}^N\keps(x,X^j)f(X^j)}{\nepsN(x)}=\frac{1}{N}\sum_{j=1}^N f(X^j) + O(\frac{1}{\epsilon}).
\end{equation*}

It is also easy to see, e.g., by using a Neumann series solution, that
in the asymptotic limit as $\epsilon\rightarrow\infty$, the solution
of the fixed-point equation~\cref{eq:finite-N-fixed-pt}~is given by 
\[
\phivec = \epsilon (\hvec-\frac{1}{N}\sum_{l=1}^N \hvec_l) + O(1).
\]
Therefore,
\begin{align*}
r &=  \phivec + \epsilon\hvec= 2 \epsilon \hvec - \epsilon(\frac{1}{N}\sum_{l=1}^N \hvec_l)+  O(1)\\
s_{ij} &=   \frac{1}{2\epsilon}\Ten_{ij}(r_j-\sum_{k=1}^N
\Ten_{ik}r_k) = \frac{1}{N} (\hvec_j -\frac{1}{N}\sum_{l=1}^N \hvec_l
         ) + O(\frac{1}{\epsilon}),
\end{align*}
and using the gain approximation formula~\cref{eq:gain-linear-form},
\[
{\sf K}_i =\sum_{j=1}^N s_{ij}X^j = \frac{1}{N}\sum_{j=1}^N (\hvec_j -\frac{1}{N}\sum_{l=1}^N \hvec_l)X^j + O(\frac{1}{\epsilon}).
\]

\item The calculations for the kernel formula are entirely analogous.
  In the asymptotic limit as $\epsilon\rightarrow\infty$,
\begin{align*}
\Teps f(x) & = \int f(x) \pr(x)\ud x + O(\frac{1}{\epsilon})\\
\phi_\epsilon(x) & = \epsilon (h(x)-\hat{h}_\pr) + O(1),
\end{align*}
and, using $\theta(x) = x$ to denote the coordinate function and
$\cdot$ to denote function multiplication, the gain
approximation formula \cref{eq:Keps-def} evaluates to
\begin{equation*}
\begin{aligned}
\K_\epsilon(x)&=\frac{1}{2\epsilon}\left[\Teps (\theta \cdot \phieps +
  \epsilon(h-\hat{h}_\pr)) - \Teps(\theta)\, \Teps(\phieps + \epsilon(h-\hat{h}_\pr))\right]  \\
&=\frac{1}{2}T_\epsilon (\theta \cdot \frac{\phieps}{\epsilon} +
h-\hat{h}_\pr) - \frac{1}{2} \Teps(\theta)\, \Teps(\frac{\phieps}{\epsilon} + h-\hat{h}_\pr) +O(\frac{1}{\epsilon})\\
&=T_\epsilon (\theta \cdot h-\hat{h}_\pr) - \Teps(\theta)\,\Teps( h-\hat{h}_\pr)+O(\frac{1}{\epsilon})\\
%&= \int \theta(x)(h-\hat{h}_\pr)}) \\
&=
\int x(h(x)-\hat{h}_\pr) \rho (x)\ud x + O(\frac{1}{\epsilon}).
\end{aligned}
\end{equation*} 

\end{romannum}

\end{document}